\documentclass[10pt]{amsart}
\usepackage[margin=2.5cm]{geometry}
\usepackage{tikz-cd}
\usepackage{graphicx}					
\usepackage{amssymb}
\usepackage{amsmath}
\usepackage{subfig, graphicx}
\usepackage{mathrsfs}
\usepackage{amssymb, amsthm, indentfirst}
\usepackage{relsize}
\usepackage{hyperref}
\usepackage{stmaryrd}
\usepackage{lineno}
\usepackage{amsmath,amsfonts,amsthm,bm}

\numberwithin{equation}{section}
\usepackage{color}
\theoremstyle{plain}
\newtheorem{thm}{Theorem}[section]

\newtheorem{conj}[thm]{Conjecture}
\newtheorem{lemma}[thm]{Lemma}
\newtheorem{prop}[thm]{Proposition}
\newtheorem{corollary}[thm]{Corollary}
\theoremstyle{definition}

\newtheorem{rmk}[thm]{Remark}
\def\Gal{\operatorname{Gal}}

\newtheorem*{hypothesis*}{Hypothesis}

\author{SHIH-YU CHEN}
\title{Algebraicity of the central critical values of twisted triple product $L$-functions}
\address{Institute of Mathematics~\\Academia Sinica~\\ 6F, Astronomy-Mathematics Building, No.\,1, Sec.\,4, Roosevelt Road, Taipei 10617, Taiwan, ROC}
\email{sychen0626@gate.sinica.edu.tw}

\def\GL{{\rm{GL}}}

\def\PGL{{\rm PGL}}
\def\o{\frak{o}}

\def\A{{\mathbb A}}
\def\C{{\mathbb C}}
\def\E{{\mathbb E}}
\def\F{{\mathbb F}}
\def\K{{\mathbb K}}
\def\R{{\mathbb R}}
\def\Q{{\mathbb Q}}
\def\Z{{\mathbb Z}}

\def\<{\langle}
\def\>{\rangle}
\def\G{{G}}

\def\bp{\begin{pmatrix}}
\def\ep{\end{pmatrix}}

\def\<{\langle}
\def\>{\rangle}

\def\GL{\operatorname{GL}}

\def\1{\mathbf{1}}

\def\itPi{\mathit{\Pi}}
\def\itPhi{\mathit{\Phi}}

\begin{document}

\begin{abstract}
We study the algebraicity of the central critical values of twisted triple product $L$-functions associated to motivic Hilbert cusp forms over a totally real \'etale cubic algebra in the totally unbalanced case. 
The algebraicity is expressed in terms of the cohomological period constructed via the theory of coherent cohomology on quaternionic Shimura varieties developed by Harris \cite{Harris1990}.
As an application, we generalize our previous result \cite{CC2017} on Deligne's conjecture for certain automorphic $L$-functions for $\GL_3 \times \GL_2$.
We also establish a relation for the cohomological periods under twisting by algebraic Hecke characters.
\end{abstract}

\maketitle
\section{Introduction}

Let $\E$ be a totally real cubic extension over $\Q$ and $\Sigma_\E=\{\infty_1, \infty_2, \infty_3\}$ be the set of embeddings of $\E$ into $\R$.
Let $\itPi = \bigotimes_v \pi_v$ be an irreducible unitary cuspidal automorphic representation of $({\rm R}_{\E/\Q}\GL_{2/\E})(\A_\Q) = \GL_2(\A_\E)$ with central character $\omega_\itPi$,
where $v$ runs through the places of $\Q$ and ${\rm R}_{\E/\Q}$ denotes Weil's restriction of scalars.
We assume $\itPi$ is motivic of weight $(\kappa_1,\kappa_2,\kappa_3) \in \Z_{\geq 2}^3$, that is, $\pi_\infty$ is a discrete series representation of weight $(\kappa_1,\kappa_2,\kappa_3)$ with $\kappa_i$ corresponding to $\infty_i$ and $\kappa_i 
\equiv \kappa_j \,({\rm mod}\,2)$ for $1 \leq i ,j \leq 3$.
We assume further that $\kappa_1+\kappa_2+\kappa_3 \equiv 0 \,({\rm mod}\,2)$ and $\kappa_1 \geq \kappa_2 \geq \kappa_3$. Consider the (finite) Asai cube $L$-function
\[
L^{(\infty)}(s,\itPi,{\rm As}) = \prod_p L(s,\pi_p,{\rm As})
\]
of $\itPi$. 
A critical point for $L(s,\itPi,{\rm As})$ is a half-integer $m+\tfrac{1}{2}$ which is neither a pole nor a zero of the archimedean local factors $L(s,\pi_\infty,{\rm As})$ and $L(1-s,\pi_\infty,{\rm As})$.
We have the following conjecture on the algebraicity of the critical values of $L(s,\itPi,{\rm As})$ in terms of Shimura's $Q$-invariant \cite[\S\,7]{Shimura1988} which we shall now recall.
Let $I\subseteq \Sigma_\E$. Suppose there exists a quaternion algebra $B$ over $\E$ such that
\begin{itemize}
\item $B$ is unramified at places in $I$ and ramified at places in $\Sigma_\E \smallsetminus I$;
\item there exists an irreducible cuspidal automorphic representation $\itPi^B$ of $B^{\times}(\A_\E)$ associated with $\itPi$ by the Jacquet--Langlands correspondence.
\end{itemize}
Then the period invariant $Q(\itPi,I) \in \C^\times / \overline{\Q}^\times$ is define to be the class represented by the Petersson norm of a non-zero $\overline{\Q}$-rational vector-valued cusp form on $B^\times(\A_\E)$ associated to $\itPi^B$ (we refer to \cite[\S\,2]{Shimura1981} for the notion of $\overline{\Q}$-rational automorphic forms). 
Note that $Q(\itPi,I)$ dose not depend on the choice of $B$ and $\overline{\Q}$-rational automorphic forms (cf.\,\cite[Theorem 6.6]{Yoshida1994}) and is always defined when ${}^\sharp I=1,3$.

\begin{conj}\label{C:Deligne Asai}
Let $m+\tfrac{1}{2}$ be critical for $L(s,\itPi,{\rm As})$.
We have
\[
\frac{L^{(\infty)}(m+\tfrac{1}{2},\itPi,{\rm As})}{(2\pi\sqrt{-1})^{4m}\cdot q(\itPi,{\rm As})} \in \overline{\Q},
\]
where
\[
q(\itPi,{\rm As}) = \begin{cases}
\pi^{\kappa_1+\kappa_2+\kappa_3+2}Q(\itPi,\Sigma_\E) & \mbox{ if $\kappa_1<\kappa_2+\kappa_3$},\\
\pi^{2\kappa_1+2}Q(\itPi,\{\infty_1\})^2 & \mbox{ if $\kappa_1 \geq \kappa_2+\kappa_3$}.
\end{cases}
\]
\end{conj}

Similar conjecture was proposed by Blasius \cite{Blasius1987} if we replace $\E$ by $\Q \times \Q \times \Q$.  One can also propose a conjecture if we replace $\E$ by $\K \times \Q$ for some real quadratic extension $\K$ over $\Q$.
We remark that the conjecture is compatible with Deligne's conjecture \cite{Deligne1979} by Yoshida's calculation of the motivic Asai periods in \cite[(5.11)]{Yoshida1994} (see also Remark \ref{R:motivic period}).
When $\kappa_1 < \kappa_2+\kappa_3$, the conjecture was proved by Garrett and Harris in \cite[Theorem 4.6]{GH1993} for $|m|>1$ and by the author and Cheng in \cite[Corollary 6.4]{CC2017} for $m=0$.
When $\kappa_1 \geq \kappa_2 + \kappa_3$, we have the following result which is a special case of our main result Theorem \ref{T: main thm} (see Remark \ref{R:Q-invariant}).

\begin{thm}\label{T:introduction}
Assume $\kappa_1 \geq \kappa_2 + \kappa_3$, $\omega_\itPi \vert_{\A_\Q^\times}$ is trivial, and ${\rm Hom}_{\GL_2(\Q_p)}(\pi_p,\C) \neq 0$ for all rational primes $p$. Then Conjecture \ref{C:Deligne Asai} holds for $m=0$.
\end{thm}

\begin{rmk}
By the results of Prasad \cite{Prasad1990} and \cite{Prasad1992}, ${\rm Hom}_{\GL_2(\Q_p)}(\pi_p,\C) \neq 0$ whenever $\pi_p$ is a principal series representation.
\end{rmk}

\subsection{Main results}
Let $\E$ be a totally real \'etale cubic algebra over a totally real number field $\F$. Let $\itPi = \bigotimes_v \pi_v$ be an irreducible cuspidal automorphic representation of $({\rm R}_{\E/\F}\GL_{2/\E})(\A_\F) = \GL_2(\A_\E)$ with central character $\omega_\itPi$, where $v$ runs through the places of $\F$.
We have the Asai cube representation
\[
{\rm As} : {}^L({\rm R}_{\E/\F}\GL_{2/\E}) \longrightarrow \GL(\C^2\otimes\C^2\otimes\C^2)
\]
of the $L$-group ${}^L({\rm R}_{\E/\F}\GL_{2/\E})$ of ${\rm R}_{\E/\F}\GL_{2/\E}$.
The associated automorphic $L$-function
\[
L(s,\itPi,{\rm As}) = \prod_v L(s,\pi_v,{\rm As})
\]
is called the twisted triple product $L$-function of $\itPi$.
We denote by $L^{(\infty)}(s,\itPi,{\rm As})$ the $L$-function obtained by excluding the archimedean $L$-factors.
Let $\Sigma_\E$ (resp.\,$\Sigma_\F$) be the set of non-zero algebra homomorphisms from $\E$ (resp.\,$\F$) into $\R$. 
We assume that $\omega_\itPi \vert_{\A_\F^\times}$ is trivial and $\itPi$ is motivic (cf.\,\S\,\ref{SS:2.2}) of weight 
\[
\underline{\kappa} = \sum_{w \in \Sigma_\E}\kappa_w w \in \Z_{\geq 1}[\Sigma_\E].
\]
We say $\itPi$ is totally unbalanced (resp.\,totally balanced) if for all $v \in \Sigma_\F$ we have
\[
2\max_{w \mid v}\{\kappa_w\} - \sum_{w \mid v}\kappa_w \geq 0 \quad ({\rm resp.}\,<0).
\]
In the totally balanced case, the algebraicity of $L(s,\itPi,{\rm As})$ at the critical points (except for $s=-\tfrac{1}{2},\tfrac{3}{2}$) were proved by Garrett--Harris \cite{GH1993} and C.--Cheng \cite{CC2017} in terms of the Petersson norm of the normalized newform of $\itPi$ and the result is compatible with Deligne's conjecture \cite{Deligne1979}.
The aim of this paper is to prove, in the totally unbalanced case, the algebraicity of the central critical value 
$L(\tfrac{1}{2},\itPi,{\rm As})$ in terms of Harris' cohomological period which we shall now describe.
Our result is compatible with Deligne's conjecture which predict that the algebraicity can be expressed in terms of the (conjectural) motivic periods in \cite{Yoshida1994} (cf.\,Remark \ref{R:motivic period}).
Suppose the global root number $\varepsilon(\itPi,{\rm As})$ of $\itPi$ with respect to the Asai cube representation is equal to $1$.
By the results of Prasad \cite{Prasad1990} and \cite{Prasad1992} and Loke \cite{Loke2001}, there exists a unique quaternion algebra $D$ over $\F$ so that there exists an irreducible cuspidal automorphic representation $\itPi^D = \bigotimes_v \pi_v^D$ of $D^\times(\A_\E)$ associated to $\itPi$ by the Jacquet--Langlands correspondence and such that
\[
{\rm Hom}_{D^\times(\F_v)}(\pi_v^D,\C) \neq 0
\]
for all places $v$ of $\F$.
Note that $D$ is totally indefinite if and only if $\itPi$ is totally unbalanced. 
In this case, for each subset $I$ of $\Sigma_\E$, we denote by $\Omega^I(\itPi^D)\in \C^\times$ (resp.\,$\Omega^I((\itPi^D)^\vee)\in \C^\times$) the period obtained by comparing the rational structures on $\pi_f^D = \bigotimes_{v \nmid \infty} \pi_v^D$ (resp.\,$(\pi_f^D)^\vee = \bigotimes_{v \nmid \infty} (\pi_v^D)^\vee$) via the zeroth and ${}^\sharp{} I$-th coherent cohomology of certain automorphic line bundles on toroidal compactification of the quaternionic Shimura variety associated to $D^\times$ (see \S\,\ref{S:2} for the precise definition).

For $\sigma \in {\rm Aut}(\C)$, there exists a unique motivic irreducible cuspidal automorphic representation ${}^\sigma\!\itPi$ of $\GL_2(\A_\E)$ such that its finite part is isomorphic to the $\sigma$-conjugate of the finite part of $\itPi$. We denote by ${}^\sigma\!\underline{\kappa}$ the weight of ${}^\sigma\!\itPi$.
The rationality field $\Q(\mathit{\Pi})$ of $\mathit{\Pi}$ is define to be the fixed field of
$\left\{\sigma \in {\rm Aut}(\C) \, \vert \, {}^\sigma\!\itPi = \itPi \right\}$ and is a number field.
For each $v \in \Sigma_\F$, let $ \tilde{v}(\underline{\kappa})\in \Sigma_\E$ be the homomorphism such that $\max_{w \mid v}\{\kappa_w\} = \kappa_{\tilde{v}(\underline{\kappa})}$. Put
\[
I_{\underline{\kappa}} = \{ \tilde{v}(\underline{\kappa}) \, \vert \, v \in \Sigma_\F\} \subset \Sigma_\E.
\]
Following is our main result for totally unbalanced $\itPi$. When $\E = \F \times \F \times \F$ and $D$ is the matrix algebra, the theorem was proposed and proved by Harris \cite{Harris1989I}. Following the ideas in \cite{Harris1989I}, we generalize the result of Harris to arbitrary totally real \'etale cubic algebra $\E$ over $\F$.

\begin{thm}\label{T: main thm} 
Let $\itPi$ be a motivic irreducible cuspidal automorphic representation of $\GL_2(\A_\E)$.
Assume that $\omega_{\itPi} \vert_{\A_\F^{\times}}$ is trivial and $\itPi$ is totally unbalanced.
\begin{itemize}
\item[(1)] If $\varepsilon(\itPi,{\rm As})=-1$, then 
\[
L(\tfrac{1}{2},{}^\sigma\!\itPi, {\rm As})=0
\]
for all $\sigma \in {\rm Aut}(\C)$.
\item[(2)] Suppose $\varepsilon(\itPi,{\rm As})=1$.
For $\sigma \in {\rm Aut}(\C)$, we have
\begin{align*}
&\sigma\left(\frac{L^{(\infty)}(\tfrac{1}{2},\itPi, {\rm As}) }{D_\E^{1/2}(2\pi\sqrt{-1})^{2[\F:\Q]}\cdot\Omega^{I_{\underline{\kappa}}}(\itPi^D)\cdot\Omega^{I_{\underline{\kappa}}}((\itPi^D)^\vee)}\right)\\
&= \frac{L^{(\infty)}(\tfrac{1}{2},{}^\sigma\!\itPi, {\rm As}) }{D_\E^{1/2}(2\pi\sqrt{-1})^{2[\F:\Q]}\cdot\Omega^{I_{{}^\sigma\!\underline{\kappa}}}({}^\sigma\!\itPi^D)\cdot\Omega^{I_{{}^\sigma\!\underline{\kappa}}}({}^\sigma\!(\itPi^D)^\vee)}.
\end{align*}
Here $D_\E$ is the absolute discriminant of $\E/\Q$ and $D$ is the unique quaternion algebra over $\F$ such that ${\rm Hom}_{D^\times(\F_v)}(\pi_v^D,\C) \neq 0$ for all places $v$ of $\F$.
In particular, we have
\[
\frac{L^{(\infty)}(\tfrac{1}{2},\itPi, {\rm As}) }{D_\E^{1/2}(2\pi\sqrt{-1})^{2[\F:\Q]}\cdot\Omega^{I_{\underline{\kappa}}}(\itPi^D)\cdot\Omega^{I_{\underline{\kappa}}}((\itPi^D)^\vee)} \in \Q(\itPi).
\]
Moreover, when $D$ is the matrix algebra, we can replace $\Omega^{I_{\underline{\kappa}}}((\itPi^D)^\vee)$ by $\Omega^{I_{\underline{\kappa}}}(\itPi)$.
\end{itemize}
\end{thm}

\begin{rmk}\label{R:Q-invariant}
Assume $\underline{\kappa} \in \Z_{\geq 2}[\Sigma_\E]$. Let $I \subseteq \Sigma_\E$. By the result of Harris \cite[Theorem 1]{Harris1994II} and the period relation in Lemma \ref{L:decomposition} below, suppose that Shimura's period invariant $Q(\itPi,I) \in \C^\times / \overline{\Q}^\times$ is defined, then we have
\[
Q(\itPi,I) = (2\pi\sqrt{-1})^{-\sum_{w \in I}\kappa_w}\cdot\Omega^I(\itPi) \,({\rm mod}\,\overline{\Q}^\times).
\]
Therefore, if $D$ is the matrix algebra and $Q(\itPi,I_{\underline{\kappa}})$ is defined, then we can express the algebraicity of $L^{(\infty)}(\tfrac{1}{2},\itPi,{\rm As})$ in terms of $Q(\itPi,I_{\underline{\kappa}})$. In particular, Theorem \ref{T:introduction} holds.
\end{rmk}

\begin{rmk}
Let $\sigma \in {\rm Aut}(\C)$ and $I$ be a subset of $\Sigma_\E$ which is admissible with respect to $\underline{\kappa}$ (see \S\,\ref{SS:notation} for the notion of admissibility). It was conjectured in \cite[Conjecture 7.1.6]{Harris1990b} that 
\[
\sigma\left( \frac{\Omega^I(\itPi^D)}{\Omega^I(\itPi)}\right) = \frac{\Omega^{{}^\sigma\!I}({}^\sigma\!\itPi^D)}{\Omega^{{}^\sigma\!I}({}^\sigma\!\itPi)}.
\]
Similarly for $\itPi^\vee$. This conjecture holds trivially for $I=\emptyset$ and is known for $I=\Sigma_\E$ (cf.\,Corollary \ref{C:Petersson norm}).
By the period relation in Theorem \ref{T: main thm 2} below, we also expect that 
\[
\sigma\left( \frac{\Omega^I(\itPi^\vee)}{\Omega^I(\itPi)}\right) = \frac{\Omega^{{}^\sigma\!I}({}^\sigma\!\itPi^\vee)}{\Omega^{{}^\sigma\!I}({}^\sigma\!\itPi)}.
\]
\end{rmk}

\begin{rmk}\label{R:motivic period}
We remark that the theorem is compatible with Deligne's conjecture. Assume $\E$ is a field. The other cases can be verified in a similar way.
Suppose $\kappa_w \geq 2$ for all $w \in \Sigma_\E$. Let $M(\itPi)$ be the (conjectural) motive attached to $\itPi$ of rank $2$ over $\E$ with coefficients in $\Q(\itPi)$ satisfying conditions in \cite[Conjecture 4.1]{Yoshida1994} with $k_0$ therein replaced by $0$.
Let $\Sigma_{\Q(\itPi)}$ be the set of embeddings from $\Q(\itPi)$ to $\C$ and identify $\Q(\itPi) \otimes_\Q \C$ with $\C^{\Sigma_{\Q(\itPi)}}$ in a natural way. We also identify $\Q(\itPi)$ as a subfield of $\Q(\itPi) \otimes_\Q \C$ by the diagonal embedding.
For $w \in \Sigma_\E$, let
\[
c^{\pm}_w(M(\itPi)) = (c^{\pm}_w(\sigma,M(\itPi)))_\sigma \in (\Q(\itPi) \otimes_\Q \C)^\times,\quad \delta_w({\rm Art}_{\omega_\itPi^{-1}}) = (\delta_w(\sigma,{\rm Art}_{\omega_\itPi^{-1}}))_\sigma \in(\Q(\itPi) \otimes_\Q \C)^\times
\]
be the covariantly defined $w$-periods in \cite[\S\,2.4]{Yoshida1994}. Here ${\rm Art}_{\omega_\itPi^{-1}}$ is the Artin motive attached to $\omega_\itPi^{-1}$ defined as in \cite[\S\,6]{Deligne1979} and $\sigma \in \Sigma_{\Q(\itPi)}$.
Comparing \cite[(4.14)]{Yoshida1994} with \cite[Theorem 3.5.1]{Harris1989I} (see also Theorem \ref{T:RS} below) on the algebraicity of Rankin--Selberg $L$-functions for $\GL_2 \times \GL_2$, it is natural to expect that
\[
(\Omega^{I_{{}^\sigma\!\underline{\kappa}}}({}^\sigma\!\itPi))_\sigma \equiv (2\pi\sqrt{-1})^{-[\F:\Q]}\left(\prod_{w \in I_{{}^\sigma\!\underline{\kappa}}}c_w^+(\sigma,M(\itPi))c_w^-(\sigma,M(\itPi))\delta_w(\sigma,{\rm Art}_{\omega_\itPi^{-1}})^{-1} \right)_\sigma \,({\rm mod}\,\Q(\itPi)^\times).
\]
On the other hand, let ${\rm As}(M(\itPi))$ be the (conjectural) Asai motive associated to $M(\itPi)$, that is, ${\rm As}(M(\itPi)) = \bigotimes_\Omega {\rm Res}_{\E/\F}\,M(\itPi)$ in the notation of \cite[Conjecture 1.8]{Yoshida1994} with $\Omega = \Gal(\overline{\Q} / \F) / \Gal(\overline{\Q} / \E)$. Then we have $L({\rm R}_{\F/\Q}\,{\rm As}(M(\itPi)),s) = (L^{(\infty)}(s+\tfrac{3}{2},{}^\sigma\!\itPi,{\rm As}))_\sigma$ and Deligne's conjecture predicts that
\[
\frac{L({\rm R}_{\F/\Q}\,{\rm As}(M(\itPi)),m)}{(2\pi\sqrt{-1})^{4[\F:\Q]m}\cdot c^{(-1)^m}({\rm R}_{\F/\Q}\,{\rm As}(M(\itPi)))} \in \Q(\itPi)
\]
for all critical points $m \in \Z$ for ${\rm R}_{\F/\Q}\,{\rm As}(M(\itPi))$. Here $c^\pm({\rm R}_{\F/\Q}\,{\rm As}(M(\itPi)))$ are Deligne's periods attached to ${\rm R}_{\F/\Q}\,{\rm As}(M(\itPi))$.
Now we explicate Yoshida's calculation \cite[(5.9)]{Yoshida1994} in our case, we then deduce from the totally unbalanced condition that
\begin{align*}
&c^{\pm}({\rm R}_{\F/\Q}\,{\rm As}(M(\itPi))) \\
&\in (2\pi\sqrt{-1})^{4[\F:\Q]}\left(G({}^\sigma\!\omega_\itPi)^2\prod_{w \in I_{\underline{\kappa}}}c_w^+(\sigma,M(\itPi))^2c_w^-(\sigma,M(\itPi))^2\delta_w(\sigma,{\rm Art}_{\omega_\itPi^{-1}})^{-2} \right)_\sigma\cdot (\widetilde{\E}^\times)^{\Sigma_{\Q(\itPi)}},
\end{align*}
where $G({}^\sigma\!\omega_\itPi)$ is the Gauss sum of ${}^\sigma\!\omega_\itPi$ and $\widetilde{\E}$ is the Galois closure of $\E$ over $\Q$. Note that the assumption $\omega_\itPi \vert_{\A_\F^\times}$ is trivial and (\ref{E:Galois Gauss sum}) imply that $(G({}^\sigma\!\omega_\itPi))_\sigma \in \Q(\itPi)^\times$. 
Therefore, our main result is compatible with Deligne's conjecture for $m=-1$, at least modulo $\widetilde{\E}^\times$.
\end{rmk}

As an application to Theorem \ref{T: main thm}, we generalize our previous result \cite{CC2017}, which is compatible with Deligne's conjecture, on the algebraicity of the central critical value of certain automorphic $L$-functions for $\GL_3\times\GL_2$. 
Let $\itPi = \bigotimes_v\pi_v$, $\itPi' = \bigotimes_v\pi'_v$ be motivic irreducible cuspidal automorphic representations of $\GL_2(\A_\F)$ with central characters $\omega_\itPi$, $\omega_{\itPi'}$ and of weights $\underline{\ell}, \,\underline{\ell}' \in \Z_{\geq 1}[\Sigma_\F]$, respectively.
For each subset $I$ of $\Sigma_\F$ which is admissible with respect to $\underline{\ell}'$, let $\Omega^I(\itPi') \in \C^\times$ be the Harris' period of $\itPi$ recalled in \S\,\ref{SS: Harris' period}.
Let ${\rm Sym}^2(\itPi)$ be the Gelbart--Jacquet lift \cite{GJ1978} of $\itPi$, which is an isobaric automorphic representation of $\GL_3(\A_\F)$ and is the functorial lift of the symmetric square representation of $\GL_2$ associated to $\itPi$.
Let 
\[
L(s,{\rm Sym}^2(\itPi)\times \itPi') = \prod_v L(s,{\rm Sym}^2(\pi_v)\times \pi'_v)
\]
the Rankin--Selberg automorphic $L$-function for $\GL_3(\A_\F)\times\GL_2(\A_\F)$ associated to ${\rm Sym}^2(\itPi)\times \itPi'$.
We denote by $L^{(\infty)}(s,{\rm Sym}^2(\itPi)\times \itPi')$ the $L$-function obtained by excluding the archimedean $L$-factors.
\begin{thm}\label{C:main}
Suppose the following conditions hold:
\begin{itemize}
\item[(i)] $\omega_\itPi^2\omega_{\itPi'}$ is trivial;
\item[(ii)] $\underline{\ell}' - 2\underline{\ell} \in \Z_{\geq 0}[\Sigma_\F]$;
\item[(iii)] there exists a totally real quadratic extension $\K$ over $\F$ such that
\[
\varepsilon(\itPi' \otimes \omega_\itPi\omega_{\K/\F}) = \varepsilon({\rm Sym}^2(\itPi)\times \itPi'),
\]
where $\varepsilon(\star)$ is the global root number of $\star$ and $\omega_{\K/\F}$ is the quadratic Hecke character of $\A_\F^\times$ associated to $\K/\F$ by class field theory.
\end{itemize}
Then we have
\begin{align*}
&\sigma\left( 
\frac{L^{(\infty)}(\tfrac{1}{2},{\rm Sym}^2(\itPi)\times \itPi')}
{D_\F^{1/2}(2\pi\sqrt{-1})^{[\F:\Q](1+r)}\cdot G(\omega_\itPi)\cdot\Omega^{\Sigma_\F}(\itPi')\cdot p((-1)^{1+r}{\rm sgn}(\omega_\itPi),\itPi')}
\right )\\
&=
\frac{L^{(\infty)}(\tfrac{1}{2},{\rm Sym}^2({}^\sigma\!\itPi)\times {}^\sigma\!\itPi')}
{D_\F^{1/2}(2\pi\sqrt{-1})^{[\F:\Q](1+r)}\cdot G({}^\sigma\!\omega_\itPi)\cdot\Omega^{\Sigma_\F}({}^\sigma\!\itPi')\cdot p((-1)^{1+r}{\rm sgn}({}^\sigma\!\omega_\itPi),{}^\sigma\!\itPi')}
\end{align*}
for all $\sigma \in {\rm Aut}(\C)$.
Here $r \in \Z$ is defined so that $|\omega_\itPi| = |\mbox{ }|_{\A_\F}^r$, $G(\omega_\itPi)$ is the Gauss sum of $\omega_\itPi$, ${\rm sgn}(\omega_\itPi) \in \{\pm 1\}^{\Sigma_\F}$ is the signature of $\omega_\itPi$, and $p(\varepsilon,\itPi')$ are the periods for $\itPi'$ defined for $\varepsilon \in \{\pm 1\}^{\Sigma_\F}$ in \cite{Shimura1978}.
\end{thm}

\begin{rmk}
Condition (iii) is satisfied when either $\pi_v'$ is not a discrete series representation for any finite place $v$ or the conductors of $\itPi$ and $\itPi'$ are square-free.
We also refer to Lemma \ref{L:Petersson norm} for the period relation between Petersson norm and $\Omega^{\Sigma_\F}(\itPi')$.
\end{rmk}

\begin{rmk}
When $\underline{\ell}' - 2 \underline{\ell} \in \Z_{<0}[\Sigma_\F]$, the algebraicity of $L^{(\infty)}(\tfrac{1}{2},{\rm Sym}^2({}^\sigma\!\itPi)\times {}^\sigma\!\itPi')$ was proved by several authors.
Suppose that $\F=\Q$ and the conductors of $\itPi$ and $\itPi'$ are square-free, it is proved in \cite[Corollary 2.6]{Ichino2005}, \cite[Theorem 1.1]{Xue2019}, \cite[Corollary 1.4]{PV-P2019}, \cite[Theorem A]{CC2017}, and \cite[Corollary 1.2]{Chen2017} in terms of Petersson norm and Shimura's period \cite{Shimura1977} and the result is compatible with Deligne's conjecture. The algebraicity is also proved in \cite{Raghuram2016} in terms of certain cohomological period.
\end{rmk}

Another result we would like to present in this paper is on the behavior of the cohomological periods under twisting by algebraic Hecke characters.
More precisely, let $\itPi$ be a motivic irreducible cuspidal automorphic representation of $\GL_2(\A_\F)$ such that its weight belongs to $\Z_{\geq 2}[\Sigma_\F]$. 
Let $M(\itPi)$ be the (conjectural) motive attached to $\itPi$ of rank $2$ over $\F$ with coefficients in $\Q(\itPi)$ satisfying conditions in \cite[Conjecture 4.1]{Yoshida1994} with $k_0$ therein replaced by $-r$, where $|\omega_\itPi| = |\mbox{ }|_{\A_\F}^r$ for some $r \in \Z$. 
Let $I$ be a subset of $\Sigma_\F$. Denote by $\Omega^I(\itPi) \in \C^\times$ the Harris' period of $\itPi$.
As mentioned in Remark \ref{R:motivic period}, we expect that 
\[
(\Omega^{{}^\sigma\!I}({}^\sigma\!\itPi))_\sigma \equiv (2\pi\sqrt{-1})^{-{}^\sharp I}(\sqrt{-1})^{r{}^\sharp I}\left(\prod_{v \in {}^\sigma\!I}c_v^+(\sigma,M(\itPi))c_v^-(\sigma,M(\itPi))\delta_v(\sigma,{\rm Art}_{\omega_\itPi^{-1}})^{-1}\right)_\sigma \,({\rm mod}\,\Q(\itPi,I)^\times).
\]
Here we enlarge the coefficients to $\Q(\itPi,I) = \Q(\itPi)\cdot \Q(I)$, ${\rm Art}_{\omega_\itPi^{-1}}$ is the Artin motive attached to $\omega_\itPi^{-1}$, and 
\[
c^{\pm}_v(M(\itPi)) = (c^{\pm}_v(\sigma,M(\itPi)))_\sigma \in (\Q(\itPi,I) \otimes_\Q \C)^\times,\quad \delta_v({\rm Art}_{\omega_\itPi^{-1}}) = (\delta_v(\sigma,{\rm Art}_{\omega_\itPi^{-1}}))_\sigma \in(\Q(\itPi,I) \otimes_\Q \C)^\times
\]
are the covariantly defined $v$-periods in \cite[\S\,2.4]{Yoshida1994} for each $v \in \Sigma_\F$. 
On the other hand, let $\chi$ be an algebraic Hecke character of $\A_\F^\times$. Enlarging the coefficients to $\Q(\itPi,I,\chi) = \Q(\itPi)\cdot\Q(I)\cdot\Q(\chi)$ and applying \cite[Proposition 3.1]{Yoshida1994} to $M \otimes N = M(\itPi)\otimes {\rm Art}_{\chi^{-1}} = M(\itPi \otimes \chi)$, we have
\[
c_v^+(M(\itPi \otimes \chi))c_v^-(M(\itPi \otimes \chi))\delta_v({\rm Art}_{\omega_\itPi^{-1}\chi^{-2}})^{-1} = c_v^+(M(\itPi))c_v^-(M(\itPi))\delta_v({\rm Art}_{\omega_\itPi^{-1}})^{-1} 
\]
for each $v \in \Sigma_\F$.
Therefore, it is natural to expect that
\[
(\Omega^{{}^\sigma\!I}({}^\sigma\!\itPi\otimes{}^\sigma\!\chi))_\sigma \equiv (\Omega^{{}^\sigma\!I}({}^\sigma\!\itPi))_\sigma \,({\rm mod}\,\Q(\itPi,I,\chi)^\times).
\]
Indeed, we have the following result.

\begin{thm}\label{T: main thm 2}
Let $\itPi$ be a motivic irreducible cuspidal automorphic representation of $\GL_2(\A_\F)$ and $\chi$ an algebraic Hecke character of $\A_\F^\times$.
Suppose the weight of $\itPi$ belongs to $\Z_{\geq 4}[\Sigma_\F]$.
For $I \subseteq \Sigma_\F$ and $\sigma \in {\rm Aut}(\C)$, we have
\[
\sigma\left(\frac{\Omega^I(\itPi)}{\Omega^I(\itPi\otimes\chi)}\right) = \frac{\Omega^{{}^\sigma\!I}({}^\sigma\!\itPi)}{\Omega^{{}^\sigma\!I}({}^\sigma\!\itPi \otimes {}^\sigma\!\chi)}.
\]
\end{thm}

\subsection{An outline of the proof}

There are two main ingredients for the proof of Theorem \ref{T: main thm}:
\begin{itemize}
\item Ichino's central value formula for $L(\tfrac{1}{2},\itPi,{\rm As})$;
\item cohomological interpretation of the global trilinear period integral.
\end{itemize}
We have the global trilinear period integral $I^D \in {\rm Hom}_\C(\itPi^D \otimes (\itPi^D)^\vee,\C)$ in (\ref{E:global period}) defined by integration on $D^\times(\A_\F) \times D^\times(\A_\F)$ of cusp forms in $\itPi^D \otimes (\itPi^D)^\vee$.
In \cite{Ichino2008}, Ichino established a formula which decomposes the global trilinear period integral $I^D$ into a product of $L(\tfrac{1}{2},\itPi,{\rm As})$ and the local trilinear period integrals $I_v^D$ defined in (\ref{E:local period integral}). The formula is a special case of the refined Gan--Gross--Prasad conjecture proposed by Ichino--Ikeda \cite{IchinoIkeda2010}.
Note that our choice of the quaternion algebra $D$ guarantees the non-vanishing of $I_v^D$ (cf.\,Lemma \ref{L:trilinear form}).
The Galois equivariant property of these local trilinear period integrals at non-archimedean places is proved in Lemma \ref{L:Galois equiv. 2}, and the calculation for the archimedean local trilinear periods integral was settled in \cite{CC2017} and is recalled in Lemma \ref{L:archimedean local period}.
On the other hand, consider the automorphic line bundles $\mathcal{L}_{(\underline{\kappa},\underline{r})}$ and $\mathcal{L}_{(\underline{\kappa}(I_{\underline{\kappa}}),\underline{r})}$ associated to the algebraic characters $\rho_{(\underline{\kappa},\underline{r})}$ and $\rho_{(\underline{\kappa}(I_{\underline{\kappa}}),\underline{r})}$ of $(\R^\times\cdot {\rm SO}(2))^{\Sigma_\E}$ in (\ref{E:algebraic rep.}) on the quaternion Shimura variety for the Shimura datum $({\rm R}_{\E/\Q}(D \otimes_\F \E)^\times, (\frak{H}^\pm)^{\Sigma_\E})$.
Let $H^0(\mathcal{L}_{(\underline{\kappa},\underline{r})}^{\rm sub})$ and $H^{[\F:\Q]}(\mathcal{L}_{(\underline{\kappa}(I_{\underline{\kappa}}),\underline{r})}^{\rm sub})$ be the zeroth and $[\F:\Q]$-th coherent cohomology groups of the subcanonical extensions $\mathcal{L}_{(\underline{\kappa},\underline{r})}^{\rm sub}$ and $\mathcal{L}_{(\underline{\kappa}(I_{\underline{\kappa}}),\underline{r})}^{\rm sub}$, respectively, on a toroidal compactification of the quaternion Shimura variety.
These cohomology groups have canonical $\Q(\underline{\kappa})$-rational structures and admit natural action of $D^\times(\A_{\E,f})$.
Inside these coherent cohomology groups, we have the cuspidal cohomology groups $H_{\rm cusp}^0(\mathcal{L}_{(\underline{\kappa},\underline{r})})$ and $H^{[\F:\Q]}_{\rm cusp}(\mathcal{L}_{(\underline{\kappa}(I_{\underline{\kappa}}),\underline{r})})$ which are admissible semisimple $D^\times(\A_{\E,f})$-submodules and consisting of cohomology classes represented by cusp forms on $D^\times(\A_\E)$.
The representation $\pi_f^D$ occurs with multiplicity one in the $\pi_f^D$-isotypic components $H_{\rm cusp}^0(\mathcal{L}_{(\underline{\kappa},\underline{r})})[\pi_f^D]$ and $H^{[\F:\Q]}_{\rm cusp}(\mathcal{L}_{(\underline{\kappa}(I_{\underline{\kappa}}),\underline{r})})[\pi_f^D]$ and these isotypic components have canonical $\Q(\itPi)$-rational structures inherit from that of $\mathcal{L}_{(\underline{\kappa},\underline{r})}^{\rm sub}$ and $\mathcal{L}_{(\underline{\kappa}(I_{\underline{\kappa}}),\underline{r})}^{\rm sub}$ (see Lemmas \ref{L:2.2} and \ref{L:coherent structure}).
Note that each class in $H_{\rm cusp}^0(\mathcal{L}_{(\underline{\kappa},\underline{r})})[\pi_f^D]$ (resp.\,$H^{[\F:\Q]}_{\rm cusp}(\mathcal{L}_{(\underline{\kappa}(I_{\underline{\kappa}}),\underline{r})})[\pi_f^D]$) is uniquely represented by a holomorphic cusp form in $\itPi^D$ (resp.\,by a cusp form in $\itPi^D$ which is anti-holomorphic at $w \in I_{\underline{\kappa}}$ and holomorphic elsewhere).
The period $\Omega^{I_{\underline{\kappa}}}(\itPi^D) \in \C^\times$ is the non-zero complex number, unique up to $\Q(\itPi)^\times$, such that
\[
\frac{\varphi^{I_{\underline{\kappa}}}}{\Omega^{I_{\underline{\kappa}}}(\itPi^D)}
\]
represents a $\Q(\itPi)$-rational cusp form in $H^{[\F:\Q]}_{\rm cusp}(\mathcal{L}_{(\underline{\kappa}(I_{\underline{\kappa}}),\underline{r})})[\pi_f^D]$ for any $\Q(\itPi)$-rational holomorphic cusp form $\varphi \in \itPi^D$. Here $\varphi^{I_{\underline{\kappa}}}$ is the right translation of $\varphi$ by $\tau_{I_{\underline{\kappa}}} \in \GL_2(\R)^{\Sigma_\E} = D^\times(\E_\infty)$ with
\[
(\tau_{I_{\underline{\kappa}}})_w = \begin{cases}
\bp -1 & 0 \\ 0 & 1\ep & \mbox{ if $w \in I_{\underline{\kappa}}$},\\
1 &\mbox{ otherwise}.
\end{cases}
\]
Under the totally unbalanced condition, we constructed in \S\,\ref{S:3} certain trilinear differential operator $[\delta(\underline{\kappa})]$ rational over $\Q(\underline{\kappa})$ from $\mathcal{L}_{(\underline{\kappa}(I_{\underline{\kappa}}),\underline{r})}$ to $\mathcal{L}'_{(\underline{2},\underline{0})}$, the automorphic line bundle associated to the algebraic character $\rho'_{(\underline{2},\underline{0})}$ of $(\R^\times\cdot {\rm SO}(2))^{\Sigma_\F}$ on the quaternion Shimura variety for the Shimura datum $({\rm R}_{\F/\Q}D^\times,(\frak{H}^\pm)^{\Sigma_\F})$. 
It induces a $D^\times(\A_{\F,f})$-module homomorphism
\[
[\delta(\underline{\kappa})] : H^{[\F:\Q]}(\mathcal{L}_{(\underline{\kappa}(I_{\underline{\kappa}}),\underline{r})}^{\rm sub}) \longrightarrow H^{[\F:\Q]}((\mathcal{L}_{(\underline{2},\underline{0})}')^{\rm sub})
\]
which is rational over $\Q(\underline{\kappa})$.
Moreover, if a class in $H^{[\F:\Q]}_{\rm cusp}(\mathcal{L}_{(\underline{\kappa}(I_{\underline{\kappa}}),\underline{r})})[\pi_f^D]$ is represented by a cusp form $\varphi$, then its image under $[\delta(\underline{\kappa})]$ is represented by $({\bf X}(\underline{\kappa})\cdot \varphi )\vert_{D^\times(\A_\F)}$ for some differential operator ${\bf X}(\underline{\kappa}) \in U(\frak{gl}_{2,\C}^{\Sigma_\E})$ defined in (\ref{E:differential operator}).
Composing the trilinear differential operator with the $\Q$-rational trace map 
\[
H^{[\F:\Q]}((\mathcal{L}_{(\underline{2},\underline{0})}')^{\rm sub})\longrightarrow \C
\]
in Lemma \ref{L:trace}, we deduce that
\[
\int_{\A_\F^\times D^\times(\F)\backslash D^\times(\A_\F)}\frac{{\bf X}(\underline{\kappa})\cdot\varphi^{I_{\underline{\kappa}}}(g)}{\Omega^{I_{\underline{\kappa}}}(\itPi^D)}\,dg^{\rm Tam} \in \Q(\itPi)
\]
for any $\Q(\itPi)$-rational holomorphic cusp form $\varphi \in \itPi^D$.
Here $dg^{\rm Tam}$ is the Tamagawa measure.
Similar assertions hold for $(\itPi^D)^\vee$ if we replace $\underline{r}$ by $-\underline{r}$ and $\pi_f^D$ by $(\pi_f^D)^\vee$.
Therefore, we have
\[
\frac{I^D({\bf X}(\underline{\kappa})\cdot\varphi_1^{I_{\underline{\kappa}}} \otimes {\bf X}(\underline{\kappa})\cdot\varphi_2^{I_{\underline{\kappa}}})}{\Omega^{I_{\underline{\kappa}}}(\itPi^D)\cdot\Omega^{I_{\underline{\kappa}}}((\itPi^D)^\vee)} \in \Q(\itPi)
\]
for $\Q(\itPi)$-rational holomorphic cusp forms $\varphi_1 \in \itPi^D$ and $\varphi_2 \in (\itPi^D)^\vee$.
Combining with Ichino's formula, we obtain the algebraicity of $L(\tfrac{1}{2},\itPi,{\rm As})$.
We mention one subtlety in the proof. In order to apply Ichino's formula, it is necessary to compare $L(1,\itPi,{\rm Ad})$, which is the special value of the adjoint $L$-function of $\itPi$ at $s=1$, with the Petersson bilinear pairing (\ref{E:Petersson pairing}) on $\itPi^D \times (\itPi^D)^\vee$. It is known that $L(1,\itPi,{\rm Ad})$ is essentially equal to the Petersson pairing of $\Q(\itPi)$-rational holomorphic cusp form in $\itPi^D \times (\itPi^D)^\vee$ (see Lemma \ref{L:Petersson norm} and Corollary \ref{C:Petersson norm}). On the other hand, the rationality of the global trilinear period integral $I^D$ is related to the rational structure of $H^{[\F:\Q]}_{\rm cusp}(\mathcal{L}_{(\underline{\kappa}(I_{\underline{\kappa}}),\underline{r})})[\pi_f^D]$ as we have explained above. This is the key reason why we need to compare the rational structures on $H_{\rm cusp}^0(\mathcal{L}_{(\underline{\kappa},\underline{r})})[\pi_f^D]$ and $H^{[\F:\Q]}_{\rm cusp}(\mathcal{L}_{(\underline{\kappa}(I_{\underline{\kappa}}),\underline{r})})[\pi_f^D]$.

Theorem \ref{T: main thm 2} is actually a direct consequence of the algebraicity of the rightmost critical value of the twisted Rankin--Selberg $L$-function $L(s,\itPi\times\itPi'\times\chi)$ for suitable motivic irreducible cuspidal automorphic representation $\itPi'$ of $\GL_2(\A_\F)$. The algebraicity is expressed in terms of the cohomological periods $\Omega^I(\itPi), \Omega^{\Sigma_\E \smallsetminus I}(\itPi')$, the Gauss sum $G(\chi^2 \omega_\itPi\omega_{\itPi'})$, and some other elementary factors. The algebraicity was proved by Shimura in \cite[Theorem 4.2]{Shimura1978} when $I=\Sigma_\F$ and by Harris in \cite[Theorem 3.5.1]{Harris1989I} for general $\itPi, \itPi'$, and $\chi = {\bf 1}$. We generalize the result to arbitrary twist by algebraic Hecke character $\chi$ in Theorem \ref{T:RS}. By applying Theorem \ref{T:RS} to the triplets $(\itPi,\itPi',\chi)$ and $(\itPi \otimes \chi,\itPi',{\bf 1})$, the period relation in Theorem \ref{T: main thm 2} follows immediately.

This paper is organized as follows. In \S\,\ref{S:2}, we recall the theory of coherent cohomology on quaternion Shimura varieties based on the general results of Harris \cite{Harris1990}. The cohomological periods $\Omega^I(\itPi)$ for admissible $I \subseteq \Sigma_\E$ are defined in Proposition \ref{P:Harris' period}. 
In \S\,\ref{S:3}, we construct the trilinear differential operator under the totally unbalanced condition. We specialize the results of Harris in \cite[\S\,3]{Harris1985} and \cite[\S\,7]{Harris1986} to the natural inclusion
\[
({\rm R}_{\F/\Q}D^\times,(\frak{H}^\pm)^{\Sigma_\F}) \subset ({\rm R}_{\F/\Q}(D\otimes_\F \E)^\times,(\frak{H}^\pm)^{\Sigma_\E})
\]
of Shimura data and the automorphic line bundles $\mathcal{L}_{(\underline{\kappa}(I_{\underline{\kappa}}),\underline{r})}$ and $\mathcal{L}_{(\underline{2},\underline{0})}'$. 
In \S\,\ref{S:4}, we prove the algebraicity of critical values of twisted Rankin--Selberg $L$-functions for $\GL_2 \times \GL_2$ over $\F$. This section is logically independent of the proof of Theorem \ref{T: main thm} except for \S\,\ref{SS:4.1}, whereas Theorem \ref{T: main thm 2} is a corollary to the main result Theorem \ref{T:RS} of \S\,\ref{S:4}.
In \S\,\ref{S:5}, we prove our main results Theorems \ref{T: main thm}, \ref{C:main}, and \ref{T: main thm 2}. 

\subsection{Notation}\label{SS:notation}

Fix a totally real number field $\F$. Let $\A_\F$ (resp.\,$\A$) be the ring of adeles of $\F$ (resp.\,$\Q$) and $\A_{\F,f}$ (resp.\,$\A_f$) be its finite part. 
Let $\o_\F$ be the ring of integers of $\F$ and $D_\F$ the absolute discriminant of $\F/\Q$.
We denote by $\hat{\o}_\F$ the closure of $\o_\F$ in $\A_{\F,f}$.
Let $\psi_\Q=\bigotimes_v\psi_{\Q_v}$ be the standard additive character of $\Q\backslash \A$ defined so that
\begin{align*}
\psi_{\Q_p}(x) & = e^{-2\pi \sqrt{-1}\,x} \mbox{ for }x \in \Z[p^{-1}],\\
\psi_{\R}(x) & = e^{2\pi \sqrt{-1}\,x} \mbox{ for }x \in \R.
\end{align*}
The standard additive character $\psi_{\F}$ of $\F \backslash \A_\F$ is defined by $\psi_{\F} = \psi_\Q \circ {\rm tr}_{\F /\Q}$. For $\alpha \in \F$, let $\psi_\F^\alpha$ be the additive character defined by $\psi_\F^\alpha(x) = \psi_\F(\alpha x)$. Similarly we define $\psi_{\F_v}^\alpha$ for $\alpha \in \F_v$.


Let $\E$ ba a totally real \'etale algebra over $\F$.
Let $\Sigma_\E$ be the set of non-zero algebra homomorphisms from $\E$ to $\R$. We identify $\E_\infty = \E \otimes_\Q \R$ with $\R^{\Sigma_\E}$ so that the $w$-th coordinate of $\R^{\Sigma_\E}$ corresponds to the completion of $\E$ at $w$. 
Let $\underline{\kappa} = \Sigma_{w \in \Sigma_\E}\kappa_ww \in \Z[\Sigma_\E]$. For $\sigma \in {\rm Aut}(\C)$, define ${}^\sigma\!\underline{\kappa} = \Sigma_{w \in \Sigma_\E}{}^\sigma\kappa_ww \in \Z[\Sigma_\E]$ with ${}^\sigma\!\kappa_w = \kappa_{\sigma^{-1}\circ w}$.
Let $\Q(\underline{\kappa})$ be the fixed field of $\{\sigma \in {\rm Aut}(\C)\,\vert\, {}^\sigma\!\underline{\kappa} = \underline{\kappa}\}$.
For each subset $I$ of $\Sigma_\E$, let $\underline{\kappa}(I) = \Sigma_{w \in \Sigma_\E}\kappa(I)_ww \in \Z[\Sigma_\E]$ defined by
\[
\kappa(I)_w = \begin{cases}
2-\kappa_w & \mbox{ if $w \in I$},\\
\kappa_w   & \mbox{ if $w \notin I$},
\end{cases}
\]
and let $\Q(I)$ be the fixed field of $\{\sigma \in {\rm Aut}(\C) \, \vert \, {}^\sigma\!I=I\}$. Note that $\Q(\underline{\kappa}(I))\subseteq \Q(\underline{\kappa})\cdot\Q(I)$.
Consider the map
\begin{align*}
\mbox{the power set of }\Sigma_\E &\longrightarrow \{0,1,\cdots,[\E:\Q]\} \times \Z[\Sigma_\E],\\
I & \longmapsto ({}^\sharp I , \underline{\kappa}(I)).
\end{align*}
We say a subset $I$ of $\Sigma_\E$ is admissible with respect to $\underline{\kappa}$ if the fiber of $({}^\sharp I , \underline{\kappa}(I))$ under the above map contains only $I$. In this case, we have $\Q(\underline{\kappa}(I)) = \Q(\underline{\kappa})\cdot\Q(I)$. It is clear that the empty set and $\Sigma_\E$ are admissible. We will use the notion of admissibility only when $\underline{\kappa} \in \Z_{\geq1}[\Sigma_\E]$. In this case, any subset $I$ such that $\kappa_w \geq 2$ for all $w \in I$ is admissible.
We assume
\[
\E = \F_1 \times \cdots \times \F_n
\]
for some totally real number fields $\F_1,\cdots,\F_n$ over $\F$.
Let $\widetilde{\E}$ be the composite of the Galois closure of $\F_i$ over $\Q$ for $1 \leq i \leq n$.
We identify $\Sigma_\E$ with the disjoint union of $\Sigma_{\F_i}$ for $1 \leq i \leq n$ in a natural way.
Let $(\underline{\kappa},\underline{r}) \in \Z[\Sigma_\E] \times \Z[\Sigma_\E]$. We say $(\underline{\kappa},\underline{r})$ is motivic if $\kappa_w \equiv r_w \,({\rm mod}\,2)$ for $w \in \Sigma_\E$ and $r_w = r_{w'}$ whenever $w,w'\in \Sigma_{\F_i}$ for some $1 \leq i \leq n$.

In $\GL_2$, let $B$ be the Borel subgroup consisting of upper triangular matrices and $N$ be its unipotent radical, and put
\[
{\bf a}(\nu) = \bp \nu & 0 \\ 0 & 1 \ep,\quad {\bf d}(\nu) =   \bp 1 & 0 \\ 0 & \nu \ep,\quad {\bf m}(t)= \bp t & 0 \\ 0 & t^{-1}\ep,\quad {\bf n}(x) = \bp 1 & x \\ 0 & 1\ep
\]
for $\nu,t \in \GL_1$ and $x \in \mathbb{G}_a$.
Let $\frak{gl}_2$ be the Lie algebra of $\GL_2(\R)$ and $\frak{gl}_{2,\C}$ be its complexification. We have
\[
\frak{gl}_{2,\C} = \C \cdot Z \oplus \C \cdot H \oplus \C \cdot X_+ \oplus \C \cdot X_-, 
\]
where
\[
Z = \bp 1 & 0 \\ 0 & 1 \ep,\quad H = \bp 0 & -\sqrt{-1} \\ \sqrt{-1} & 0 \ep ,\quad X_+ = \bp \sqrt{-1} & -1 \\ -1 & -\sqrt{-1} \ep,\quad X_- = \bp \sqrt{-1} & 1 \\ 1 & -\sqrt{-1} \ep.
\]
We denote by $U(\frak{gl}_{2,\C})$ the universal enveloping algebra of $\frak{gl}_{2,\C}$. 
Let
\[
{\rm SO}(2)  = \left \{ \left . k_{\theta} = \bp \cos \theta & \sin \theta \\ -\sin \theta & \cos \theta \ep  \,\right \vert \,\theta \in \R / 2\pi \Z\right \}.
\]
For $\kappa \in \Z_{\geq 1}$ and $\varepsilon \in \{\pm 1\}$, let $D(\kappa)^\varepsilon$ be the irreducible admissible $(\frak{gl}_2,{\rm SO}(2))$-module characterized so that there exists a non-zero ${\bf v} \in D(\kappa)^\varepsilon$ such that 
\[
Z\cdot {\bf v} =0,\quad H\cdot {\bf v} = \varepsilon \kappa \cdot {\bf v},\quad X_{-\varepsilon}\cdot {\bf v} =0.
\]
Let $D(\kappa) = D(\kappa)^+ \oplus D(\kappa)^-$, which is an irreducible admissible $(\frak{gl}_2,{\rm O}(2))$-module. 
Note that $D(\kappa)$ is the ${\rm O}(2)$-finite part of the (limit of) discrete series representation of $\GL_2(\R)$ with weight $\kappa$.
For $r \in \Z$ such that $\kappa \equiv r\,({\rm mod}\,2)$ and $\alpha \in \R^\times$, let $W_{(\kappa,r),\psi_\R^\alpha}^\pm$ be the Whittaker function for $D(\kappa) \otimes |\mbox{ }|^{r/2}$ with respect to $\psi_\R^\alpha$ of weight $\pm \kappa$ normalized so that
$W_{(\kappa,r),\psi_\R^\alpha}^\pm(1) = e^{-2\pi\alpha}$. The explicit formula is given by
\begin{align}\label{E:archimedean Whittaker}
W_{(\kappa,r),\psi_\R^\alpha}^\pm(z{\bf n}(x){\bf a}(y)k_\theta) = z^r (\pm\alpha y)^{(\kappa+r)/2}e^{2\pi\sqrt{-1}\,\alpha(x\pm\sqrt{-1}\,y)\pm \sqrt{-1}\,\kappa\theta}\cdot\mathbb{I}_{\R_{>0}}(\pm\alpha y)
\end{align}
for $x \in \R$, $y,z \in \R^\times$, and $k_\theta \in {\rm SO}(2)$.

Let $\sigma \in {\rm Aut}(\C)$. Define the $
\sigma$-linear action on $\C(X)$, which is the field of formal Laurent series in variable $X$ over $\C$, as follows:
\[
{}^\sigma\!P(X) = \sum_{n \gg -\infty}^\infty\sigma(a_n) X^n
\]
for $P(X) = \sum_{n \gg -\infty}^\infty a_n X^n \in \C(X)$. 
For a complex representation $\pi$ of a group $G$ on the space $\mathcal{V}_\pi$ of $\pi$, let ${}^\sigma\!\pi$ be the representation of $G$ defined
\begin{align*}\label{E:sigma action}
{}^\sigma\!\pi(g) = t \circ \pi(g) \circ t^{-1},
\end{align*}
where $t:\mathcal{V}_\pi \rightarrow \mathcal{V}_\pi$ is a $\sigma$-linear isomorphism. Note that the isomorphism class of ${}^\sigma\!\pi$ is independent of the choice of $t$. We call ${}^\sigma\!\pi$ the $\sigma$-conjugate of $\pi$. When $v$ is a finite place of $\F$ and $f$ is a complex-valued function on $\F_v^m$ or $(\F_v^\times)^m$ for some $m \in \Z_{\geq 1}$, we define ${}^\sigma\!f$ by ${}^\sigma\!f(x) = \sigma(f(x))$ for $x \in \F_v^m$ or $x \in (\F_v^\times)^m$.

For an algebraic Hecke character $\chi$ of $\A_\F^\times$, the Gauss sum $G(\chi)$ of $\chi$ is defined by
\[
G(\chi) = D_\F^{-1/2}\prod_{v \nmid \infty}\varepsilon(0,\chi_v,\psi_{\F_v}),
\]
where $\varepsilon(s,\chi_v,\psi_{\F_v})$ is the $\varepsilon$-factor of $\chi_v$ with respect to $\psi_{\F_v}$ defined in \cite{Tate1979}.
For $\sigma \in {\rm Aut}(\C)$, define Hecke character ${}^\sigma\!\chi$ of $\A_\F^\times$ by ${}^\sigma\!\chi(x) = \sigma(\chi(x))$.
It is easy to verify that 
\begin{align}\label{E:Galois Gauss sum}
\begin{split}
\sigma(G(\chi)) &= {}^\sigma\!\chi(u)G({}^\sigma\!\chi),\\
\sigma\left(\frac{G(\chi\chi')}{G(\chi)G(\chi')}\right) &= \frac{G({}^\sigma\!\chi{}^\sigma\!\chi')}{G({}^\sigma\!\chi)G({}^\sigma\!\chi')}
\end{split}
\end{align}
for algebraic Hecke characters $\chi,\chi'$ of $\A_\F^\times$,
where $u \in \widehat{\Z}^\times$ is the unique element such that $\sigma(\psi_{\F}(x)) = \psi_{\F}(ux)$ for $x \in \A_{\F,f}$.
Let ${\rm sgn}(\chi) \in \{\pm 1\}^{\Sigma_\F}$ be the signature of $\chi$ defined by ${\rm sgn}(\chi)_v = \chi_v(-1)$ for $v \in \Sigma_\F$.

\section{Periods of motivic quaternionic modular forms}\label{S:2}

\subsection{Coherent cohomology groups on the quaternionic Shimura variety}

Let 
\[
\E = \F_1 \times \cdots \times \F_n
\]
be a totally real \'etale algebra over $\F$, where $\F_1,\cdots,\F_n$ are totally real number fields over $\F$.
For $1 \leq i \leq n$, let $D_i$ be a totally indefinite quaternion algebra over $\F_i$. Put $D = D_1 \times \cdots \times D_n$. 
Let 
\[
G = {\rm R}_{\E/\Q}D^\times = {\rm R}_{\F_1 / \Q}D_1^\times \times \cdots {\rm R}_{\F_n / \Q}D_n^\times
\]
be a connected reductive linear algebraic group over $\Q$. We identify $G(\R)$ with $\GL_2(\R)^{\Sigma_\E}$ via the identification of $\E_\infty$ with $\R^{\Sigma_\E}$. Let $h : {\rm R}_{\C/\R}{\mathbb G}_m \rightarrow G_\R$ be the homomorphism defined by
\begin{align}\label{E:Shimura datum}
h(x+\sqrt{-1}\,y) = \left( \bp x & y \\ -y & x \ep,\cdots , \bp x & y \\ -y & x \ep \right)
\end{align}
on $\R$-points. Let $X$ be the $G(\R)$-conjugacy class containing $h$. Then $(G,X)$ is a Shimura datum and the associated Shimura variety 
\[
{\rm Sh}(G,X) = \varprojlim_{\mathcal{K}} {\rm Sh}_{\mathcal K}(G,X) = \varprojlim_{\mathcal{K}}G(\Q)\backslash X\times G(\A_f) / \mathcal{K}
\]
is called the quaternionic Shimura variety associated to $(G,X)$, where $\mathcal{K}$ runs through neat open compact subgroups of $G(\A_f)$.
It is a pro-algebraic variety over $\C$ with continuous $G(\A_f)$-action and admits canonical model over $\Q$.
Let $K_\infty$ be the stabilizer of $h$ in $G(\R)$. Note that
\[
K_\infty = Z_G(\R)\cdot {\rm SO}(2)^{\Sigma_\E}
\]
and we have isomorphisms
\begin{align}\label{E:hermitian domain}
G(\R)/K_\infty \longrightarrow X \longrightarrow (\frak{H}^{\pm})^{\Sigma_\E},\quad gK_\infty \longmapsto ghg^{-1} \longmapsto g\cdot (\sqrt{-1},\cdots,\sqrt{-1}).
\end{align}
Here $\frak{H}^{\pm} = \C \smallsetminus \R$ is the union of the upper and lower half-planes and $G(\R)$ acts on $(\frak{H}^{\pm})^{\Sigma_\E}$ by the linear fractional transformation.
Let $\frak{g}$ and $\frak{k}$ be the Lie algebras of $G(\R)$ and $K_\infty$, respectively.
The Hodge decomposition on $\frak{g}_\C$ induced by ${\rm Ad}\circ h$ is given by
\begin{align}\label{E:Hodge}
\frak{g}_\C = \frak{p}^+ \oplus \frak{k}_\C \oplus \frak{p}^-
\end{align}
with $\frak{p}^+ = \frak{g}_\C^{(-1,1)}$, $\frak{p}^- = \frak{g}_\C^{(1,-1)}$, and $\frak{k}_\C = \frak{g}_\C^{(0,0)}$. Here
\[
\frak{g}_\C^{(p,q)} = \left.\{X \in \frak{g}_\C \, \right\vert \, h(z)^{-1}Xh(z) = z^{-p}\overline{z}^{-q}X \mbox{ for $z \in \C$}\}.
\]
We identify $\frak{g}$ with $\frak{gl}_2^{\Sigma_\E}$ via the identification of $\E_\infty$ with $\R^{\Sigma_\E}$.

Let $(\underline{\kappa},\underline{r}) \in \Z[\Sigma_\E] \times \Z[\Sigma_\E]$ such that $\kappa_w \equiv r_w \,({\rm mod}\,2)$ for $w \in \Sigma_\E$. We denote by $\C_{(\underline{\kappa},\underline{r})}$ the complex field $\C$ equipped with the action $\rho_{(\underline{\kappa},\underline{r})}$ of $K_\infty$ given by
\begin{align}\label{E:algebraic rep.}
\rho_{(\underline{\kappa},\underline{r})}(\underline{a}\cdot\underline{k}_\theta) z = \prod_{w \in \Sigma_\E} a_w^{-r_w}e^{-\sqrt{-1}\,\kappa_w \theta_w}\cdot z
\end{align}
for $\underline{a} = (a_w)_{w \in \Sigma_\E} \in (\R^\times)^{\Sigma_\E}$ and $ \underline{k}_\theta = (k_{\theta_w} )_{w \in \Sigma_\E}\in {\rm SO}(2)^{\Sigma_\E}$, and $z \in \C$.
Conversely, any one-dimensional algebraic representation of $K_\infty$ over $\C$ is of this form. 
We say $\rho_{(\underline{\kappa},\underline{r})}$ is motivic if $\rho_{(\underline{\kappa},\underline{r})}\vert_{Z_G(\R)}$ is the base change of a $\Q$-rational character of $Z_G$. It is easy to see that $\rho_{(\underline{\kappa},\underline{r})}$ is motivic if and only if $(\underline{\kappa},\underline{r})$ is motivic (cf.\,\S\,\ref{SS:notation}). 
Let $\mathcal{A}_{(2)}(G(\A))$ (resp.\,$\mathcal{A}_{0}(G(\A))$) be the space of essentially square-integrable automorphic forms (resp.\,cusp forms) on $G(\A)$.
Let
\begin{align*}
C_{\rm sia}^\infty(G(\A)) &= \left.\left\{ \varphi \in C^\infty(G(\Q)\backslash G(\A))\,\right\vert\, X\cdot \varphi \mbox{ is slowly increasing for all $X \in U(\frak{g}_\C)$}\right\},\\
C_{\rm rda}^\infty(G(\A)) &= \left.\left\{\varphi \in C^\infty(G(\Q)\backslash G(\A))\,\right\vert\, X\cdot \varphi \mbox{ is rapidly decreasing for all $X \in U(\frak{g}_\C)$}\right\}.
\end{align*}
Let $\frak{P} = \frak{p}^- \oplus \frak{k}_\C$ be a parabolic subalgebra of $\frak{g}_\C$. We have the $(\frak{P},K_\infty)$-modules 
\[
\mathcal{A}_{\star}(G(\A))\otimes_\C \C_{(\underline{\kappa},\underline{r})},\quad C_{\star'}^\infty(G(\A)) \otimes_\C \C_{(\underline{\kappa},\underline{r})}
\]
for $\star \in \{ (2),0 \}$ and $\star' \in \{ {\rm sia},{\rm rda} \}$, where the action of $\frak{P}$ on $\C_{(\underline{\kappa},\underline{r})}$ factors through $\frak{k}_\C$.
Consider the complexes with respect to the Lie algebra differential operator (cf.\,\cite[Chapter I]{BW2000}):
\begin{align}\label{E:complexes}
\begin{split}
C_{(2),(\underline{\kappa},\underline{r})}^q &= \left( \mathcal{A}_{(2)}(G(\A)) \otimes_\C \bigwedge^q \frak{p}^+ \otimes_\C  \C_{(\underline{\kappa},\underline{r})}\right)^{K_\infty},\\
C_{{\rm cusp},(\underline{\kappa},\underline{r})}^q &= \left( \mathcal{A}_{0}(G(\A)) \otimes_\C \bigwedge^q \frak{p}^+ \otimes_\C  \C_{(\underline{\kappa},\underline{r})}\right)^{K_\infty},\\
C_{{\rm sia}, (\underline{\kappa},\underline{r})}^q &= \left( C_{\rm sia}^\infty(G(\A)) \otimes_\C \bigwedge^q \frak{p}^+ \otimes_\C  \C_{(\underline{\kappa},\underline{r})}\right)^{K_\infty},\\
C_{{\rm rda}, (\underline{\kappa},\underline{r})}^q &= \left( C_{\rm rda}^\infty(G(\A)) \otimes_\C \bigwedge^q \frak{p}^+ \otimes_\C  \C_{(\underline{\kappa},\underline{r})}\right)^{K_\infty}
\end{split}
\end{align}
for $q \in \Z_{\geq 0}$. 
The corresponding $q$-th cohomology groups of the above complexes are denoted respectively by 
\[
H_{(2)}^q(\mathcal{L}_{(\underline{\kappa},\underline{r})}),\quad H_{\rm cusp}^q(\mathcal{L}_{(\underline{\kappa},\underline{r})}),\quad H^q(\mathcal{L}_{(\underline{\kappa},\underline{r})}^{\rm can}),\quad H^q(\mathcal{L}_{(\underline{\kappa},\underline{r})}^{\rm sub}).
\]
Note that $G(\A_f)$ acts on the above complexes by right translation. This in turn defines $G(\A_f)$-module structures on the cohomology groups.
It is clear that $H_{\rm cusp}^q(\mathcal{L}_{(\underline{\kappa}, \underline{r})})$ and $H_{(2)}^q(\mathcal{L}_{(\underline{\kappa}, \underline{r})})$ are semisimple $G(\A_f)$-modules.
By the results of Harris \cite{Harris1990b} and \cite{Harris1990}, the relative Lie algebra cohomology groups $H^q(\mathcal{L}_{(\underline{\kappa},\underline{r})}^{\rm can})$ and $H^q(\mathcal{L}_{(\underline{\kappa},\underline{r})}^{\rm sub})$ are isomorphic to the $q$-th coherent cohomology groups of the canonical and subcanonical extension, respectively, of certain automorphic line bundle $\mathcal{L}_{(\underline{\kappa},\underline{r})}$ on  ${\rm Sh}(G,X)$ to its toroidal compactification. Thus the terminology is justified.
The natural inclusions
\[
\begin{tikzcd}
\mathcal{A}_0(G(\A)) \arrow[d, hook] \arrow[r, hook]
    &  \mathcal{A}_{(2)}(G(\A))\arrow[d, hook] \\
  C_{\rm rda}^\infty(G(\A)) \arrow[r, hook]
& C_{\rm sia}^\infty(G(\A))
\end{tikzcd}
\]
induce the following commutative diagram for $G(\A_f)$-module homomorphisms:
\begin{equation}\label{E:commute 1}
\begin{tikzcd}
H_{\rm cusp}^q(\mathcal{L}_{(\underline{\kappa},\underline{r})}) \arrow[d] \arrow[r]
    &  H_{(2)}^q(\mathcal{L}_{(\underline{\kappa},\underline{r})})\arrow[d] \\
  H^q(\mathcal{L}_{(\underline{\kappa},\underline{r})}^{\rm sub}) \arrow[r]
& H^q(\mathcal{L}_{(\underline{\kappa},\underline{r})}^{\rm can}).
\end{tikzcd}
\end{equation}
Let $H_!^q(\mathcal{L}_{(\underline{\kappa},\underline{r})})$ be the image of the homomorphism $H^q(\mathcal{L}_{(\underline{\kappa},\underline{r})}^{\rm sub}) \rightarrow H^q(\mathcal{L}_{(\underline{\kappa},\underline{r})}^{\rm can})$.
By Theorem \ref{T:Harris}-(3) below, the homomorphism $H_{\rm cusp}^q(\mathcal{L}_{(\underline{\kappa},\underline{r})}) \rightarrow H^q(\mathcal{L}_{(\underline{\kappa},\underline{r})}^{\rm sub})$ is injective. We identify $H_{\rm cusp}^q(\mathcal{L}_{(\underline{\kappa},\underline{r})})$ with a $G(\A_f)$-submodule of $H^q(\mathcal{L}_{(\underline{\kappa},\underline{r})}^{\rm sub})$ by this injection.

We recall in the following theorem some results of Harris \cite{Harris1985}, \cite{Harris1990} and Milne \cite{Milne1983} specialized to the Shimura datum $(G,X)$. Let $\mathcal{A}_{(2)}(G(\A),(\underline{\kappa},\underline{r}))$ be the space of essentially square-integrable automorphic forms $\varphi$ on $G(\A)$ such that
\begin{itemize}
\item $\varphi(\underline{a}g) = \prod_{w \in \Sigma_\E}a_w^{r_w}\cdot \varphi(g)$ for $\underline{a}=(a_w)_{w \in \Sigma_\E} \in Z_G(\R)=(\R^\times)^{\Sigma_\E}$ and $g \in G(\A)$;
\item $\varphi$ is an eigenfunction of the Casimir operator of $\frak{g}_\C$ with eigenvalue $\prod_{w \in \Sigma_\E}(\tfrac{1}{2}\kappa_w^2-\kappa_w)$.
\end{itemize}
Let $\mathcal{A}_0(G(\A),(\underline{\kappa},\underline{r}))$ be the subspace of $\mathcal{A}_{(2)}(G(\A),(\underline{\kappa},\underline{r}))$ consisting of cusp forms on $G(\A)$.
We fix a $\Q(\underline{\kappa})$-rational structure on $\C_{(\underline{\kappa},\underline{r})}$ spanned by a non-zero vector ${\bf v}_{(\underline{\kappa},\underline{r})} \in \C_{(\underline{\kappa},\underline{r})}$. We also fix a $\widetilde{\E}$-rational structure on $\frak{p}^\pm$ with $\widetilde{\E}$-basis given by
\[
\left\{ X_{\pm,w}\,\vert\, w \in \Sigma_\E\right\},
\]
where $X_{\pm,w}$ is defined so that its $w$-component is equal to $X_\pm$ and zero otherwise. 
For $\sigma \in {\rm Aut}(\C)$, we have the $\sigma$-linear isomorphisms
\begin{align}\label{E:rational structures}
\begin{split}
\C_{(\underline{\kappa},\underline{r})} &\longrightarrow \C_{({}^\sigma\!\underline{\kappa},\underline{r})},\quad z\cdot {\bf v}_{(\underline{\kappa},\underline{r})} \longmapsto \sigma(z)\cdot {\bf v}_{({}^\sigma\!\underline{\kappa},\underline{r})};\\
\frak{p}^\pm &\longrightarrow \frak{p}^\pm, \quad \sum_{w \in \Sigma_\E} z_w\cdot X_{\pm,w} \longmapsto \sum_{w \in \Sigma_\E} \sigma(z_w)\cdot X_{\pm,\sigma\circ w} .
\end{split}
\end{align}
Note that we obtain a $\Q$-rational structure on $\frak{p}^{\pm}$ by taking the ${\rm Aut}(\C)$-invariants with respect to the above $\sigma$-linear isomorphism.

\begin{thm}\label{T:Harris}
Assume $(\underline{\kappa},\underline{r}) \in \Z[\Sigma_\E] \times \Z[\Sigma_\E]$ is motivic. 
\begin{itemize}
\item[(1)] For $\sigma \in {\rm Aut}(\C)$, with respect to the $\sigma$-linear isomorphisms in (\ref{E:rational structures}), conjugation by $\sigma$ induces natural $\sigma$-linear $G(\A_f)$-module isomorphisms
\[
T_\sigma: H^q(\mathcal{L}_{(\underline{\kappa},\underline{r})}^{\rm sub}) \longrightarrow H^q(\mathcal{L}_{({}^\sigma\!\underline{\kappa},\underline{r})}^{\rm sub}),\quad T_\sigma: H^q(\mathcal{L}_{(\underline{\kappa},\underline{r})}^{\rm can}) \longrightarrow H^q(\mathcal{L}_{({}^\sigma\!\underline{\kappa},\underline{r})}^{\rm can}),
\]
and such that the diagram 
\[
\begin{tikzcd}
H^q(\mathcal{L}_{(\underline{\kappa},\underline{r})}^{\rm sub}) \arrow[r, "T_\sigma"] \arrow[d] & H^q(\mathcal{L}_{({}^\sigma\!\underline{\kappa},\underline{r})}^{\rm sub}) \arrow[d]\\
H^q(\mathcal{L}_{(\underline{\kappa},\underline{r})}^{\rm can}) \arrow[r, "T_\sigma"]  & H^q(\mathcal{L}_{({}^\sigma\!\underline{\kappa},\underline{r})}^{\rm can})
\end{tikzcd}
\]
is commute. 
Moreover, $H^q(\mathcal{L}_{(\underline{\kappa},\underline{r})}^{\rm sub})$ and $H^q(\mathcal{L}_{(\underline{\kappa},\underline{r})}^{\rm can})$ are admissible $G(\A_{f})$-modules and have canonical rational structures over $\Q(\underline{\kappa})$ given by taking the Galois invariants with respect to $T_\sigma$ for $\sigma \in {\rm Aut}(\C / \Q(\underline{\kappa}))$.
\item[(2)] We have
\begin{align*}
H_{\rm cusp}^q(\mathcal{L}_{(\underline{\kappa}, \underline{r})}) &= \left(\mathcal{A}_{0}(G(\A),(\underline{\kappa},\underline{r}))\otimes_\C \bigwedge^q \frak{p}^+ \otimes_\C \C_{(\underline{\kappa},\underline{r})} \right)^{K_\infty},\\
H_{(2)}^q(\mathcal{L}_{(\underline{\kappa}, \underline{r})}) &= \left(\mathcal{A}_{(2)}(G(\A),(\underline{\kappa},\underline{r}))\otimes_\C \bigwedge^q \frak{p}^+ \otimes_\C \C_{(\underline{\kappa},\underline{r})} \right)^{K_\infty}.
\end{align*}
\item[(3)] The composite of the left vertical and lower horizontal homomorphisms in (\ref{E:commute 1}) is an injective $G(\A_f)$-module homomorphism $H_{\rm cusp}^q(\mathcal{L}_{(\underline{\kappa}, \underline{r})})\rightarrow H_!^q(\mathcal{L}_{(\underline{\kappa},\underline{r})})$.
\item[(4)] The image of the homomorphism $H_{(2)}^q(\mathcal{L}_{(\underline{\kappa},\underline{r})}) \rightarrow H^q(\mathcal{L}_{(\underline{\kappa},\underline{r})}^{\rm can})$ in (\ref{E:commute 1}) 
contains $H_!^q(\mathcal{L}_{(\underline{\kappa},\underline{r})})$.
In particular, $H_!^q(\mathcal{L}_{(\underline{\kappa},\underline{r})})$ is a semisimple $G(\A_f)$-module.
\end{itemize}
\end{thm}

Let $\mathcal{K}$ be a neat open compact subgroup of $G(\A_f)$. For $\star = {\rm sia}$ or $\rm rda$, we have the complexe ${}_{\mathcal{K}}C_{\star,(\underline{\kappa},\underline{r})}^*$ analogous to (\ref{E:complexes}) with $C_\star^\infty(G(\Q)\backslash G(\A))$ replacing by $C_\star^\infty(G(\Q)\backslash G(\A) / \mathcal{K})$.
We denote by $H_{\mathcal{K}}^q(\mathcal{L}_{(\underline{\kappa},\underline{r})}^{\rm can})$ (resp.\,$H_{\mathcal{K}}^q(\mathcal{L}_{(\underline{\kappa},\underline{r})}^{\rm sub})$) the corresponding $q$-th cohomology group if $\star = {\rm sia}$ (resp.\,$\star = {\rm rda}$). Note that $H^q_{\mathcal{K}}(\mathcal{L}_{(\underline{\kappa},\underline{r})}^{\rm can})$ and $H^q_{\mathcal{K}}(\mathcal{L}_{(\underline{\kappa},\underline{r})}^{\rm sub})$ are finite dimensional vector spaces over $\C$ and isomorphic to the $q$-th coherent cohomology groups of the canonical and subcanonical extension, respectively, of certain automorphic line bundle ${}_{\mathcal{K}}\mathcal{L}_{(\underline{\kappa},\underline{r})}$ on  ${\rm Sh}_{\mathcal K}(G,X)$ to its toroidal compactification. 
For each $g \in G(\A_f)$, let $\mathcal{K}^g = g^{-1}\mathcal{K}g$. The natural isomorphism ${\rm Sh}_{\mathcal{K}^g}(G,X) \rightarrow {\rm Sh}_{\mathcal K}(G,X)$ induces isomorphisms
\begin{align}\label{E:level iso.}
H^q_{\mathcal{K}}(\mathcal{L}_{(\underline{\kappa},\underline{r})}^{\rm can}) \longrightarrow H^q_{\mathcal{K}^g}(\mathcal{L}_{(\underline{\kappa},\underline{r})}^{\rm can}),\quad H^q_{\mathcal{K}}(\mathcal{L}_{(\underline{\kappa},\underline{r})}^{\rm sub}) \longrightarrow H^q_{\mathcal{K}^g}(\mathcal{L}_{(\underline{\kappa},\underline{r})}^{\rm sub}).
\end{align}
Then similar assertions as in Theorem \ref{T:Harris}-(1) hold for $H^q_{\mathcal{K}}(\mathcal{L}_{(\underline{\kappa},\underline{r})}^{\rm can})$ and $H^q_{\mathcal{K}}(\mathcal{L}_{(\underline{\kappa},\underline{r})}^{\rm sub})$ so that the corresponding $\sigma$-linear isomorphisms $T_\sigma$ in (1) are compatible with (\ref{E:level iso.}). 
Moreover, the natural morphisms of complexes ${}_{\mathcal{K}}C_{\star,(\underline{\kappa},\underline{r})}^* \rightarrow C_{\star,(\underline{\kappa},\underline{r})}^*$ induce $G(\A_f)$-module isomorphisms
\[
\varinjlim_{\mathcal{K}} H^q_{\mathcal{K}}(\mathcal{L}_{(\underline{\kappa},\underline{r})}^{\rm can}) \longrightarrow H^q(\mathcal{L}_{(\underline{\kappa},\underline{r})}^{\rm can}),\quad \varinjlim_{\mathcal{K}} H^q_{\mathcal{K}}(\mathcal{L}_{(\underline{\kappa},\underline{r})}^{\rm sub}) \longrightarrow H^q(\mathcal{L}_{(\underline{\kappa},\underline{r})}^{\rm sub})
\]
which are compatible with $T_\sigma$ for all $\sigma \in {\rm Aut}(\C)$. Here the $G(\A_f)$-module structure on the direct limits are defined by the isomorphisms (\ref{E:level iso.}).

Let $\underline{2} = (2,\cdots,2)$ and $\underline{0} = (0,\cdots,0)$. The automorphic line bundle ${}_{\mathcal{K}}\mathcal{L}_{(\underline{2},\underline{0})}^{\rm sub}$ is isomorphic to the dualizing sheaf for toroidal compactification of ${\rm Sh}_{\mathcal K}(G,X)$ (cf.\,\cite[Proposition 2.2.6]{Harris1990}) with trace map
\[
H^{[\E:\Q]}_{\mathcal{K}}(\mathcal{L}_{(\underline{2},\underline{0})}^{\rm sub}) \longrightarrow \C,\quad \omega \longmapsto \int_{{\rm Sh}_{\mathcal K}(G,X)}\omega,
\]
where
\[
\int_{{\rm Sh}_{\mathcal K}(G,X)}\omega = \int_{\G(\Q)\backslash X \times G(\A_f) / \mathcal{K}}\varphi\circ\iota_{\mathcal{K}}(x)\,d\mu_{\mathcal{K}}
\]
if $\omega$ is represented by $\varphi \otimes \bigwedge_{w \in \Sigma_\E}X_{+,w} \otimes {\bf v}_{(\underline{2},\underline{0})} \in {}_{\mathcal{K}}C_{{\rm rda},(\underline{2},\underline{0})}^{[\E:\Q]}$. Here 
\[
\iota_{\mathcal{K}}: G(\Q)\backslash X \times G(\A_f) / \mathcal{K} \longrightarrow G(\Q)\backslash G(\A) / K_\infty\mathcal{K}
\]
is the natural isomorphism, and the measure on $G(\Q)\backslash X \times G(\A_f) / \mathcal{K}$ is defined as follows:
\begin{itemize}
\item on $G(\A_f) / \mathcal{K}$, we take the counting measure;
\item on $X$, we take the measure $\displaystyle \left(\frac{dz\wedge d\overline{z}}{2\pi\sqrt{-1}}\right)^{[\E:\Q]}$ on $(\frak{H}^{\pm})^{\Sigma_\E}$ via the isomorphism (\ref{E:hermitian domain}).
\end{itemize}
Under this normalization of measure, we have the following Galois equivariant property:
\begin{align}\label{E:Galois equiv. level}
\sigma \left(\int_{{\rm Sh}_{\mathcal K}(G,X)}\omega \right) = \int_{{\rm Sh}_{\mathcal K}(G,X)}T_\sigma \omega
\end{align}
for $\omega \in H^{[\E:\Q]}_{\mathcal{K}}(\mathcal{L}_{(\underline{2},\underline{0})}^{\rm sub})$ and $\sigma \in {\rm Aut}(\C)$ (cf.\,\cite[(3.8.4)]{Harris1990}).
In the following lemma, we show that the family of trace maps as $\mathcal{K}$ varies can be normalized to define a trace map $H^{[\E:\Q]}(\mathcal{L}_{(\underline{2},\underline{0})}^{\rm sub}) \rightarrow \C$.

\begin{lemma}\label{L:trace}
We have the $G(\A_f)$-equivariant trace map
\[
H^{[\E:\Q]}(\mathcal{L}_{(\underline{2},\underline{0})}^{\rm sub}) \longrightarrow \C,\quad \omega \longmapsto \int_{{\rm Sh}(G,X)}\omega,
\]
where
\[
\int_{{\rm Sh}(G,X)}\omega = [ \hat{\o}_\E^\times : \o_\E^\times \cdot \mathcal{U}]^{-1} \sum_{a \in \E^\times \backslash \A_\E^\times / \E_\infty^\times \mathcal{U}}\int_{Z_G(\A)G(\Q) \backslash G(\A)}\varphi(ag)\,dg^{\rm Tam}
\]
if $\omega$ is represented by $\varphi \otimes \bigwedge_{w \in \Sigma_\E}X_{+,w} \otimes {\bf v}_{(\underline{2},\underline{0})} \in C_{{\rm rda},(\underline{2},\underline{0})}^{[\E:\Q]}$. Here $dg^{\rm Tam}$ is the Tamagawa measure on $Z_G(\A) \backslash G(\A)$ and $\mathcal{U}$ is any open compact subgroup of $\A_{\E,f}^\times$ such that $\varphi$ is right $\mathcal{U}$-invariant.
Moreover, the trace map satisfies the Galois equivariant property:
\[
\sigma \left(\int_{{\rm Sh}(G,X)}\omega \right) = \int_{{\rm Sh}(G,X)}T_\sigma \omega
\]
for $\omega \in H^{[\E:\Q]}(\mathcal{L}_{(\underline{2},\underline{0})}^{\rm sub})$ and $\sigma \in {\rm Aut}(\C)$.
\end{lemma}

\begin{proof}

We identify $Z_G(\A)$ with $\A_\E^\times$. 
Let $\omega \in H^{[\E:\Q]}(\mathcal{L}_{(\underline{2},\underline{0})}^{\rm sub})$ be a class represented by $\varphi \otimes \bigwedge_{w \in \Sigma_\E}X_{+,w} \otimes {\bf v}_{(\underline{2},\underline{0})} \in C_{{\rm rda},(\underline{2},\underline{0})}^{[\E:\Q]}$. 
For any neat open compact subgroup $\mathcal{K}$ of $G(\A_f)$ such that $\varphi$ is right $\mathcal{K}$-invariant, we let $\mathcal{U}_{\mathcal{K}} = Z_G(\A_f)\cap \mathcal{K}$ and $\omega_{\mathcal{K}} \in H_{\mathcal{K}}^{[\E:\Q]}(\mathcal{L}_{(\underline{2},\underline{0})}^{\rm sub})$ be the class represented by $\varphi \otimes \bigwedge_{w \in \Sigma_\E}X_{+,w} \otimes {\bf v}_{(\underline{2},\underline{0})}$ considered as an element in ${}_{\mathcal K}C_{{\rm rda},(\underline{2},\underline{0})}^{[\E:\Q]}$.
By \cite[Lemmas 6.1 and 6.3]{IP2018}, there exists a non-zero rational number $C$ depending only on $G$ such that
\begin{align*}
\int_{{\rm Sh}_{\mathcal K}(G,X)} \omega_{\mathcal{K}}
&= \int_{\G(\Q)\backslash X\times G(\A_f) / \mathcal{K}} \varphi\circ\iota_{\mathcal{K}}(x)\,d\mu_{\mathcal K}\\
&= C \cdot [\mathcal{K}_0:\mathcal{K}]\cdot [\hat{\o}_\E^\times : \mathcal{U}_{\mathcal{K}}]^{-1}\cdot \int_{G(\Q) \backslash G(\A) / Z_G(\R)\mathcal{U}_{\mathcal{K}}}\varphi(g)\,dg^{\rm Tam}.
\end{align*}
Here $\mathcal{K}_0$ is any maximal open compact subgroup of $G(\A_f)$ containing $\mathcal{K}$.
We remark that the above formula was proved in \cite[Lemmas 6.3]{IP2018} for $\mathcal{U}_{\mathcal K} = \hat{\o}_\E^\times$. The general case can be proved in a similar way.
We rewrite the formula as
\begin{align*}
&C^{-1}\cdot [\mathcal{K}_0:\mathcal{K}]^{-1} \cdot [\o_\E^\times\cdot\mathcal{U}_\mathcal{K}: \mathcal{U}_\mathcal{K}]\cdot \int_{{\rm Sh}_{\mathcal K}(G,X)} \omega_{\mathcal{K}}\\
&=  [\hat{\o}_\E^\times : \o_\E^\times\cdot\mathcal{U}_{\mathcal{K}}]^{-1}\cdot \int_{G(\Q) \backslash G(\A) / Z_G(\R)\mathcal{U}_{\mathcal{K}}}\varphi(g)\,dg^{\rm Tam}\\
& = [\hat{\o}_\E^\times : \o_\E^\times\cdot\mathcal{U}_{\mathcal{K}}]^{-1}\cdot \sum_{a \in \E^\times\backslash \A_\E^\times / \E_\infty^\times \mathcal{U}_\K}\int_{Z_G(\A)G(\Q) \backslash G(\A) }\varphi(ag)\,dg^{\rm Tam}
\end{align*}
Note that the constant $C^{-1}\cdot [\mathcal{K}_0:\mathcal{K}]^{-1} \cdot [\o_\E^\times\cdot\mathcal{U}_\mathcal{K}: \mathcal{U}_\mathcal{K}]$ depends only on $G$ and $\mathcal{K}$.
Let $\mathcal{U}$ be any open compact subgroup of $\A_{\E,f}^\times$ such that $\varphi$ is right $\mathcal{U}$-invariant.
Then we have
\begin{align*}
&[\hat\o_\E^\times : \o_\E^\times\cdot(\mathcal{U}\cap\mathcal{U}_{\mathcal K})]^{-1}\sum_{a \in \E^\times \backslash \A_\E^\times / \E_\infty^\times (\mathcal{U}\cap\mathcal{U}_{\mathcal K})}\int_{Z_G(\A)G(\Q) \backslash G(\A)}\varphi(ag)\,dg^{\rm Tam}\\
& = [\hat\o_\E^\times : \o_\E^\times\cdot(\mathcal{U}\cap\mathcal{U}_{\mathcal K})]^{-1}\cdot [\E^\times \E_\infty^\times \mathcal{U} : \E^\times \E_\infty^\times (\mathcal{U}\cap\mathcal{U}_{\mathcal K})]\sum_{a \in \E^\times \backslash \A_\E^\times / \E_\infty^\times \mathcal{U}}\int_{Z_G(\A)G(\Q) \backslash G(\A)}\varphi(ag)\,dg^{\rm Tam}\\
& = [\hat\o_\E^\times : \o_\E^\times\cdot\mathcal{U}]^{-1} \sum_{a \in \E^\times \backslash \A_\E^\times / \E_\infty^\times \mathcal{U}}\int_{Z_G(\A)G(\Q) \backslash G(\A)}\varphi(ag)\,dg^{\rm Tam}.
\end{align*}
We conclude that the trace map $\omega \mapsto \int_{{\rm Sh}(G,X)}\omega$ is well-defined. Finally, the Galois equivariance property follows from $C \in \Q^\times$ and (\ref{E:Galois equiv. level}). This completes the proof.
\end{proof}

\subsection{Rational structures via the coherent cohomology}\label{SS:2.2}

Let $\itPi = \bigotimes_v \pi_v$ be an irreducible cuspidal automorphic representation of $G(\A)$. 
Let $\pi_f = \bigotimes_p \pi_p$ be the finite part of $\itPi$.
We assume that $\itPi$ is motivic, that is, there exists motivic $(\underline{\kappa},\underline{r}) \in \Z_{\geq 1}[\Sigma_\E] \times \Z[\Sigma_\E]$ such that
\[
\pi_\infty  = \large{\boxtimes}_{w \in \Sigma_\E}(D(\kappa_w)\otimes |\mbox{ }|^{-r_w/2}).
\]
We call $\underline{\kappa}$ (resp.\,$(\underline{\kappa},\underline{r})$) the weight (resp.\,motivic weight) of $\itPi$. When $\E$ is a field, we necessary have $\underline{r} = (r,\cdots,r)$ for some $r \in \Z$ and also call $(\underline{\kappa},r) \in \Z_{\geq 1}[\Sigma_\E]\times\Z$ the motivic weight of $\itPi$.
Note that $\itPi$ occurs in $\mathcal{A}_0(G(\A), (\underline{\kappa},\underline{r}))$.
For each motivic $(\underline{\kappa}',\underline{r}') \in \Z[\Sigma_\E] \times \Z[\Sigma_\E]$ and $\star \in \{{\rm cusp}, (2), !\}$, we denote by $H_\star^q(\mathcal{L}_{(\underline{\kappa}',\underline{r}')})[\pi_f]$ the $\pi_f$-isotypic component of $\pi_f$ in $H_\star^q(\mathcal{L}_{(\underline{\kappa}',\underline{r}')})$.

\begin{lemma}\label{L:2.2}
Let $(\underline{\kappa}',\underline{r}') \in \Z[\Sigma_\E] \times \Z[\Sigma_\E]$ be motivic. 
\begin{itemize}
\item[(1)] If $\underline{r} \neq \underline{r}'$ or $(\kappa_w-\kappa_w')(\kappa_w+\kappa_w'-2) \neq 0$ for some $w \in \Sigma_\E$, then $H_\star^q(\mathcal{L}_{(\underline{\kappa}',\underline{r}')})[\pi_f]=0$ for all $q$ and $\star \in \{{\rm cusp}, (2), !\}$.
\item[(2)] If $\underline{r} = \underline{r}'$ and $(\kappa_w-\kappa_w')(\kappa_w+\kappa_w'-2) = 0$ for all $w \in \Sigma_\E$, then the homomorphism in Theorem \ref{T:Harris}-(3) induces an isomorphism of $G(\A_f)$-modules
$ H_{\rm cusp}^q(\mathcal{L}_{(\underline{\kappa}',\underline{r})})[\pi_f] \rightarrow H_!^q(\mathcal{L}_{(\underline{\kappa}',\underline{r})})[\pi_f]$.
Moreover, the multiplicity of $\pi_f$ in $H_{\rm cusp}^q(\mathcal{L}_{(\underline{\kappa}',\underline{r})})[\pi_f]$ is equal to the number of subsets $I$ of $\Sigma_\E$ such that
\[
\underline{\kappa}' = \underline{\kappa}(I),\quad q = {}^\sharp I.
\]
\end{itemize}
\end{lemma}

\begin{proof} 

For each subset $I$ of $\Sigma_\E$, let 
\[
\varepsilon(I)_w = \begin{cases}
- & \mbox{ if $w \in I$},\\
+  & \mbox{ if $w \notin I$}
\end{cases}
\]
for $w \in \Sigma_\E$ and
\[
\pi_{\infty,I} = \boxtimes_{w \in \Sigma_\E}(D(\kappa_w)^{\varepsilon(I)_w}\otimes |\mbox{ }|^{-r_w/2})
\]
be an irreducible admissible $(\frak{g},K_\infty)$-module. Note that we have a $(\frak{g},K_\infty)$-module isomorphism
\[
\pi_\infty \simeq \bigoplus_{I \subseteq \Sigma_\E}\pi_{\infty,I}.
\]
Specializing \cite[Theorem 4.6.2]{Harris1990} to the (limit of) discrete series representation $\pi_{\infty,I}$, we have
\begin{align}\label{E:d.s. cohomology}
{\rm dim}\,\left(\pi_{\infty,I} \otimes_\C \bigwedge^q \frak{p}^+ \otimes_\C \C_{(\underline{\kappa}',\underline{r}')}\right)^{K_\infty} = \begin{cases}
0 & \mbox{ if $(\underline{\kappa}',\underline{r}') \neq (\underline{\kappa}(I),\underline{r})$ or $q \neq {}^\sharp I$},\\
1 & \mbox{ if $(\underline{\kappa}',\underline{r}') = (\underline{\kappa}(I),\underline{r})$ and $q = {}^\sharp I$}.
\end{cases}
\end{align}
On the other hand, by the strong multiplicity one theorem, we have
\[
\mathcal{A}_0(G(\A),(\underline{\kappa},\underline{r}))[\pi_f] = \mathcal{A}_{(2)}(G(\A),(\underline{\kappa},\underline{r}))[\pi_f] = \itPi.
\]
The assertions then follow from Theorem \ref{T:Harris}-(2)-(4) and (\ref{E:d.s. cohomology}).
This completes the proof.
\end{proof}

For $\sigma \in {\rm Aut}(\C)$, let ${}^\sigma\!\itPi$ be the irreducible admissible representation of $G(\A)$ defined by
\[
{}^\sigma\!\itPi = {}^\sigma\!\pi_{\infty} \otimes {}^\sigma\!\pi_f,
\]
where ${}^\sigma\!\pi_f$ is the $\sigma$-conjugate of $\pi_f$ and ${}^\sigma\!\pi_{\infty}$ is the representation of $G(\R) = \GL_2(\R)^{\Sigma_\E}$ so that its $w$-component is equal to the $\sigma^{-1} \circ w$-component of $\pi_\infty$. 
The following lemma is well-known and can be deduced from the result of Shimura \cite[Proposition 1.6]{Shimura1978}. When $\kappa_w \geq 2$ for all $w \in \Sigma_\E$, the lemma was also proved by Waldspurger \cite{Wald1985B} and Harder \cite{Harder1987} and is based on the study of $(\frak{g},K_\infty)$-cohomology. We provide another proof based on the results of Harris \cite{Harris1990}, which is $(\frak{P},K_\infty)$-cohomological in natural.

\begin{lemma}
For $\sigma \in {\rm Aut}(\C)$, the representation ${}^\sigma\!\itPi$ is a motivic irreducible cuspidal automorphic representation of $G(\A)$ of motivic weight $({}^\sigma\!\underline{\kappa},\underline{r})$.
Moreover, the rationality field $\Q(\itPi)$ of $\itPi$ is equal to the fixed field of $\left\{\sigma \in {\rm Aut}(\C) \, \vert \, {}^\sigma\!\itPi = \itPi \right\}$.
\end{lemma}

\begin{proof}
Fix $\sigma \in {\rm Aut}(\C)$. Since $H_!^0(\mathcal{L}_{(\underline{\kappa},\underline{r})})[\pi_f] \simeq \pi_f$, we have 
\[
H_!^0(\mathcal{L}_{({}^\sigma\!\underline{\kappa},\underline{r})})[{}^\sigma\!\pi_f] = T_\sigma (H_!^0(\mathcal{L}_{(\underline{\kappa},\underline{r})})[\pi_f]) \simeq {}^\sigma\!\pi_f.
\]
On the other hand, the homomorphism in Theorem \ref{T:Harris}-(3) is an isomorphism by \cite[Proposition 5.4.2]{Harris1990}.
It follows that
\[
H_{\rm cusp}^0(\mathcal{L}_{({}^\sigma\!\underline{\kappa},\underline{r}))})[{}^\sigma\!\pi_f] = \left(\mathcal{A}_0(G(\A),({}^\sigma\!\underline{\kappa},\underline{r}))[{}^\sigma\!\pi_f]\otimes_\C \C_{({}^\sigma\!\underline{\kappa},\underline{r})} \right)^{K_\infty} \simeq {}^\sigma\!\pi_f.
\]
We conclude that $\mathcal{A}_0(G(\A),({}^\sigma\!\underline{\kappa},\underline{r}))[{}^\sigma\!\pi_f] = \itPi'$ for some irreducible cuspidal automorphic representation $\itPi' = \pi_\infty' \otimes {}^\sigma\!\pi_f$ of $G(\A)$ such that
\[
 \left(\pi_\infty'  \otimes_\C \C_{({}^\sigma\!\underline{\kappa},\underline{r})} \right)^{K_\infty} \neq 0.
\]
This implies that there exists a non-zero vector $v \in \pi_\infty'$ such that
\begin{align}\label{E:2.1}
\pi_\infty'(\underline{a}\cdot\underline{k}_\theta) v = \prod_{w \in \Sigma_\E} a_w^{r_w}e^{\sqrt{-1}\,{}^\sigma\!\kappa_w \theta_w}\cdot v
\end{align}
for $\underline{a} = (a_w)_{w \in \Sigma_\E} \in (\R^\times)^{\Sigma_\E}$ and $\underline{k}_\theta = (k_{\theta_w} )_{w \in \Sigma_\E}\in {\rm SO}(2)^{\Sigma_\E}$.
In particular, we have
\[
(H,\cdots,H)\cdot v = \prod_{w \in \Sigma_\E}{}^\sigma\!\kappa_w \cdot v = \prod_{w \in \Sigma_\E} \kappa_w\cdot v.
\]
Note that the Casimir operator of $\frak{gl}_{2,\C}$ is given by
\[
\tfrac{1}{2}H^2-H-\tfrac{1}{2}X_+X_-.
\]
We deduce from the second condition defining the space $\mathcal{A}_0(G(\A),({}^\sigma\!\underline{\kappa},\underline{r}))$ that
\begin{align}\label{E:2.2}
(X_+X_-,\cdots,X_+X_-)\cdot v = 0.
\end{align}
By \cite[Lemma 5.6]{JLbook}, we see that (\ref{E:2.1}) and (\ref{E:2.2}) imply that 
\[
\pi_\infty'  = \boxtimes_{w \in \Sigma_\E}(D({}^\sigma\!\kappa_w)\otimes |\mbox{ }|^{-r_w/2}).
\]
Therefore $\itPi'$ is isomorphic to ${}^\sigma\!\itPi$. 
For the second assertion, assume ${}^\sigma\!\pi_f = \pi_f$. Then it follows from the strong multiplicity one theorem that ${}^\sigma\!\itPi = \itPi$. This completes the proof.
\end{proof}

\begin{lemma}\label{L:coherent structure}
Let $I \subseteq \Sigma_\E$.
For any field extension $\Q(\itPi,I) \subseteq F \subseteq \C$, we have the $F$-rational structure on $H_{\rm cusp}^{{}^\sharp I}(\mathcal{L}_{(\underline{\kappa}(I) ,\underline{r})})[\pi_f]$  by taking the ${\rm Aut}(\C/F)$-invariants
\begin{align*}
H_{\rm cusp}^{{}^\sharp I}(\mathcal{L}_{(\underline{\kappa}(I) ,\underline{r})})[\pi_f]^{{\rm Aut}(\C/F)} &= \left.\left\{c \in H_{\rm cusp}^{{}^\sharp I}(\mathcal{L}_{(\underline{\kappa}(I) ,\underline{r})})[\pi_f] \,\right\vert\, T_\sigma c=c  \mbox{ for all }\sigma \in {\rm Aut}(\C/F)\right\}.
\end{align*}
Here $\Q(\itPi,I) = \Q(\itPi)\cdot\Q(I)$.
Moreover, suppose that $I$ is admissible with respect to $\underline{\kappa}$, then the $G(\A_f)$-module $H_{\rm cusp}^{{}^\sharp I}(\mathcal{L}_{(\underline{\kappa}(I) ,\underline{r})})[\pi_f]$ is isomorphic to $\pi_f$ and the $F$-rational structures are unique up to homotheties. 
\end{lemma}

\begin{proof}
By the commutativity of the diagram in Theorem \ref{T:Harris}-(1) and Lemma \ref{L:2.2}-(2), we have
\begin{align}\label{E:2.12 Galois}
H_{\rm cusp}^{{}^\sharp I}(\mathcal{L}_{({}^\sigma\!\underline{\kappa}({}^\sigma\!I) ,\underline{r})})[{}^\sigma\!\pi_f] = T_\sigma (H_{\rm cusp}^{{}^\sharp I}(\mathcal{L}_{(\underline{\kappa}(I) ,\underline{r})})[\pi_f])
\end{align}
for all $\sigma \in {\rm Aut}(\C)$.
Let $\Q(\itPi)\cdot\Q(I) \subseteq F \subseteq \C$ be a field extension.
By Theorem \ref{T:Harris}-(1), $H^{{}^\sharp I}(\mathcal{L}_{(\underline{\kappa}(I) ,\underline{r})}^{\rm sub})$ admits a $F$-rational structure given by
\[
H^{{}^\sharp I}(\mathcal{L}_{(\underline{\kappa}(I) ,\underline{r})}^{\rm sub})^{{\rm Aut}(\C/F)}= \left.\left\{c \in H^{{}^\sharp I}(\mathcal{L}_{(\underline{\kappa}(I) ,\underline{r})}^{\rm sub}) \,\right\vert\, T_\sigma c=c  \mbox{ for all }\sigma \in {\rm Aut}(\C/F))\right\}.
\]
By (\ref{E:2.12 Galois}), $H_{\rm cusp}^{{}^\sharp I}(\mathcal{L}_{(\underline{\kappa}(I) ,\underline{r})})[\pi_f]$ is invariant by $T_\sigma$ for $\sigma \in {\rm Aut}(\C/F)$.
Therefore, by \cite[Lemme 3.2.1]{Clozel1990}, we have a $F$-rational structure on $H_{\rm cusp}^{{}^\sharp I}(\mathcal{L}_{(\underline{\kappa}(I) ,\underline{r})})[\pi_f]$ given by
\[
H_{\rm cusp}^{{}^\sharp I}(\mathcal{L}_{(\underline{\kappa}(I) ,\underline{r})})[\pi_f] \cap H^{{}^\sharp I}(\mathcal{L}_{(\underline{\kappa}(I) ,\underline{r})}^{\rm sub})^{{\rm Aut}(\C/F)} = H_{\rm cusp}^{{}^\sharp I}(\mathcal{L}_{(\underline{\kappa}(I) ,\underline{r})})[\pi_f]^{{\rm Aut}(\C/F)}.
\]
Finally, suppose that $I$ is admissible with respect to $\underline{\kappa}$. Then $H_{\rm cusp}^{{}^\sharp I}(\mathcal{L}_{(\underline{\kappa}(I) ,\underline{r})})[\pi_f] \simeq \pi_f$ is irreducible by Lemma \ref{L:2.2}-(2) and the uniqueness up to homotheties then follows from Schur's lemma. 
This completes the proof.
\end{proof}

\subsection{Harris' periods}\label{SS: Harris' period}

Let $\itPi$ be a motivic irreducible cuspidal automorphic representation of $G(\A)$ of motivic weight $(\underline{\kappa},\underline{r}) \in \Z_{\geq 1}[\Sigma_\E]\times\Z[\Sigma_\E]$. 
Let $\itPi_{\,{\rm hol}}$ be the subspace of $\itPi$ consisting of holomorphic cusp forms $\varphi \in \itPi$, that is, 
\[
\varphi(g\underline{k}_\theta) = e^{\sqrt{-1}\,\sum_{w \in \Sigma_\E}\kappa_w\theta_w}\varphi(g)
\]
for $\underline{k}_\theta = (k_{\theta_w} )_{w \in \Sigma_\E}\in {\rm SO}(2)^{\Sigma_\E}$ and $g \in G(\A)$.
Let $I$ be a subset of $\Sigma_\E$. 
Define $\tau_I \in \G(\R) = \GL_2(\R)^{\Sigma_\E}$ by
\begin{align}\label{E:tau_I}
\tau_I = ({\bf a}(\varepsilon(I)_w))_{w \in \Sigma_\E},
\end{align}
where
\[
\varepsilon(I)_w = \begin{cases}
-1 & \mbox{ if $w \in I$},\\
1  & \mbox{ if $w \notin I$}.
\end{cases}
\]
For $\varphi \in \itPi_{\,\rm hol}$, let $\varphi^I \in \itPi$ defined by
\begin{align}\label{E:phi^I}
\varphi^I(g) = \varphi(g \cdot \tau_I)
\end{align}
for $g \in G(\A)$. 
We then have the homomorphism of $G(\A_f)$-modules:
\begin{align}\label{E:xi_I}
\xi_I : \itPi_{\,\rm hol} \longrightarrow H_{\rm cusp}^{{}^\sharp I}(\mathcal{L}_{(\underline{\kappa}(I) ,\underline{r})})[\pi_f],\quad
\varphi \longmapsto \varphi^I \otimes \bigwedge_{w \in I} X_{+,w} \otimes {\bf v}_{(\underline{\kappa}(I),\underline{r})}.
\end{align}
Here we fix an ordering of the wedge $\bigwedge_{w \in I} X_{+,w}$ once and for all such that
$
\bigwedge_{w \in I} X_{+,w} \mapsto \bigwedge_{w \in {}^\sigma\!I} X_{+,w}
$
under the $\sigma$-linear isomorphism in (\ref{E:rational structures}) for all $\sigma \in {\rm Aut}(\C)$.
Note that $\xi_I$ is an isomorphism if and only if $I$ is admissible with respect to $\underline{\kappa}$ by Lemma \ref{L:2.2}-(2).
By taking $I=\emptyset$ to be the empty set, we identify $\itPi_{\,\rm hol}$ with $H_{\rm cusp}^{0}(\mathcal{L}_{(\underline{\kappa} ,\underline{r})})[\pi_f]$ by the isomorphism $\xi_\emptyset$.
Moreover, for $\sigma \in {\rm Aut}(\C)$ we have the $\sigma$-linear isomorphism
\begin{align}\label{E:sigma iso.}
\itPi_{\,\rm hol} \longrightarrow {}^\sigma\!\itPi_{\,\rm hol}, \quad \varphi \longmapsto {}^\sigma\!\varphi \end{align}
defined so that the diagram
\[
\begin{tikzcd}
\itPi_{\,\rm hol} \arrow[d, "\xi_\emptyset"] \arrow[r]
    &  {}^\sigma\!\itPi_{\,\rm hol}\arrow[d, "\xi_\emptyset"] \\
  H_{\rm cusp}^{0}(\mathcal{L}_{(\underline{\kappa} ,\underline{r})})[\pi_f] \arrow[r, "T_\sigma"]
& H_{\rm cusp}^{0}(\mathcal{L}_{({}^\sigma\!\underline{\kappa} ,\underline{r})})[{}^\sigma\!\pi_f]
\end{tikzcd}
\]
is commute.
Comparing the $\Q(\itPi,I)$-rational structures on $H_{\rm cusp}^{0}(\mathcal{L}_{(\underline{\kappa} ,\underline{r})})[\pi_f]$ and $H_{\rm cusp}^{{}^\sharp I}(\mathcal{L}_{(\underline{\kappa}(I) ,\underline{r})})[\pi_f]$ defined in Lemma \ref{L:coherent structure}, we obtain the following result.

\begin{prop}\label{P:Harris' period}
Let $I \subseteq \Sigma_\E$ be admissible with respect to $\underline{\kappa}$.
There exists $\Omega^I(\itPi) \in \C^\times$, unique up to $\Q(\itPi,I)^\times$, such that
\[
\frac{\xi_I \left(\itPi_{\rm hol}^{{\rm Aut}(\C / \Q(\itPi,I))}\right)}{\Omega^I(\itPi)}  = H_{\rm cusp}^{{}^\sharp I}(\mathcal{L}_{(\underline{\kappa}(I) ,\underline{r})})[\pi_f]^{\rm Aut(\C/\Q(\itPi,I))}.
\]
Moreover, we can normalize the periods so that
\begin{align*}
T_\sigma\left(\frac{\xi_I(\varphi)}{\Omega^I(\itPi)}\right) = \frac{\xi_{{}^\sigma\!I}({}^\sigma\!\varphi)}{\Omega^{{}^\sigma\!I}({}^\sigma\!\itPi)}
\end{align*}
for $\varphi \in \itPi_{\,\rm hol}$ and $\sigma \in {\rm Aut}(\C)$.
\end{prop}

\begin{rmk}
For non-admissible $I \subseteq \Sigma_\E$, the multiplicity of $\pi_f$ in $H_{\rm cusp}^{{}^\sharp I}(\mathcal{L}_{(\underline{\kappa}(I) ,\underline{r})})[\pi_f]$ is greater than one and it is not known whether the equality $T_\sigma\circ \xi_I(\itPi_{\rm hol}) = \xi_{{}^\sigma\!I}({}^\sigma\!\itPi_{\rm hol})$ holds for any $\sigma \in {\rm Aut}(\C)$. Therefore, we do not know whether $\xi_I(\itPi_{\rm hol})$ is defined over $\Q(\itPi,I)$.
However, since $H_{\rm cusp}^{{}^\sharp I}(\mathcal{L}_{(\underline{\kappa}(I) ,\underline{r})})[\pi_f]$ is defined over $\Q(\itPi)\cdot\Q(\underline{\kappa}(I))$, it follows that $\xi_I(\itPi_{\rm hol})$ is define over some finite extension $F$ over $\Q(\itPi)\cdot\Q(\underline{\kappa}(I))$.
Hecne there exists $\Omega^I(\itPi) \in \C^\times$, unique up to $F^\times$, such that
\[
\frac{\xi_I \left(\itPi_{\rm hol}^{{\rm Aut}(\C / F)}\right)}{\Omega^I(\itPi)}  = \xi_I(\itPi_{\rm hol})\cap H_{\rm cusp}^{{}^\sharp I}(\mathcal{L}_{(\underline{\kappa}(I) ,\underline{r})})[\pi_f]^{{\rm Aut}(\C/F)}.
\]
\end{rmk}

We have assume that $\E = \F_1 \times \cdots \times \F_n$
and $D = D_1 \times \cdots \times D_n$ for some totally real number fields $\F_i$ and some totally indefinite quaternion algebra $D_i$ over $\F_i$. Then
\[
\itPi = \itPi_1 \times \cdots \times \itPi_n
\]
for some motivic irreducible cuspidal automorphic representation $\itPi_i$ of $D_i^\times(\A_{\F_i})$ for $1 \leq i \leq n$.
We identify $\Sigma_\E$ with the disjoint union of $\Sigma_{\F_i}$ for $1 \leq i \leq n$ in a natural way. Then 
\[
I = \bigsqcup_{i=1}^{n}I \cap \Sigma_{\F_i}.
\]
We have the following period relation.
\begin{lemma}\label{L:decomposition}
Let $I \subseteq \Sigma_\E$ be admissible with respect to $\underline{\kappa}$.
For $\sigma \in {\rm Aut}(\C)$, we have
\[
\sigma\left( \frac{\Omega^I(\itPi)}{\prod_{i=1}^n\Omega^{I\cap \Sigma_{\F_i}}(\itPi_i)}\right) = \frac{\Omega^{{}^\sigma\!I}({}^\sigma\!\itPi)}{\prod_{i=1}^n\Omega^{{}^\sigma\!I\cap \Sigma_{\F_i}}({}^\sigma\!\itPi_i)}.
\]
\end{lemma}

\begin{proof}
Indeed, we have
\[
{\rm Sh}(G,X) = {\rm Sh}(G_1,X_1) \times \cdots \times {\rm Sh}(G_n,X_n),
\]
where $(G_i,X_i) = ({\rm R}_{\F_i/\Q}D_i^\times,(\frak{H}^{\pm})^{\Sigma_{\F_i}})$.
Write 
\[
(\underline{\kappa},\underline{r}) =   (\underline{\kappa}_1,\underline{r}_1) \times \cdots \times (\underline{\kappa}_n,\underline{r}_n)
\]
under the identification
\[
\Z_{\geq 1}[\Sigma_\E] \times \Z[\Sigma_\E]=(\Z_{\geq 1}[\Sigma_{\F_1}] \times \Z[\Sigma_{\F_1}]) \times \cdots \times (\Z_{\geq 1}[\Sigma_{\F_n}] \times \Z[\Sigma_{\F_n}]).
\]
Then we have
\[
\mathcal{L}_{(\underline{\kappa}(I),\underline{r})} = \mathcal{L}_{(\underline{\kappa}_i(I\cap \Sigma_{\F_1}),\underline{r}_1)} \times \cdots \times \mathcal{L}_{(\underline{\kappa}_i(I\cap \Sigma_{\F_n}),\underline{r}_n)}.
\]
Then it follows from the K\"unneth formula that we have a canonical $G(\A_f)$-module isomorphism
\[
H^{{}^\sharp I}(\mathcal{L}_{(\underline{\kappa}(I),\underline{r})}^\star) \simeq \bigoplus_{q_1+\cdots+q_n={}^\sharp\!I}\bigotimes_{i=1}^{n} H^{q_i}(\mathcal{L}_{(\underline{\kappa}_i(I\cap \Sigma_{\F_i}),\underline{r}_i)}^\star)
\]
for $\star \in \{{\rm sub}, {\rm can}\}$.
Taking the $\pi_f$-isotypic parts and note that 
$H_!^{q_i}(\mathcal{L}_{(\underline{\kappa}_i(I\cap \Sigma_{\F_i}),\underline{r}_i)})[\pi_{i,f}]$ is zero unless  $q_i = {}^\sharp I\cap \Sigma_{\F_i}$ by Lemma \ref{L:2.2} and the admissibility of $I$, we thus obtain an isomorphism
\[
H_!^{{}^\sharp I}(\mathcal{L}_{(\underline{\kappa}(I),\underline{r})})[\pi_f] \simeq \bigotimes_{i=1}^{n} H_!^{{}^\sharp I\cap \Sigma_{\F_i}}(\mathcal{L}_{(\underline{\kappa}_i(I\cap \Sigma_{\F_i}),\underline{r}_i)})[\pi_{i,f}].
\]
We deduce from Lemma \ref{L:2.2}-(2) that
\[
H_{\rm cusp}^{{}^\sharp I}(\mathcal{L}_{(\underline{\kappa}(I),\underline{r})})[\pi_f] \simeq \bigotimes_{i=1}^{n} H_{\rm cusp}^{{}^\sharp I\cap \Sigma_{\F_i}}(\mathcal{L}_{(\underline{\kappa}_i(I\cap \Sigma_{\F_i}),\underline{r}_i)})[\pi_{i,f}].
\]
In the above isomorphism, we normalize the $\Q(\underline{\kappa}_i)$-rational structure on $\C_{(\underline{\kappa}_i,\underline{r}_i)}$ for $1 \leq i \leq n$ such that
\[
{\bf v}_{(\underline{\kappa},\underline{r})} = \bigotimes_{i=1}^{n} {\bf v}_{(\underline{\kappa}_i,\underline{r}_i)}
\]
under the isomorphism
\[
(\rho_{(\underline{\kappa},\underline{r})}, \C_{(\underline{\kappa},\underline{r})}) \simeq \bigotimes_{i=1}^{n} (\rho_{(\underline{\kappa}_i,\underline{r}_i)}, \C_{(\underline{\kappa}_i,\underline{r}_i)}).
\]
Then for $\sigma \in {\rm Aut}(\C)$, we have
\[
T_\sigma = T_{\sigma}^{(1)} \otimes \cdots \otimes T_{\sigma}^{(n)}.
\]
Here
\[
T_\sigma^{(i)} : H^{{}^\sharp I\cap \Sigma_{\F_i}}(\mathcal{L}_{(\underline{\kappa}_i(I\cap \Sigma_{\F_i}),\underline{r}_i)}^{\rm sub}) \longrightarrow H^{{}^\sharp I\cap \Sigma_{\F_i}}(\mathcal{L}_{({}^\sigma\!\underline{\kappa}_i({}^\sigma\!I\cap \Sigma_{\F_i}),\underline{r}_i)}^{\rm sub})
\]
is the $\sigma$-linear isomorphism in Theorem \ref{T:Harris}
-(1).
The assertion then follows at once.
\end{proof}

For $I = \Sigma_\E$, the period $\Omega^{\Sigma_\E}(\itPi)$ can be expressed in terms of the Petersson pairing of holomorphic cusp forms. Let $\<\,,\,\> : \itPi_{\rm hol} \times \itPi^\vee_{\rm hol} \rightarrow \C$ be the $G(\A_f)$-equivariant Petersson bilinear pairing defined by
\begin{align}\label{E:Petersson pairing}
\<\varphi_1,\varphi_2\> = \int_{Z_G(\A)G(\Q)\backslash G(\A)}\varphi_1^{{\Sigma_\E}}(g)\varphi_2(g)\,dg^{\rm Tam}.
\end{align}
Here $dg^{\rm Tam}$ is the Tamagawa measure on $Z_G(\A)\backslash G(\A)$.

\begin{lemma}\label{L:Petersson norm}
We have
\[
\sigma \left( \frac{\<\varphi_1,\varphi_2\>}{\Omega^{\Sigma_\E}(\itPi)}\right) = \frac{\<{}^\sigma\!\varphi_1,{}^\sigma\!\varphi_2\>}{\Omega^{\Sigma_\E}({}^\sigma\!\itPi)}
\]
for $\varphi_1 \in \itPi_{\rm hol}$, $\varphi_2 \in \itPi^\vee_{\rm hol}$, and $\sigma \in {\rm Aut}(\C)$.
\end{lemma}

\begin{proof}
We have the morphism for complexes
\begin{align*}
C_{{\rm rda},(\underline{2}-\underline{\kappa},\underline{r})}^{[\E:\Q]} \times C_{{\rm sia},(\underline{\kappa},-\underline{r})}^0 &\longrightarrow C_{{\rm sia},(\underline{2},\underline{0})}^{[\E:\Q]}\\
\left(\varphi_1 \otimes \bigwedge_{w \in \Sigma_\E}X_{+,w} \otimes {\bf v}_{(\underline{2}-\underline{\kappa},\underline{r})},\, \varphi_2 \otimes {\bf v}_{(\underline{\kappa},-\underline{r})}\right) & \longmapsto \varphi_1\varphi_2 \otimes \bigwedge_{w \in \Sigma_\E}X_{+,w}  \otimes {\bf v}_{(\underline{2},\underline{0})}.
\end{align*}
This induces $G(\A_f)$-module homomorphism of cohomology groups
\[
H^{[\E:\Q]}(\mathcal{L}_{(\underline{2}-\underline{\kappa},\underline{r})}^{\rm sub}) \times H^0(\mathcal{L}_{(\underline{\kappa},-\underline{r})}^{\rm can}) \longrightarrow H^{[\E:\Q]}(\mathcal{L}_{(\underline{2},\underline{0})}^{\rm sub}),\quad (c_1, c_2) \longmapsto c_1 \wedge c_2.
\]
The homomorphism satisfies the Galois equivariant property
\[
T_\sigma ( c_1 \wedge c_2 ) = T_\sigma c_1 \wedge T_\sigma c_2
\]
for all $\sigma \in {\rm Aut}(\C)$.
Composing with the trace map in Lemma \ref{L:trace}, we obtain the $G(\A_f)$-equivariant homomorphism
\[
\itPi_{\rm hol} \times \itPi^\vee_{\rm hol} \longrightarrow \C,\quad (\varphi_1,\varphi_2) \longmapsto \int_{{\rm Sh}(G,X)} \xi_{\Sigma_\E}(\varphi_1) \wedge  \xi_{\emptyset}(\varphi_2)
\]
which satisfies 
\[
\sigma \left(\int_{{\rm Sh}(G,X)} \frac{\xi_{\Sigma_\E}(\varphi_1)}{\Omega^{\Sigma_\E}(\itPi)} \wedge \frac{\xi_{\emptyset}(\varphi_2)}{\Omega^{\emptyset}(\itPi^\vee)}\right) = \int_{{\rm Sh}(G,X)} \frac{ \xi_{\Sigma_\E}({}^\sigma\!\varphi_1)}{\Omega^{\Sigma_\E}({}^\sigma\!\itPi)} \wedge \frac{\xi_{\emptyset}({}^\sigma\!\varphi_2)}{\Omega^{\emptyset}({}^\sigma\!\itPi^\vee)}
\]
for all $\sigma \in {\rm Aut}(\C)$. Note that we may take $\Omega^\emptyset(\itPi^\vee)=1$ by definition. Since the class $\xi_{\Sigma_\E}(\varphi_1) \wedge \xi_{\emptyset}(\varphi_2)$ in $H^{[\E:\Q]}(\mathcal{L}_{(\underline{2},\underline{0})}^{\rm sub})$ is represented by $\varphi_1^{\Sigma_\E}\cdot\varphi_2 \otimes \bigwedge_{w \in \Sigma_\E}X_{+,w} \otimes {\bf v}_{(\underline{2},\underline{0})}$, by Lemma \ref{L:trace} we have
\[
\int_{{\rm Sh}(G,X)}\xi_{\Sigma_\E}(\varphi_1) \wedge \xi_{\emptyset}(\varphi_2) = [\A_\E^\times : \E^\times \E_\infty^\times \hat{\o}_\E^\times] \cdot \<\varphi_1,\varphi_2\>.
\]
This completes the proof.
\end{proof}

Let 
\[
L(s,\itPi,{\rm Ad}) = \prod_v L(s,\pi_v,{\rm Ad})
\]
be the adjoint $L$-function of $\itPi$, where ${\rm Ad}$ is the adjoint representation of ${}^LG$ on $\frak{pgl}_2(\C)^{[\E:\Q]}$. Note that $L(s,\itPi,{\rm Ad})$ is holomorphic and non-zero at $s=1$. Combining with the results of Shimura and Takase, we obtain the following corollary.

\begin{corollary}\label{C:Petersson norm}
We have
\[
\sigma \left( \frac{L(1,\itPi,{\rm Ad})}{(2\pi\sqrt{-1})^{-\sum_{w \in \Sigma_\E} \kappa_w}\cdot\pi^{[\E:\Q]}\cdot\Omega^{\Sigma_\E}(\itPi)}\right) = \frac{L(1,{}^\sigma\!\itPi,{\rm Ad})}{(2\pi\sqrt{-1})^{-\sum_{w \in \Sigma_\E} \kappa_w}\cdot\pi^{[\E:\Q]}\cdot\Omega^{\Sigma_\E}({}^\sigma\!\itPi)}
\]
for all $\sigma \in {\rm Aut}(\C)$.
\end{corollary}

\begin{proof}
By Lemma \ref{L:decomposition}, it suffices to consider the case when $\E$ is a field. Let $\itPi'$ be the Jacquet-Langlands transfer of $\itPi$ to $\GL_2(\A_\E)$. 
By a variant of the result \cite[Theorem 12.3]{HK1991} (see also \cite[Theorem 3.8]{Shimura1981}), we have
\begin{align}\label{E:2.12}
\sigma \left( \frac{\<\varphi_1,\varphi_2\>}{\<\varphi_3,\varphi_4\>}\right) =  \frac{\<{}^\sigma\!\varphi_1,{}^\sigma\!\varphi_2\>}{\<{}^\sigma\!\varphi_3,{}^\sigma\!\varphi_4\>}
\end{align}
for $\varphi_1 \in \itPi_{\rm hol}$, $\varphi_2 \in \itPi_{\rm hol}^\vee$, $\varphi_3 \in \itPi_{\rm hol}'$, $\varphi_4 \in (\itPi')_{\rm hol}^\vee$ with $\<\varphi_3,\varphi_4\> \neq 0$ and $\sigma \in {\rm Aut}(\C)$.
We remark that although the assertion in \cite[Theorem 12.3]{HK1991} is stated for $\E=\Q$, one can prove (\ref{E:2.12}) for general totally real number fields following the same argument in \cite[\S\,15.2]{HK1991}.
On the other hand, it follows from the result of Takase \cite[Proposition 1]{Takase1986} that
\[
\sigma \left( \frac{L(1,\itPi,{\rm Ad})}{(2\pi\sqrt{-1})^{-\sum_{w \in \Sigma_\E} \kappa_w}\cdot\pi^{[\E:\Q]}\cdot \<\varphi_1,\varphi_2\>}\right) = \frac{L(1,{}^\sigma\!\itPi,{\rm Ad})}{(2\pi\sqrt{-1})^{-\sum_{w \in \Sigma_\E} \kappa_w}\cdot\pi^{[\E:\Q]}\cdot\<{}^\sigma\!\varphi_1,{}^\sigma\!\varphi_2\>}
\]
for $\varphi_1 \in \itPi_{\rm hol}'$, $\varphi_2 \in (\itPi')^{\vee}_{\rm hol}$ with $\<\varphi_1,\varphi_2\> \neq 0$, and $\sigma \in {\rm Aut}(\C)$. We remark that the factor $(2\pi\sqrt{-1})^{-\sum_{w \in \Sigma_\E} \kappa_w}$ is obtained from the comparison between rational structures on $\itPi'_{\rm hol}$ given by the zeroth coherent cohomology and by the Whittaker model (cf.\,\cite[(A.4.6)]{GH1993} or Lemma \ref{L:Galois equiv. cusp form} below).
The assertion then follows immediately from Lemma \ref{L:Petersson norm}.

\end{proof}

\section{Trilinear differential operators}\label{S:3}

Let $\E$ be a totally real \'etale cubic algebra over a totally real number field $\F$.
Let $D$ be a totally indefinite quaternion algebra over $\F$. Let
\[
G' = {\rm R}_{\F/\Q}D^\times, \quad G = {\rm R}_{\E/\Q}(D\otimes_\F \E)^\times
\] 
be connected reductive linear algebraic groups over $\Q$. We identify $G'(\R)$ and $G(\R)$ with $\GL_2(\R)^{\Sigma_\F}$ and $\GL_3(\R)^{\Sigma_\E}$ via the identifications $\F_\infty = \R^{\Sigma_\F}$ and $\E_\infty=\R^{\Sigma_\E}$, respectively. 
Let $X'$ (resp.\,$X$) be the $G'(\R)$-conjugacy class (resp.\,$G(\R)$-conjugacy class) containing $h'$ (resp.\,$h$) defined as in (\ref{E:Shimura datum}). 
The natural diagonal embedding $\F \rightarrow \E$ induces the natural injective morphism $G' \rightarrow G$, which defines the inclusion of Shimura data
\[
(G',X') \subset (G,X).
\]

We begin with a well-known lemma.
\begin{lemma}\label{L:differential operator}
Let $(\ell_1,\ell_2,\ell_3) \in \Z_{\geq 1}^3$ such that $\ell_1 \geq \ell_2+\ell_3$ and $\ell_1 +\ell_2+\ell_3\equiv 0\,({\rm mod}\,2)$.
Let ${\bf v}_{\ell_2}$ and ${\bf v}_{\ell_3}$ be non-zero vectors in $D(\ell_2)^+$ and $D(\ell_3)^+$ of weights $\ell_2$ and $\ell_3$, respectively.
Then there exists a non-zero element in $U(\frak{gl}_{2,\C}^2)$ of the form
\[
\sum_{2m_1+2m_2 = \ell_1-\ell_2-\ell_3}c_{m_1,m_2}(\ell_1,\ell_2,\ell_3)(X_+^{m_1} \otimes X_+^{m_2}),
\]
unique up to scalars, such that 
\begin{align}\label{E:linear equ.2}
\begin{split}
&H\cdot\sum_{2m_1+2m_2 = \ell_1-\ell_2-\ell_3}c_{m_1,m_2}(\ell_1,\ell_2,\ell_3)(X_+^{m_1}{\bf v}_{\ell_2} \otimes X_+^{m_2}{\bf v}_{\ell_3})\\
& = \ell_1\cdot\sum_{2m_1+2m_2 = \ell_1-\ell_2-\ell_3}c_{m_1,m_2}(\ell_1,\ell_2,\ell_3)(X_+^{m_1}{\bf v}_{\ell_2} \otimes X_+^{m_2}{\bf v}_{\ell_3}),
\end{split}
\end{align}
and
\begin{align}\label{E:linear equ.}
&X_-\cdot  \sum_{2m_1+2m_2 = \ell_1-\ell_2-\ell_3}c_{m_1,m_2}(\ell_1,\ell_2,\ell_3)(X_+^{m_1}{\bf v}_{\ell_2} \otimes X_+^{m_2}{\bf v}_{\ell_3}) = 0.
\end{align}
Here the action of $\frak{gl}_{2,\C}$ on $D(\ell_2)^+\otimes D(\ell_3)^+$ is given by
\[
X\cdot ({\bf v} \otimes {\bf v}') = X\cdot {\bf v} \otimes {\bf v}' + {\bf v} \otimes X \cdot {\bf v}'.
\]
Moreover, the coefficients can be normalized so that $c_{m_1,m_2}(\ell_1,\ell_2,\ell_3) \in \Q$.
\end{lemma}

\begin{proof}
The existence and uniqueness follow directly from the decomposition (cf.\,\cite{Repka1976})
\[
D(\ell_2)^+\otimes D(\ell_3)^+ \vert_{\frak{gl}_{2,\C}} = \bigoplus_{j=0}^\infty D(\ell_2+\ell_3+2j)^+.
\]
For the rationality of the differential operator, note that
\[
H \cdot {\bf v}_{\ell_2} = \ell_2 \cdot {\bf v}_{\ell_2},\quad H \cdot {\bf v}_{\ell_3} = \ell_3 \cdot {\bf v}_{\ell_3},\quad X_+X_--X_-X_+ = -4H.
\]
Therefore, by a simple induction argument, we see that the linear equations defined by (\ref{E:linear equ.}) have coefficients in $\Q$.
In particular, the solutions can be chosen to be rational numbers.
\end{proof}
For a triplet $(\ell_1,\ell_2,\ell_3) \in \Z_{\geq 1}^3$ such that $\ell_1 +\ell_2+\ell_3\equiv 0\,({\rm mod}\,2)$ and satisfying the unbalanced condition 
\[
2\max\{\ell_1,\ell_2,\ell_3\} \geq \ell_1+\ell_2+\ell_3,
\]
we fix a choice of $c_{m_1,m_2}(\ell_1,\ell_2,\ell_3) \in \Q$ for each pair of non-negative integers $m_1,m_2$ with $2m_1+2m_2 = 2\max\{\ell_1,\ell_2,\ell_3\} - \ell_1+\ell_2+\ell_3$ such that the similar assertion in Lemma \ref{L:differential operator} holds.
Define ${\bf X}(\ell_1,\ell_2,\ell_3) \in U(\frak{gl}_{2,\C}^3)$ by
\begin{align*}
{\bf X}(\ell_1,\ell_2,\ell_3)  = \begin{cases}
\sum_{2m_1+2m_2 = \ell_1-\ell_2-\ell_3}c_{m_1,m_2}(\ell_1,\ell_2,\ell_3)(1 \otimes X_+^{m_1} \otimes X_+^{m_2}) & \mbox{ if $\ell_1 \geq \ell_2+\ell_3$},\\
\sum_{2m_1+2m_2 = \ell_2-\ell_1-\ell_3}c_{m_1,m_2}(\ell_1,\ell_2,\ell_3)(X_+^{m_1} \otimes 1 \otimes X_+^{m_2}) & \mbox{ if $\ell_2 \geq \ell_1+\ell_3$},\\
\sum_{2m_1+2m_2 = \ell_3-\ell_1-\ell_2}c_{m_1,m_2}(\ell_1,\ell_2,\ell_3)(X_+^{m_1} \otimes X_+^{m_2} \otimes 1) & \mbox{ if $\ell_3 \geq \ell_1+\ell_2$}.
\end{cases}
\end{align*}

Let $(\underline{\kappa},\underline{r}) \in \Z_{\geq 1}[\Sigma_\E]\times \Z[\Sigma_\E]$ be motivic. 
We assume $\underline{\kappa}$ satisfies the totally unbalanced condition
\begin{align}\label{E:totally unbalanced}
2\max_{w \mid v}\{\kappa_w\} - \sum_{w \mid v}\kappa_w \geq 0 \quad
\end{align}
for all $v \in \Sigma_\F$ and
\[
\sum_{w \in \Sigma_\E} r_w=0.
\]
For each $v \in \Sigma_\F$, let $v^{(1)}, v^{(2)}, v^{(3)} \in \Sigma_\E$ be the extensions of $v$ and $ \tilde{v}(\underline{\kappa})\in \Sigma_\E$ the homomorphism such that $\max_{w \mid v}\{\kappa_w\} = \kappa_{\tilde{v}(\underline{\kappa})}$. Put
\[
I_{\underline{\kappa}} = \{ \tilde{v}(\underline{\kappa}) \, \vert \, v \in \Sigma_\F\} \subset \Sigma_\E.
\]
Let $\mathcal{L}'_{(\underline{2},\underline{0})}$ and $\mathcal{L}_{(\underline{\kappa}(I_{\underline{\kappa}}),\underline{r})}$ be the automorphic line bundles on ${\rm Sh}(G',X')$ and ${\rm Sh}(G,X)$ defined by the motivic algebraic representations $\rho'_{(\underline{2},\underline{0})}$ of $K_\infty'$ and $\rho_{(\underline{\kappa}(I_{\underline{\kappa}}),\underline{r})}$ of $K_\infty$ in (\ref{E:algebraic rep.}), respectively.
Here
\[
(\underline{2},\underline{0}) = ((2,\cdots,2),(0,\cdots,0)) \in \Z[\Sigma_\F] \times \Z[\Sigma_\F].
\]

\begin{prop}\label{P:differential operator}
There exists a homogeneous $\Q(\underline{\kappa})$-rational differential operator $[\delta(\underline{\kappa})]$ from $\mathcal{L}_{(\underline{\kappa}(I_{\underline{\kappa}}),\underline{r})}$ to $\mathcal{L}'_{(\underline{2},\underline{0})}$ satisfying the following conditions:
\begin{itemize}
\item[(1)] Let 
\[
[\delta(\underline{\kappa})] : H^{[\F:\Q]}(\mathcal{L}_{(\underline{\kappa}(I_{\underline{\kappa}}),\underline{r})}^{\rm sub}) \longrightarrow H^{[\F:\Q]}((\mathcal{L}_{(\underline{2},\underline{0})}')^{\rm sub})
\]
be the induced $G'(\A_f)$-module homomorphism. Then we have
\[
T_\sigma \circ [\delta(\underline{\kappa})] = [\delta({}^\sigma\!\underline{\kappa})] \circ T_\sigma
\]
for all $\sigma \in {\rm Aut}(\C)$.
\item[(2)] If a class in $H^{[\F:\Q]}(\mathcal{L}_{(\underline{\kappa}(I_{\underline{\kappa}}),\underline{r})}^{\rm sub})$ is represented by $\varphi \otimes \bigwedge_{w \in I_{\underline{\kappa}}}X_{+,w} \otimes {\bf v}_{(\underline{\kappa}(I_{\underline{\kappa}}),\underline{r})}$, then its image under $[\delta(\underline{\kappa})]$ in $H^{[\F:\Q]}((\mathcal{L}_{(\underline{2},\underline{0})}')^{\rm sub})$ is represented by 
\[
({\bf X}(\underline{\kappa})\cdot\varphi) \vert_{G'(\A)} \otimes \bigwedge_{v \in \Sigma_\F}X_{+,v} \otimes {\bf v}_{(\underline{2},\underline{0})}.
\]
Here ${\bf X}(\underline{\kappa})\in U(\frak{gl}_{2,\C}^{\Sigma_\E})$ is defined by
\begin{align}\label{E:differential operator}
{\bf X}(\underline{\kappa}) = \bigotimes_{v \in \Sigma_\F}{\bf X}(\kappa_{v^{(1)}},\kappa_{v^{(2)}},\kappa_{v^{(3)}}).
\end{align}
\end{itemize}
\end{prop}

\begin{proof}
Recall $\frak{g}$ and $\frak{k}$ (resp.\,$\frak{g}'$ and $\frak{k}'$) are the Lie algebras of $G(\R)$ and $K_\infty = Z_G(\R)\cdot{\rm SO}(2)^{\Sigma_\E}$ (resp.\,$G'(\R)$ and $K_\infty' = Z_{G'}(\R)\cdot{\rm SO}(2)^{\Sigma_\F}$), respectively, and we have the Hodge decompositions for $\frak{g}_\C$ and $\frak{g}'_\C$ as in (\ref{E:Hodge}).
We identify $\frak{g}$ and $\frak{g}'$ with $\frak{gl}_2^{\Sigma_\E}$ and $\frak{gl}_2^{\Sigma_\F}$, respectively.
For each $w \in \Sigma_\E$ (resp.\,$v \in \Sigma_\F$), let $\frak{g}_w$ (resp.\,$\frak{g}'_v$) be the $w$-component of $\frak{g}$ (resp.\,$v$-component of $\frak{g}'$).
Let
\[
\frak{P} = \frak{p}^- \oplus \frak{k}_\C,\quad  \frak{P}' = (\frak{p}')^- \oplus \frak{k}_\C'.
\]
By \cite[Theorem 4.8]{Harris1985} and \cite[Lemma 7.2]{Harris1986}, each element
\[
\delta^* \in {\rm Hom}_{U(\frak{P}')}(\C_{(\underline{2},\underline{0})}^\vee,U(\frak{g}_\C)\otimes_{U(\frak{P})} \C^\vee_{(\underline{\kappa}(I_{\underline{\kappa}}),\underline{r})}) = {\rm Hom}_{U(\frak{P}')}(\C_{(-\underline{2},\underline{0})},U(\frak{g}_\C)\otimes_{U(\frak{P})} \C_{(-\underline{\kappa}(I_{\underline{\kappa}}),-\underline{r})})
\]
gives rise to a homogeneous differential operator $[\delta]$ from $\mathcal{L}_{(\underline{\kappa}(I_{\underline{\kappa}}),\underline{r})}$ to $\mathcal{L}'_{(\underline{2},\underline{0})}$.
Here the action of $\frak{P}$ and $\frak{P}'$ on $\C_{(-\underline{\kappa}(I_{\underline{\kappa}}),-\underline{r})}$ and $\C_{(-\underline{2},\underline{0})}$ factor through $\frak{k}_\C$ and $\frak{k}_\C'$, respectively.
We define $[\delta(\underline{\kappa})]$ be the homogeneous differential operator corresponds to the element
\[
\delta(\underline{\kappa})^* \in {\rm Hom}_{U(\frak{P}')}(\C_{(-\underline{2},\underline{0})},U(\frak{g}_\C)\otimes_{U(\frak{P})} \C_{(-\underline{\kappa}(I_{\underline{\kappa}}),-\underline{r})})
\]
defined by
\[
\delta(\underline{\kappa})^*({\bf v}_{(-\underline{2},\underline{0})}) = {\bf X}(\underline{\kappa})\cdot(1\otimes {\bf v}_{(-\underline{\kappa}(I_{\underline{\kappa}}),-\underline{r})}),
\]
where ${\bf X}(\underline{\kappa}) \in U(\frak{g}_\C)=U(\frak{gl}_{2,\C}^{\Sigma_\E})$ is in (\ref{E:differential operator}).
We shall show that $\delta(\underline{\kappa})^*$ is indeed $U(\frak{P}')$-equivariant.
For $(\ell,t) \in \Z \times \Z$ with $\ell \equiv t \,({\rm mod}\,2)$, let $\C_{(\ell,t)}$ be the complex field equipped with the action of $\R^\times \cdot {\rm SO}(2)$ given by $ak_\theta \cdot z = a^{-t}e^{-\sqrt{-1}\,\ell\theta}\cdot z$ for $a \in \R^\times$ and $k_\theta \in {\rm SO}(2)$.
Thus we have
\[
\C_{(-\underline{\kappa}(I_{\underline{\kappa}}),-\underline{r})} = \bigotimes_{w \in \Sigma_\E}\C_{(-{\kappa}(I_{\underline{\kappa}})_w,-r_w)}
\]
as algebraic characters of $K_\infty = (\R^\times \cdot {\rm SO}(2))^{\Sigma_\E}$.
In the above isomorphism, we fix ${\bf v}_w \in \C_{(-{\kappa}(I_{\underline{\kappa}})_w,-r_w)}$ for each $w \in \Sigma_\E$ such that 
\[
{\bf v}_{(-\underline{\kappa}(I_{\underline{\kappa}}),-\underline{r})} = \bigotimes_{w \in \Sigma_\E}{\bf v}_w.
\]
Note that we also have
\[
U(\frak{g}_\C)\otimes_{U(\frak{P})} \C_{(-\underline{\kappa}(I_{\underline{\kappa}}),-\underline{r})} = \bigotimes_{w \in \Sigma_\E} U(\frak{g}_{\C,w})\otimes_{U(\frak{P}_w)} \C_{(-\kappa(I_{\underline{\kappa}})_w,-r_w)}.
\]
We write $\mathbb{D}_w = U(\frak{g}_{\C,w})\otimes_{U(\frak{P}_w)} \C_{(-\kappa(I_{\underline{\kappa}})_w,-r_w)}$ for each $w \in \Sigma_\E$.
Let $v \in \Sigma_\F$. The action of $\frak{g}'_v$ on $\mathbb{D}_{v^{(1)}} \otimes \mathbb{D}_{v^{(2)}} \otimes \mathbb{D}_{v^{(3)}}$ is given by
\[
X\cdot ({\bf v}_1 \otimes {\bf v}_2 \otimes {\bf v}_3) = X\cdot{\bf v}_1 \otimes {\bf v}_2 \otimes {\bf v}_3+{\bf v}_1 \otimes X\cdot{\bf v}_2 \otimes {\bf v}_3+{\bf v}_1 \otimes {\bf v}_2 \otimes X\cdot{\bf v}_3
\]
for $X \in \frak{g}'_v$ and ${\bf v}_i \in \mathbb{D}_{v^{(i)}}$ for $i=1,2,3$.
Then we have
\begin{align}\label{E:3.5}
\begin{split}
Z\cdot \left({\bf X}(\kappa_{v^{(1)}},\kappa_{v^{(2)}},\kappa_{v^{(3)}})\cdot \bigotimes_{i=1}^3(1 \otimes {\bf v}_{v^{(i)}})\right) &=0,\\
H \cdot \left({\bf X}(\kappa_{v^{(1)}},\kappa_{v^{(2)}},\kappa_{v^{(3)}})\cdot \bigotimes_{i=1}^3(1 \otimes {\bf v}_{v^{(i)}})\right) &= 2\cdot \left({\bf X}(\kappa_{v^{(1)}},\kappa_{v^{(2)}},\kappa_{v^{(3)}})\cdot \bigotimes_{i=1}^3(1 \otimes {\bf v}_{v^{(i)}})\right),\\
X_{-}\cdot \left({\bf X}(\kappa_{v^{(1)}},\kappa_{v^{(2)}},\kappa_{v^{(3)}})\cdot \bigotimes_{i=1}^3(1 \otimes {\bf v}_{v^{(i)}})\right) &=0.
\end{split}
\end{align}
Indeed, suppose $v^{(1)} = \tilde{v}(\underline{\kappa})$, then $\kappa(I_{\underline{\kappa}})_{v^{(1)}} = 2 -\kappa_{v^{(1)}}$, $\kappa(I_{\underline{\kappa}})_{v^{(2)}} = \kappa_{v^{(2)}}$, and $\kappa(I_{\underline{\kappa}})_{v^{(3)}} = \kappa_{v^{(3)}}$. 
Thus we have
\[
\mathbb{D}_{v^{(2)}} = D(\kappa_{v^{(2)}})^+\otimes |\mbox{ }|^{r_{v^{(2)}}/2},\quad \mathbb{D}_{v^{(3)}} = D(\kappa_{v^{(3)}})^+\otimes |\mbox{ }|^{r_{v^{(3)}}/2}.
\]
Also note that 
\[
H\cdot (1 \otimes {\bf v}_{v^{(1)}}) = (2-\kappa_{v^{(1)}})\cdot (1 \otimes {\bf v}_{v^{(1)}}),\quad  X_-\cdot(1 \otimes {\bf v}_{v^{(1)}}) = 0
\]
by definition.
Thus (\ref{E:3.5}) follows from the condition $r_{v^{(1)}} + r_{v^{(2)}} + r_{v^{(3)}} =0$, (\ref{E:linear equ.2}), and (\ref{E:linear equ.}).
Since
\[
{\bf X}(\underline{\kappa})\cdot(1\otimes {\bf v}_{(-\underline{\kappa}(I_{\underline{\kappa}}),-\underline{r})}) = \bigotimes_{v \in \Sigma_\F}\left({\bf X}(\kappa_{v^{(1)}},\kappa_{v^{(2)}},\kappa_{v^{(3)}})\cdot \bigotimes_{i=1}^3(1 \otimes {\bf v}_{v^{(i)}})\right),
\]
we deduce from (\ref{E:3.5}) that $\delta(\underline{\kappa})^*$ is $U(\frak{P}')$-equivariant.
Moreover, it is clear that the diagram 
\[
\begin{tikzcd}
\C_{(-\underline{2},\underline{0})} \arrow[d] \arrow[r, "\delta(\underline{\kappa})^*"]
    &  U(\frak{g}_\C)\otimes_{U(\frak{P})} \C_{(-\underline{\kappa}(I_{\underline{\kappa}}),-\underline{r})}\arrow[d] \\
  \C_{(-\underline{2},\underline{0})} \arrow[r, "\delta({}^\sigma\!\underline{\kappa})^*"]
& U(\frak{g}_\C)\otimes_{U(\frak{P})} \C_{(-{}^\sigma\!\underline{\kappa}(I_{{}^\sigma\!\underline{\kappa}}),-\underline{r})}
\end{tikzcd}
\]
is commutative for all $\sigma \in {\rm Aut}(\C)$, where the vertical homomorphisms are induced by the $\sigma$-linear isomorphisms in (\ref{E:rational structures}).
The assertions (1) and (2) then follow from the construction of $\delta(\underline{\kappa})^*$.
This completes the proof.
\end{proof}

\section{Algebraicity of Rankin--Selberg $L$-functions for $\GL_2 \times \GL_2$}\label{S:4}

The aim of this section is to prove Theorem \ref{T:RS} on the algebraicity of critical values of Rankin--Selberg $L$-functions, from which Theorem \ref{T: main thm 2} follows immediately as we will show in \S\,\ref{SS:5.5} below.

\subsection{$q$-expansion principle}\label{SS:4.1}
For an automorphic form $\varphi$ on $\GL_2(\A_\F)$ and a non-trivial additive character $\psi$ of $\F\backslash \A_\F$, let $W_{\varphi,\psi}$ be the Whittaker function of $\varphi$ with respect to $\psi$ defined by
\[
W_{\varphi,\psi}(g) = \int_{\F \backslash \A_\F}\varphi({\bf n}(x)g)\overline{\psi(x)}\,dx^{\rm Tam}.
\]
Here $dx^{\rm Tam}$ is the Tamagawa measure on $\A_\F$.
Let $(\underline{\kappa},\underline{r}) \in \Z_{\geq 1}[\Sigma_\F] \times \Z[\Sigma_\F]$ be motivic, that is, $\underline{r}=(r,\cdots,r)$ for some $r \in \Z$ and $\kappa_v\equiv r \,({\rm mod}\,2)$ for $v \in \Sigma_\F$. Consider the automorphic line bundle $\mathcal{L}_{(\underline{\kappa},r)}$ on the Shimura variety ${\rm Sh}({\rm R}_{\F/\Q}\GL_{2/\F},(\frak{H}^\pm)^{\Sigma_\F})$ and the zeroth coherent cohomology group $H^0(\mathcal{L}_{(\underline{\kappa},\underline{r})}^{\rm can})$. 
Let $\varphi \otimes {\bf v}_{(\underline{\kappa},\underline{r})} \in H^0(\mathcal{L}_{(\underline{\kappa},\underline{r})}^{\rm can})$. For $\sigma \in {\rm Aut}(\C)$, let ${}^\sigma\!\varphi$ be the automorphic form on $\GL_2(\A_\F)$ defined so that 
\begin{align}\label{E:Galois action on forms}
T_\sigma(\varphi \otimes {\bf v}_{(\underline{\kappa},\underline{r})}) = {}^\sigma\!\varphi \otimes {\bf v}_{({}^\sigma\!\underline{\kappa},\underline{r})}.
\end{align}
Note that the notation is compatible with (\ref{E:sigma iso.}). Let ${\bf f}_\varphi : \frak{H}^{\Sigma_\F} \times \GL_2(\A_{\F,f}) \rightarrow \C$ defined by
\[
{\bf f}_\varphi(\tau,g_f) = \prod_{v \in \Sigma_\F}y_v^{-(\kappa_v+r)/2}\cdot\varphi(({\bf n}(x_v){\bf a}(y_v))_{v \in \Sigma_\F}\cdot g_f)
\]
for $\tau = (x_v+\sqrt{-1}\,y_v)_{v \in \Sigma_\F} \in \frak{H}^{\Sigma_\F}$ and $g_f \in \GL_2(\A_{\F,f})$.
Since $\varphi$ is a holomorphic, we have the Fourier expansion
\[
{\bf f}_\varphi(\tau,g_f) = \sum_{\alpha \in \F} W_\alpha(\varphi,g_f)e^{2\pi\sqrt{-1}\,\sum_{v \in \Sigma_\F}v(\alpha)\tau_v}.
\]
By the results of Harris \cite[Theorem 6.4]{Harris1986} and \cite[(1.1.13)]{BHR1994} (see also \cite[(A.4.6) and (A.4.7)]{GH1993}), we have the following theorem on the algebraicity of the Fourier coefficients. 

\begin{thm}[$q$-expansion principle]\label{T:q-expansion}
Let $\varphi \otimes {\bf v}_{(\underline{\kappa},\underline{r})} \in H^0(\mathcal{L}_{(\underline{\kappa},\underline{r})}^{\rm can})$.
For $\sigma \in {\rm Aut}(\C)$, we have
\begin{align*}
\sigma\left( (2\pi\sqrt{-1})^{-\sum_{v \in \Sigma_\F}(\kappa_v+r)/2} \cdot W_\alpha (\varphi,{\bf a}(u)^{-1}g_f)\right)= (2\pi\sqrt{-1})^{-\sum_{v \in \Sigma_\F}(\kappa_v+r)/2} \cdot W_\alpha ({}^\sigma\!\varphi,g_f)
\end{align*}
for $\alpha \in \F$ and $\sigma \in {\rm Aut}(\C)$.
Here $u \in \widehat{\Z}^\times$ is the unique element such that $\sigma(\psi_{\F}(x)) = \psi_{\F}(ux)$ for $x \in \A_{\F,f}$.
\end{thm}

Let $\psi = \bigotimes_v \psi_v$ be a non-trivial additive character of $\F \backslash \A_\F$. Suppose $\varphi$ is cuspidal. There exists a unique Whittaker function
$W_{\varphi,\psi}^{(\infty)}$ on $\GL_2(\A_{\F,f})$ with respect to $\psi_f = \bigotimes_{v \nmid \infty}\psi_v$ such that
\[
W_{\varphi,\psi} = \prod_{v \in \Sigma_\F}W_{(\kappa_v,r),\psi_v}^+\cdot W_{\varphi,\psi}^{(\infty)}.
\]
Here the archimedean Whittaker function $W_{(\kappa_v,r),\psi_v}^+$ is defined in \S\,\ref{SS:notation}.

\begin{lemma}\label{L:Galois equiv. cusp form}
Let $\varphi \otimes {\bf v}_{(\underline{\kappa},\underline{r})} \in H^0(\mathcal{L}_{(\underline{\kappa},\underline{r})}^{\rm can})$. Suppose $\varphi$ is cuspidal. 
For $\sigma \in {\rm Aut}(\C)$, we have
\begin{align*}
\sigma\left( (2\pi\sqrt{-1})^{-\sum_{v \in \Sigma_\F}(\kappa_v+r)/2} \cdot W_{\varphi,\psi}^{(\infty)} ({\bf a}(u)^{-1}g_f)\right)= (2\pi\sqrt{-1})^{-\sum_{v \in \Sigma_\F}(\kappa_v+r)/2} \cdot W^{(\infty)}_{{}^\sigma\!\varphi,\psi} (g_f)
\end{align*}
for $g_f \in \GL_2(\A_{\F,f})$.
Here $u \in \widehat{\Z}^\times$ is the unique element such that $\sigma(\psi_{\F}(x)) = \psi_{\F}(ux)$ for $x \in \A_{\F,f}$.
\end{lemma}

\begin{proof}
We have the Fourier expansion
\[
\varphi(g) = \sum_{\alpha \in \F^\times} W_{\varphi,\psi^\alpha}(g) = \sum_{\alpha \in \F^\times} W_{\varphi,\psi}({\bf a}(\alpha)g)
\]
for $g \in \GL_2(\A_\F)$.
Comparing the Fourier expansions of ${\bf f}_\varphi$ and $\varphi$, by (\ref{E:archimedean Whittaker}), we see that
\[
W_\alpha(\varphi,g_f) = \begin{cases}
\prod_{v \in \Sigma_\F}v(\alpha)^{(\kappa_v+r)/2}\cdot W_{\varphi,\psi}^{(\infty)}({\bf a}(\alpha\beta^{-1})g_f) & \mbox{ if $\alpha \gg 0$},\\
0 & \mbox{ otherwise}.
\end{cases}
\]
Here $\beta \in \F^\times$ is the element such that $\psi = \psi_\F^\beta$. 
Note that
\[
\sigma\left(\prod_{v \in \Sigma_\F}v(\alpha)^{(\kappa_v+r)/2}\right) = \prod_{v \in \Sigma_\F}v(\alpha)^{(\kappa_{\sigma^{-1}\circ v}+r)/2} = \prod_{v \in \Sigma_\F}v(\alpha)^{({}^\sigma\!\kappa_v+r)/2}.
\]
The assertion then follows from Theorem \ref{T:q-expansion}.
\end{proof}

\subsection{Eisenstein series}

Let $\mu_1$ and $\mu_2$ be Hecke characters of $\A_\F^\times$.
Consider the space $I(\mu_1,\mu_2,s)$ consisting of smooth and right $({\rm O}(2)^{\Sigma_\F}\times \GL_2(\hat{\frak{o}}_\F))$-finite functions $f_s : \GL_2(\A_\F) \rightarrow \C$ such that
\[
f_s({\bf n}(x){\bf a}(a){\bf d}(d)g) = \mu_1(a)\mu_2(d)|ad^{-1}|_{\A_\F}^s\cdot f_s(g)
\]
for $x \in \A_F$, $a,d \in \A_\F^\times$, and $g \in \GL_2(\A_\F)$. It is a $((\frak{gl}_{2}^{\Sigma_\F},{\rm O}(2)^{\Sigma_\F}) \times \GL_2(\A_{\F,f}))$-module in a natural way.
For a place $v$ of $\F$, we can define the space $I(\mu_{1,v},\mu_{2,v},s)$ in a similar way. When $v \nmid \infty$ and $\sigma \in {\rm Aut}(\C)$, we have the $\sigma$-linear isomorphism 
\[
I(\mu_{1,v},\mu_{2,v},s) \longrightarrow I({}^\sigma\!\mu_{1,v},{}^\sigma\!\mu_{2,v},s),\quad f_{s,v} \longmapsto {}^\sigma\!f_{s,v}
\]
defined by 
\[
{}^\sigma\!f_{s,v}({\bf n}(x){\bf a}(a){\bf d}(d)k) = {}^\sigma\!\mu_{1,v}(a){}^\sigma\!\mu_{2,v}(d)|ad^{-1}|_{\F_v}^s\cdot \sigma(f_{s,v}(k))
\]
for $x \in \F_v$, $a,d \in \F_v^\times$, and $k \in \GL_2(\frak{o}_{\F_v})$.
For $f_s \in I(\mu_1,\mu_2,s)$, we define the Eisenstein series
\[
E(g;f_s) = \sum_{\gamma \in B(\F)\backslash\GL_2(\F)}f_s(\gamma g).
\]
The above series converges absolutely for ${\rm Re}(s) \gg 0$ and admits meromorphic continuation to $s \in \C$.
Let $\mathcal{S}(\A_\F^2)= \bigotimes_v \mathcal{S}(\F_v^2)$ be the space of Bruhat--Schwartz functions on $\A_\F^2$.
For $\Phi \in \mathcal{S}(\A_\F^2)$, we define the Godement section $f_{\mu_1,\mu_2,\Phi,s} \in I(\mu_1,\mu_2,s)$ by
\[
f_{\mu_1,\mu_2,\Phi,s}(g) = \mu_1(\det(g))|\det(g)|_{\A_\F}^s\int_{\A_\F^\times} \Phi((0,t)g)\mu_1\mu_2^{-1}(t)|t|^{2s}_{\A_\F}\,d^\times t^{\rm std}.
\]
Here $d^\times t^{\rm std} = \prod_v d^\times t_v^{\rm std}$ is the standard measure on $\A_\F^\times$ defined so that $d^\times t_v^{\rm std} = \frac{dt_v}{|t_v|}$ if $v \in \Sigma_\F$ and $dt_v$ is the Lebesgue measure on $\R = \F_v$, and ${\rm vol}(\frak{o}_{\F_v}^\times,d^\times t_v^{\rm std})=1 $ if $v \nmid \infty$.
For a place $v$ of $\F$, we can define $f_{\mu_{1,v},\mu_{2,v},\Phi_v,s}$ for $\Phi_v \in \mathcal{S}(\F_v^2)$ in a similar way. When $v \nmid \infty$ and $\sigma \in {\rm Aut}(\C)$, it is easy to verify that
\begin{align}\label{E:Galois action on sections}
{}^\sigma\!f_{\mu_{1,v},\mu_{2,v},\Phi_v,s} = f_{{}^\sigma\!\mu_{1,v},{}^\sigma\!\mu_{2,v},{}^\sigma\!\Phi_v,s}.
\end{align}

Now we assume that $\mu_1$ and $\mu_2$ are algebriac Hecke characters of $\A_\F^\times$ with $|\mu_1|=|\mbox{ }|_{\A_\F}^{n_1}$ and $|\mu_2|=|\mbox{ }|_{\A_\F}^{n_2}$ for some $n_1, n_2 \in \Z$, and $\mu_1\mu_2$ has parallel signature with 
\[
{\rm sgn}(\mu_1\mu_2) = (-1)^{n_1+n_2}.
\]
For $\kappa \in \Z_{\geq 1}$, let $\Phi^{[\kappa]} \in \mathcal{S}(\R^2)$ defined by
\[
\Phi^{[\kappa]}(x,y) = 2^{-\kappa}(x+\sqrt{-1}\,y)^\kappa e^{-\pi(x^2+y^2)}.
\]
For $\Phi \in S(\A_{\F,f}^2)$ and $\kappa \in \Z_{\geq 1}$ with $\kappa \equiv n_1+n_2\,({\rm mod}\,2)$, we define the Eisenstein series $E^{[\kappa]}(\mu_1,\mu_2,\Phi)$ by
\[
E^{[\kappa]}(\mu_1,\mu_2,\Phi) = E(f_{\mu_1,\mu_2,\otimes_{v \in \Sigma_\F}\Phi^{[\kappa]} \otimes_{v \nmid \infty}\Phi,s})\vert_{s=(\kappa-n_1+n_2)/2}.
\]
We have the following result on the algebraicity of Eisenstein series. 
When $\kappa\geq 3$, the series defining $E^{[\kappa]}(\mu_1,\mu_2,\Phi)$ is absolutely convergent and the lemma below then follows from (\ref{E:Galois action on sections}) and the result \cite[Theorem 3.2.1]{Harris1984} of Harris, which is proved by geometric method. 
Here we prove the algebraicity for any $\kappa \geq 1$ by compute the Fourier coefficients directly and determine the explicit relation among them under Galois conjugation.
\begin{lemma}\label{L:Eisenstein series}
Let $\Phi \in S(\A_{\F,f}^2)$ and $\kappa \in \Z_{\geq 1}$ with $\kappa \equiv n_1+n_2\,({\rm mod}\,2)$. 
Then $E^{[\kappa]}(\mu_1,\mu_2,\Phi)$ is holomorphic of motivic weight $(\underline{\kappa},n_1+n_2) = ((\kappa,\cdots,\kappa),n_1+n_2) \in \Z_{\geq 1}[\Sigma_\F]\times \Z$.
Moreover, for $\sigma \in {\rm Aut}(\C)$, we have
\[
\frac{{}^\sigma\!E^{[\kappa]}(\mu_1,\mu_2,\Phi)}{\sigma\left(D_\F^{1/2}(2\pi\sqrt{-1})^{-[\F:\Q](\kappa+n_1+n_2)/2}\cdot G(\mu_2^{-1})\right)} = \frac{E^{[\kappa]}({}^\sigma\!\mu_1,{}^\sigma\!\mu_2,{}^\sigma\!\Phi)}{D_\F^{1/2}(2\pi\sqrt{-1})^{-[\F:\Q](\kappa+n_1+n_2)/2}\cdot G({}^\sigma\!\mu_2^{-1})}.
\]
Here ${}^\sigma\!E^{[\kappa]}(\mu_1,\mu_2,\Phi)$ is defined in (\ref{E:Galois action on forms}).
\end{lemma}

\begin{proof}
The first assertion is clear. Indeed, for each $v \in \Sigma_\F$ we have
\[
D(\kappa)\otimes |\mbox{ }|^{(n_1+n_2)/2} \subset I(\mu_{1,v},\mu_{2,v},s)\vert_{s=(\kappa-n_1+n_2)/2}
\]
and 
\[
f_{\mu_{1,v},\mu_{2,v},\Phi^{[\kappa]},s}(gk_\theta) = e^{\sqrt{-1}\,\kappa\theta}f_{\mu_{1,v},\mu_{2,v},\Phi^{[\kappa]},s}(g)
\]
for $k_\theta \in {\rm SO}(2)$ and $g \in \GL_2(\F_v) = \GL_2(\R)$. 
To prove the second assertion, we may assume that $\Phi = \bigotimes_{v \nmid \infty}\Phi_v$ is a pure tensor.
We write $\Phi_v = \Phi^{[\kappa]}$ for $v \in \Sigma_\F$.
For any $f_s \in I(\mu_1,\mu_2,s)$, we have the Fourier expansion
\begin{align*}
E(g;f_s) & = \sum_{\alpha \in \F}W_{E(f_s),\psi_\F^\alpha}(g)\\
& = f_s(g) + Mf_s(g) +\sum_{\alpha \in \F^\times}W_{E(f_s),\psi_\F}({\bf a}(a)g),
\end{align*}
where $Mf_s \in I(\mu_2,\mu_1,1-s)$ is defined by the intertwining integral
\[
Mf_s(g) = \int_{\A_\F}f_s\left(\bp 0 & -1 \\ 1 & 0\ep{\bf n}(x)g\right)\,dx^{\rm Tam}.
\]
For $f_s = f_{\mu_1,\mu_2,\otimes_{v }\Phi_v,s}$, we have (cf.\,\cite[(5.1.8)]{Ikeda1989})
\begin{align*}
&Mf_{\mu_1,\mu_2,\otimes_{v }\Phi_v,s}(g)\\
&= L(2s-1,\mu_1\mu_2^{-1})\prod_{v}\varepsilon(2s-1,\mu_{1,v}\mu_{2,v}^{-1},\psi_{\F_v})^{-1}L(2-2s,\mu_{1,v}^{-1}\mu_{2,v})^{-1}f_{\mu_{2,v},\mu_{1,v},\widehat{\Phi}_{v},1-s}(g_v).
\end{align*}
Here $\widehat{\Phi}_v$ is the symplectic Fourier transform of $\Phi_v$ with respect to $\psi_\F$ defined by
\[
\widehat{\Phi}_v(x,y) = \int_{\F_v^2}\Phi_v(u,w)\psi_{\F_v}(uy-wx)\,dudw
\]
and the Haar measures $du$ and $dw$ are the self-dual Haar measures on $\F_v$ with respect to $\psi_{\F_v}$.
Note that the local factors appearing in the above infinite product are holomorphic in $s$ and equal to $1$ for almost all $v$. Since $\kappa \equiv n_1+n_2\,({\rm mod}\,2)$, we have $L(\kappa-n_1+n_2-1,\mu_1\mu_2^{-1})=0$ by the functional equation for $L(s,\mu_1\mu_2^{-1})$. In particular, we have $Mf_{\mu_1,\mu_2,\otimes_{v }\Phi_v,s} \vert_{s = (\kappa-n_1+n_2)/2}=0$. 
For the non-zero Fourier coefficient, we have the following factorization into local Whittaker functions (cf.\,\cite[(5.1.6)]{Ikeda1989}):
\begin{align*}
W_{E(f_s),\psi_\F}(g) &= D_\F^{-1/2}\mu_1(\det(g))|\det(g)|^s_{\A_\F}\int_{\A_\F^\times} \widetilde{\rho(g)\Phi}(t,-t^{-1})\mu_1\mu_2^{-1}(t)|t|_{\A_\F}^{2s-1}\,d^\times t^{\rm std}\\
&= D_\F^{-1/2}\prod_v \mu_{1,v}(\det(g_v))|\det(g_v)|^s_{\F_v}\int_{\F_v^\times} \widetilde{\rho(g_v)\Phi_v}(t_v,-t_v^{-1})\mu_{1,v}\mu_{2,v}^{-1}(t_v)|t_v|_{\F_v}^{2s-1}\,d^\times t_v^{\rm std}\\
&= D_\F^{-1/2} \prod_v W_{\psi_{\F_v}}(g_v;f_{\mu_{1,v},\mu_{2,v},\Phi_v,s}).
\end{align*}
Here $\widetilde{\rho(g_v)\Phi_v}$ is the partial Fourier transform of $\rho(g_v)\Phi_v$ with respect to $\psi_{\F_v}$ defined by
\[
\widetilde{\rho(g_v)\Phi_v}(x,y) = \int_{\F_v}\Phi_v((x,w)g)\psi_{\F_v}(wy)\,dw^{\rm std}
\]
and $dw^{\rm std}$ is the standard measure on $\F_v$ defined so that $dw^{\rm std}$ is the Lebesgue measure if $v \in \Sigma_\F$ and ${\rm vol}(\frak{o}_{\F_v},dw^{\rm std})=1$ if $v \nmid \infty$.
Note that the factor $D_\F^{-1/2}$ is the ratio between the standard and Tamagawa measures on $\A_\F$.
For $v \in \Sigma_\F$, it is easy to show that (cf.\,\cite[Lemma 4.5]{CH2020})
\[
f_{\mu_{1,v},\mu_{2,v},\Phi_v,s}(1) = 2^{-\kappa}(\sqrt{-1})^\kappa\pi^{-(s+(\kappa+n_1-n_2)/2)} \Gamma(s+\tfrac{\kappa+n_1-n_2}{2})
\]
and
\[
W_{\psi_{\F_v}}(f_{\mu_{1,v},\mu_{2,v},\Phi_v,s})\vert_{s=(\kappa-n_1+n_2)/2} = W_{(\kappa,n_1+n_2),\psi_{\F_v}}^+.
\]
We conclude that for $g_f = (g_v)_{v \nmid \infty} \in \GL_2(\A_{\F,f})$ and $\alpha \in \F$, 
\begin{align}\label{E:4.11}
\begin{split}
&W_\alpha(E^{[\kappa]}(\mu_1,\mu_2,\Phi),g_f)\\
& = \begin{cases}
(-2\pi\sqrt{-1})^{-[\F:\Q]\kappa}\Gamma(\kappa)^{[\F:\Q]}\prod_{v \nmid \infty}f_{\mu_{1,v},\mu_{2,v},\Phi_v,s}(g_v) \vert_{s=(\kappa-n_1+n_2)/2} & \mbox{ if $\alpha=0$},\\
D_\F^{-1/2}N_{\F/\Q}(\alpha)^{(\kappa+n_1+n_2)/2} \prod_{v \nmid \infty}W_{\psi_{\F_v}}(g_v;f_{\mu_{1,v},\mu_{2,v},\Phi_v,s})\vert_{s=(\kappa-n_1+n_2)/2} & \mbox{ if $\alpha \gg 0$},\\
0 & \mbox{ otherwise}.
\end{cases}
\end{split}
\end{align}
Let $\sigma \in {\rm Aut}$ and $u = (u_p)_p \in \widehat{\Z}^\times$ the unique element such that $\sigma(\psi_\F(x)) = \psi_\F(ux)$ for $x \in \A_{\F,f}$.
For $v \nmid \infty$ and $v \mid p$ for some rational prime $p$, we have
\begin{align}\label{E:4.12}
{}^\sigma W_{\psi_{\F_v}}({\bf a}(u_p)^{-1}g_v;f_{\mu_{1,v},\mu_{2,v},\Phi_v,s}) = {}^\sigma\!\mu_{2,v}(u_p)^{-1}W_{\psi_{\F_v}}(g_v;f_{{}^\sigma\!\mu_{1,v},{}^\sigma\!\mu_{2,v},{}^\sigma\!\Phi_v,s})
\end{align}
as rational functions in $q_v^{-s}$, where $q_v$ is the cardinality of the residue field of $\F_v$.
Indeed, $\widetilde{\rho(g_v)\Phi_v}(t_v,-t_v^{-1})$ vanishes when $|t_v|_{\F_v}$ is either sufficiently small or sufficiently large, thus $W_{\psi_{\F_v}}(g_v;f_{\mu_{1,v},\mu_{2,v},\Phi_v,s})$ is a finite sum of polynomials in $\C[q_v^s,q_v^{-s}]$. In particular, we have
\begin{align*}
&{}^\sigma W_{\psi_{\F_v}}({\bf a}(u_p)^{-1}g_v;f_{\mu_{1,v},\mu_{2,v},\Phi_v,s})\\
&= {}^\sigma\!\mu_{1,v}(u_p^{-1}\det(g_v))|\det(g_v)|_{\F_v}^s \int_{\F_v^\times}{}^\sigma\!\left(\widetilde{\rho({\bf a}(u_p)^{-1}g_v)\Phi_v} \right)(t_v,-t_v^{-1}){}^\sigma\!\mu_{1,v}{}^\sigma\!\mu_{2,v}^{-1}(t_v)|t_v|_{\F_v}^{2s-1}\,d^\times t_v^{\rm std}\\
&= {}^\sigma\!\mu_{2,v}(u_p)^{-1}{}^\sigma\!\mu_{1,v}(\det(g_v))|\det(g_v)|_{\F_v}^s \int_{\F_v^\times}{}^\sigma\!\left(\widetilde{\rho({\bf a}(u_p)^{-1}g_v)\Phi_v} \right)(u_pt_v,-u_p^{-1}t_v^{-1}){}^\sigma\!\mu_{1,v}{}^\sigma\!\mu_{2,v}^{-1}(t_v)|t_v|_{\F_v}^{2s-1}\,d^\times t_v^{\rm std}.
\end{align*}
For the partial Fourier transform, we have
\begin{align*}
{}^\sigma\!\left(\widetilde{\rho({\bf a}(u_p)^{-1}g_v)\Phi_v} \right) (x,y) &= \int_{\F_v} {}^\sigma\!\Phi_v((x,w){\bf a}(u_p)^{-1}g_v){}^\sigma\!\psi_{\F_v}(wy)\,dw^{\rm std}\\
&= \int_{\F_v}{}^\sigma\!\Phi_v((u_p^{-1}x,w)g_v)\psi_{\F_v}(wu_py)\,dw^{\rm std}\\
&= \widetilde{\rho(g_v){}^\sigma\!\Phi_v}(u_p^{-1}x,u_py).
\end{align*}
Thus (\ref{E:4.12}) holds.
Also note that when $\mu_{1,v}, \mu_{2,v}$ are unramified, $\F_v/\Q_p$ is unramified, $\Phi_v = \mathbb{I}_{\frak{o}_{\F_v}^2}$, and $g_v \in \GL_2(\frak{o}_{\F_v})$, we have
\[
W_{\psi_{\F_v}}(g_v;f_{\mu_{1,v},\mu_{2,v},\Phi_v,s})=1.
\]
Therefore, by (\ref{E:Galois Gauss sum}), (\ref{E:4.11}), and (\ref{E:4.12}) we have
\[
\sigma\left( \frac{W_\alpha(E^{[\kappa]}(\mu_1,\mu_2,\Phi),g_f)}{D_\F^{1/2}G(\mu_2^{-1})}\right) = \frac{W_\alpha(E^{[\kappa]}({}^\sigma\!\mu_1,{}^\sigma\!\mu_2,{}^\sigma\!\Phi),g_f)}{D_\F^{1/2}G({}^\sigma\!\mu_2^{-1})}
\]
for $\alpha \in \F^\times$. This together with Theorem \ref{T:q-expansion} imply that the holomorphic automorphic forms
\[
\frac{{}^\sigma\!E^{[\kappa]}(\mu_1,\mu_2,\Phi)}{\sigma\left(D_\F^{1/2}(2\pi\sqrt{-1})^{-[\F:\Q](\kappa+n_1+n_2)/2}\cdot G(\mu_2^{-1})\right)}
\]
and
\[
\frac{E^{[\kappa]}({}^\sigma\!\mu_1,{}^\sigma\!\mu_2,{}^\sigma\!\Phi)}{D_\F^{1/2}(2\pi\sqrt{-1})^{-[\F:\Q](\kappa+n_1+n_2)/2}\cdot G({}^\sigma\!\mu_2^{-1})}
\]
have the same Fourier coefficients for $\alpha \neq 0$ and same motivic weight $(\underline{\kappa},n_1+n_2)$. Thus they must be equal. This completes the proof.
\end{proof}


\subsection{$L$-functions for $\GL_2 \times \GL_2$}\label{S:4.3}

Let $\itPi = \bigotimes_v \pi_v$ and $\itPi' = \bigotimes_v \pi_v'$ be motivic irreducible cuspidal automorphic representations of $\GL_2(\A_\F)$ with motivic weights $(\underline{\ell},r) \in \Z_{\geq 1}[\Sigma_\F] \times \Z$ and $(\underline{\ell}',r') \in \Z_{\geq 1}[\Sigma_\F] \times \Z$, respectively, and $\chi$ an algebraic Hecke character of $\A_\F^\times$ with $|\chi| = |\mbox{ }|_{\A_\F}^{r_0}$ for some $r_0 \in \Z$. Let 
\[
L(s,\itPi \times \itPi' \times \chi) = \prod_v L(s,\pi_v \times \pi_v' \times \chi_v)
\]
be the twisted Rankin--Selberg $L$-function of $\itPi \times \itPi' \times \chi$ defined by the tensor representation
\[
{}^L(\GL_{2/\F} \times \GL_{2/\F} \times \GL_{1/\F}) \longrightarrow \GL(\C^2 \otimes \C^2 \otimes \C). 
\]
We denote by $L^{(\infty)}(s,\itPi \times \itPi' \times \chi)$ the $L$-function obtained by excluding the archimedean $L$-factors.
A standard unfolding argument (cf.\,\cite[\S\,19]{JLbook2}) shows that
\begin{align}\label{E:integral rep.}
\int_{\A_\F^\times \GL_2(\F) \backslash \GL_2(\A_\F)} \varphi(g)\varphi'(g)E(g;f_s)\,dg^{\rm Tam} = D_\F^{-1}\zeta_\F(2)^{-1}\cdot L(s,\itPi \times \itPi' \times \chi)\cdot \prod_v Z^*(W_v,W_v',f_{s,v})
\end{align}
as meromorphic functions in $s \in \C$, for $\varphi \in \itPi$, $\varphi' \in \itPi'$, and $f_s \in I(\chi,\chi^{-1}\omega_{\itPi}^{-1}\omega_{\itPi'}^{-1},s)$ such that 
\[
W_{\varphi,\psi_\F} = \prod_v W_v,\quad W_{\varphi',\psi_\F} = \prod_v W_v',\quad f_s = \prod_v f_{s,v}.
\]
Here 
\[
Z^*(W_v,W_v',f_{s,v}) = \frac{Z(W_v,W_v',f_{s,v})}{L(s,\pi_v \times \pi_v' \times \chi_v)}
\]
and $Z(W_v,W_v',f_{s,v})$ is the local zeta integral defined by
\[
Z(W_v,W_v',f_{s,v}) = \int_{\F_v^\times N(\F_v) \backslash \GL_2(\F_v)} W_v(g_v)W_v'({\bf a}(-1)g_v)f_{s,v}(g_v)\,dg_v^{\rm std}
\]
and $dg_v^{\rm std}$ is the standard measure defined so that $dg_v^{\rm std}$ is the quotient measure of the measure on $\PGL_2(\F_v)$ defined in \S\,\ref{SS:Ichino} by the Lebesgue measure on $\F_v$ if $v \in \Sigma_\F$, and ${\rm vol}(\o_{\F_v}^\times N(\o_{\F_v})\backslash \GL_2(\o_{\F_v}))=1$ if $v \nmid \infty$.
Note that the factor $D_\F^{-1/2}\zeta_\F(2)^{-1}$ is the ratio between the Tamagawa and standard measures.
We have the following lemmas on basic properties of the non-archimedean local zeta integrals and explicit calculation of certain archimedean local zeta integrals.
\begin{lemma}\label{L:Galois equiv. local zeta}
Assume $v$ is a finite place of $\F$ and $v \mid p$ for some rational prime $p$. Let $W_v$ and $W_v'$ be Whittaker functions of $\pi_v$ and $\pi_v'$ with respect to $\psi_{\F_v}$, respectively, and $f_{s,v} \in I(\chi_v,\chi_v^{-1}\omega_{\itPi,v}^{-1}\omega_{\itPi',v}^{-1},s)$.
\begin{itemize}
\item[(1)] The integral $Z(W_v,W_v',f_{s,v})$ is absolutely convergent for ${\rm Re}(s)$ sufficiently large and admits meromorphic continuation to $s \in \C$. Moreover, if $f_{s,v} = f_{\chi_v,\chi_v^{-1}\omega_{\itPi,v}^{-1}\omega_{\itPi',v}^{-1},\Phi_v,s}$ for some $\Phi_v \in \mathcal{S}(\F_v^2)$, then $Z^*(W_v,W_v',f_{\chi_v,\chi_v^{-1}\omega_{\itPi,v}^{-1}\omega_{\itPi',v}^{-1},\Phi_v,s})$ is a polynomial in $\C[q_v^{s},q_v^{-s}]$. Here $q_v$ is the cardinality of the residue field of $\F_v$.
\item[(2)] Suppose $\pi_v, \pi'_v, \chi_v, \psi_{\F_v}$ are unramified. Let $W_v^\circ$ and $(W_v')^\circ$ be the right $\GL_2(\o_{\F_v})$-invariant Whittaker functions of $\pi_v$ and $\pi_v'$ with respect to $\psi_{\F_v}$, respectively, normalized so that $W_v^\circ(1) = (W_v')^\circ(1) = 1$ and  
$\Phi_v^\circ = \mathbb{I}_{\o_{\F_v}^2}$.
Then we have
\[
Z^*(W_v^\circ, (W_v')^\circ, f_{\chi_v,\chi_v^{-1}\omega_{\itPi,v}^{-1}\omega_{\itPi',v}^{-1},\Phi_v^\circ,s}) = 1.
\]
\item[(3)] For $\sigma \in {\rm Aut}(\C)$, we have
\[
{}^\sigma\! Z^*(W_v,W_v',f_{\chi_v,\chi_v^{-1}\omega_{\itPi,v}^{-1}\omega_{\itPi',v}^{-1},\Phi_v,s}) = {}^\sigma\!\chi_v(u_p)^{-1}Z^*(t_\sigma W_v,t_\sigma W_v',f_{{}^\sigma\!\chi_v,{}^\sigma\!\chi_v^{-1}{}^\sigma\!\omega_{\itPi,v}^{-1}{}^\sigma\!\omega_{\itPi',v}^{-1},{}^\sigma\!\Phi_v,s})
\]
as polynomials in $\C[q_v^{s},q_v^{-s}]$.
Here $u_p \in \Z_p^\times$ is the unique element such that ${}^\sigma\!\psi_{\F_v} = \psi_{\F_v}^u$ and $t_\sigma W_v$ is the Whittaker function of 
${}^\sigma\!\pi_v$ with respect to $\psi_{\F_v}$ defined by
\[
t_\sigma W_v(g) = {}^\sigma W_v({\bf a}(u_p)^{-1}g).
\]
\end{itemize}
\end{lemma}

\begin{proof}
The assertions (1) and (2) were proved in \cite[Theorem 14.7]{JLbook2}. Let $\sigma \in {\rm Aut}(\C)$. Analogous to the proof of Lemma \ref{L:Galois equiv. 1} below, we can show that
\[
{}^\sigma\!L(s, \pi_v \times \pi_v' \times \chi_v) = L(s, {}^\sigma\!\pi_v \times {}^\sigma\!\pi_v' \times {}^\sigma\!\chi_v)
\]
as rational functions in $q_v^{-s}$. To prove assertion (3), it suffices to show that
\[
{}^\sigma\! Z(W_v,W_v',f_{\chi_v,\chi_v^{-1}\omega_{\itPi,v}^{-1}\omega_{\itPi',v}^{-1},\Phi_v,s}) = {}^\sigma\!\chi_v(u_p)^{-1}Z(t_\sigma W_v,t_\sigma W_v',f_{{}^\sigma\!\chi_v,{}^\sigma\!\chi_v^{-1}{}^\sigma\!\omega_{\itPi,v}^{-1}{}^\sigma\!\omega_{\itPi',v}^{-1},{}^\sigma\!\Phi_v,s}).
\]
We recall a type of integral
\[
\int_{\F_v^\times}f(a)\mu(a)|a|_{\F_v}^s(\log_{q_v}|a|_{\F_v})^n\,d^\times a,
\]
where $f \in \mathcal{S}(\F_v)$, $\mu$ is a character of $\F_v^\times$, and $n \in \Z_{\geq 0}$. The integral converges absolutely for ${\rm Re}(s)$ sufficiently large and defines a rational function in $q_v^{-s}$. Moreover, we have
\begin{align}\label{E:Galois equiv. 4.5}
{}^\sigma\!\left( \int_{\F_v^\times}f(a)\mu(a)|a|_{\F_v}^s(\log_{q_v}|a|_{\F_v})^n\,d^\times a\right) = \int_{\F_v^\times}{}^\sigma\!f(a){}^\sigma\!\mu(a)|a|_{\F_v}^s(\log_{q_v}|a|_{\F_v})^n\,d^\times a
\end{align}
as rational functions in $q_v^{-s}$ (cf.\,\cite[Proposition A]{Grobner2018}). Here ${\rm vol}(\o_{\F_v}^\times,d^\times a)=1$.
It is well-known that (cf.\,\cite[Lemma 14.3]{JLbook2}) there exist locally constant functions $f_1,f_2,f_1',f_2'$ on $\F_v \times \GL_2(\o_{\F_v})$ with compact support, characters $\mu_1,\mu_2,\mu_1',\mu_2'$ of $\F_v^\times$, and integers $n_1,n_2,n_1',n_2'$ such that 
\begin{align*}
W_v({\bf a}(a)k) &= f_1(a,k)\mu_1(a)(\log_{q_v}|a|_{\F_v})^{n_1} + f_2(a,k)\mu_2(a)(\log_{q_v}|a|_{\F_v})^{n_2},\\
W_v'({\bf a}(-a)k) &= f_1'(a,k)\mu_1'(a)(\log_{q_v}|a|_{\F_v})^{n_1'} + f_2'(a,k)\mu_2'(a)(\log_{q_v}|a|_{\F_v})^{n_2'}
\end{align*}
for $a \in \F_v^\times$ and $k \in \GL_2(\o_{\F_v})$.
Therefore, we have
\begin{align*}
&Z(W_v,W_v',f_{\chi_v,\chi_v^{-1}\omega_{\itPi,v}^{-1}\omega_{\itPi',v}^{-1},\Phi_v,s}) \\
& = \int_{\GL_2(\o_{\F_v})}\int_{\F_v^\times}W_v({\bf a}(a)k)W_v'({\bf a}(-a)k)f_{\chi_v,\chi_v^{-1}\omega_{\itPi,v}^{-1}\omega_{\itPi',v}^{-1},\Phi_v,s}({\bf a}(a)k)\,\frac{d^\times a}{|a|_{\F_v}} dk\\
& = \sum_{1 \leq i,j \leq 2}\int_{\GL_2(\o_{\F_v})}f_{\chi_v,\chi_v^{-1}\omega_{\itPi,v}^{-1}\omega_{\itPi',v}^{-1},\Phi_v,s}(k) \int_{\F_v^\times} f_i(a,k)f_j'(a,k)\mu_i\mu_j'\chi_v(a)|a|_{\F_v}^{s-1}(\log_{q_v}|a|_{\F_v})^{n_i+n_j'}\,d^\times a dk.
\end{align*}
Here ${\rm vol}(\GL_2(\o_{\F_v}),dk)=1$.
It then follows from (\ref{E:Galois action on sections}) and (\ref{E:Galois equiv. 4.5}) that
\begin{align*}
&{}^\sigma\! Z(W_v,W_v',f_{\chi_v,\chi_v^{-1}\omega_{\itPi,v}^{-1}\omega_{\itPi',v}^{-1},\Phi_v,s})\\
& = \sum_{1 \leq i,j \leq 2}\int_{\GL_2(\o_{\F_v})}f_{{}^\sigma\!\chi_v,{}^\sigma\!\chi_v^{-1}{}^\sigma\!\omega_{\itPi,v}^{-1}{}^\sigma\!\omega_{\itPi',v}^{-1},{}^\sigma\!\Phi_v,s}(k) \int_{\F_v^\times} {}^\sigma\!f_i(a,k){}^\sigma\!f_j'(a,k){}^\sigma\!\mu_i{}^\sigma\!\mu_j'{}^\sigma\!\chi_v(a)|a|_{\F_v}^{s-1}(\log_{q_v}|a|_{\F_v})^{n_i+n_j'}\,d^\times a dk\\
& = \int_{\GL_2(\o_{\F_v})}\int_{\F_v^\times}{}^\sigma W_v({\bf a}(a)k){}^\sigma W_v'({\bf a}(a)k)f_{{}^\sigma\!\chi_v,{}^\sigma\!\chi_v^{-1}{}^\sigma\!\omega_{\itPi,v}^{-1}{}^\sigma\!\omega_{\itPi',v}^{-1},{}^\sigma\!\Phi_v,s}({\bf a}(a)k)\,\frac{d^\times a}{|a|_{\F_v}} dk\\
& = {}^\sigma\!\chi_v(u_p)^{-1}\int_{\GL_2(\o_{\F_v})}\int_{\F_v^\times} t_\sigma W_v({\bf a}(a)k)t_\sigma W_v'({\bf a}(a)k)f_{{}^\sigma\!\chi_v,{}^\sigma\!\chi_v^{-1}{}^\sigma\!\omega_{\itPi,v}^{-1}{}^\sigma\!\omega_{\itPi',v}^{-1},{}^\sigma\!\Phi_v,s}({\bf a}(a)k)\,\frac{d^\times a}{|a|_{\F_v}} dk\\
& = {}^\sigma\!\chi_v(u_p)^{-1}Z(t_\sigma W_v,t_\sigma W_v',f_{{}^\sigma\!\chi_v,{}^\sigma\!\chi_v^{-1}{}^\sigma\!\omega_{\itPi,v}^{-1}{}^\sigma\!\omega_{\itPi',v}^{-1},{}^\sigma\!\Phi_v,s}).
\end{align*}
This completes the proof.
\end{proof}

\begin{lemma}\label{L:archimedean local zeta}
Assume $v \in \Sigma_\F$ and $\ell_v > \ell_v'$. Let $m$ be an integer such that
\[
\frac{1}{2} \leq m+r_0+\frac{r+r'}{2} \leq \frac{\ell_v-\ell_v'}{2}.
\]
Put $m'=2m+2r_0+r+r'$. For integers $m_1,m_2 \in \Z_{\geq 0}$ such that $m'+2m_1+2m_2 = \ell_v - \ell_v'$, we have
\begin{align*}
&Z(W_{(\ell_v,r),\psi_{\F_v}}^-,X_+^{m_1}\cdot W_{(\ell_v',r'),\psi_{\F_v}}^+,X_+^{m_2}\cdot f_{\chi_v,\chi_v^{-1}\omega_{\itPi,v}^{-1}\omega_{\itPi',v}^{-1},\Phi^{[m']},s})\vert_{s = m}\\
&=\chi_v(-1)(-1)^{m'+m_1}2^{2-\ell_v'-m'}(\sqrt{-1})^{\ell_v}\pi\cdot (2\pi\sqrt{-1})^{-(\ell_v+\ell_v'+3m')/2}\cdot\Gamma\left(\tfrac{\ell_v+\ell_v'+m'}{2}-1\right)\Gamma\left(\tfrac{\ell_v-\ell_v'+m'}{2}\right).
\end{align*}

\end{lemma}

\begin{proof}
Note that we have
\begin{align*}
W_{(\ell_v,r),\psi_{\F_v}}^-({\bf a}(-a)) & = a^{(\ell_v+r)/2}e^{-2\pi a}\cdot \mathbb{I}_{\R>0}(a),\\
X_+^{m_1}\cdot W_{(\ell_v',r'),\psi_{\F_v}}^+({\bf a}(a)) &= (2\sqrt{-1})^{m_1}\sum_{j=0}^{m_1}(-4\pi)^j{m_1 \choose j}\frac{\Gamma(\ell_v'+m_1)}{\Gamma(\ell_v'+j)}a^{(\ell_v'+r')/2+j}e^{-2\pi a}\cdot \mathbb{I}_{\R>0}(a),\\
X_+^{m_2}\cdot f_{\chi_v,\chi_v^{-1}\omega_{\itPi,v}^{-1}\omega_{\itPi',v}^{-1},\Phi^{[m']},s} \vert_{s=m}({\bf a}(a)) & = (-8\pi\sqrt{-1})^{m_2}f_{\chi_v,\chi_v^{-1}\omega_{\itPi,v}^{-1}\omega_{\itPi',v}^{-1},\Phi^{[m'+2m_2]},s} \vert_{s=m}({\bf a}(a))\\
&=2^{-m'+m_2}(\sqrt{-1})^{m'+m_2}\pi^{-m'}\Gamma(m'+m_2)\chi_v(a)|a|^{m}
\end{align*}
for $a \in \F_v^\times=\R^\times$.
Here we refer to \cite[Lemma 3.3]{CC2017} for the second equation.
Hence
\begin{align*}
&Z(W_{(\ell_v,r),\psi_{\F_v}}^-,X_+^{m_1}\cdot W_{(\ell_v',r'),\psi_{\F_v}}^+,X_+^{m_2}\cdot f_{\chi_v,\chi_v^{-1}\omega_{\itPi,v}^{-1}\omega_{\itPi',v}^{-1},\Phi^{[m']},s})\vert_{s = m}\\
& = \int_{\R^\times} W_{(\ell_v,r),\psi_{\F_v}}^-({\bf a}(a))(X_+^{m_1}\cdot W_{(\ell_v',r'),\psi_{\F_v}}^+)({\bf a}(-a))(X_+^{m_2}\cdot f_{\chi_v,\chi_v^{-1}\omega_{\itPi,v}^{-1}\omega_{\itPi',v}^{-1},\Phi^{[m']},s}\vert_{s=m})({\bf a}(a))\,\frac{d^\times a}{|a|}\\
&=\chi_v(-1)2^{-m'+m_1+m_2}(\sqrt{-1})^{m'+m_1+m_2}\pi^{-m'}\Gamma(\ell_v'+m_1)\Gamma(m'+m_2)\\
&\times\sum_{j=0}^{m_1}(-4\pi)^j{m_1 \choose j}\Gamma(\ell_v'+j)^{-1}\int_0^\infty a^{(\ell_v+\ell_v'+m')/2+j-1}e^{-4\pi a}\,d^\times a\\
&=\chi_v(-1)2^{-m'+m_1+m_2}(\sqrt{-1})^{m'+m_1+m_2}\pi^{-m'}\Gamma(\ell_v'+m_1)\Gamma(m'+m_2)\\
&\times (4\pi)^{1-(\ell_v+\ell_v'+m')/2}\sum_{j=0}^{m_1}(-1)^j{m_1 \choose j}\frac{\Gamma\left(\tfrac{\ell_v+\ell_v'+m'}{2}+j-1\right)}{\Gamma(\ell_v'+j)}.
\end{align*}
By \cite[Lemma 2.1]{Ikeda1998}, we have
\[
\sum_{j=0}^{m_1}(-1)^j{m_1 \choose j}\frac{\Gamma\left(\tfrac{\ell_v+\ell_v'+m'}{2}+j-1\right)}{\Gamma(\ell_v'+j)} = (-1)^{m_1}\frac{\Gamma\left(\tfrac{\ell_v+\ell_v'+m'}{2}-1\right)\Gamma\left(\tfrac{\ell_v-\ell_v'+m'}{2}\right)}{\Gamma(\ell_v'+m_1)\Gamma(m'+m_2)}.
\]
The assertion thus follows.
\end{proof}

\begin{rmk}
When $\ell_v < \ell_v$, we have similar formula by exchanging $\ell_v$ and $\ell_v'$ and excluding $\chi_v(-1)$.
\end{rmk}

\subsection{Algebraicity of the critical values}

We keep the notation of \S\,\ref{S:4.3}.
Put
\[
I = \{v \in \Sigma_\F\,\vert\, \ell_v > \ell_v'\},\quad J = \{v \in \Sigma_\F\,\vert\, \ell_v < \ell_v'\}.
\]
It is clear that $I$ and $J$ are admissible with respect to $\underline{\ell}$ and $\underline{\ell}'$, respectively.

The following result is on the algebraicity of the critical values of $L(s,\itPi \times \itPi' \times \chi)$. We generalize the result \cite[Theorem 3.5.1]{Harris1989I} of Harris where $\chi$ is trivial.
\begin{thm}\label{T:RS}
Assume that $\ell_v \neq \ell_v'$ for all $v \in \Sigma_\F$. Let $m \in \Z$ such that
\[
\frac{1}{2} \leq m+r_0+\frac{r+r'}{2} \leq \frac{|\ell_v-\ell_v'|}{2}
\]
for all $v \in \Sigma_\F$.
We have
\begin{align*}
&\sigma \left(\frac{L^{(\infty)}(m,\itPi \times \itPi' \times \chi)}{(2\pi\sqrt{-1})^{[\F:\Q](2m+2r_0+r+r')}(\sqrt{-1})^{\sum_{v \in I}\ell_v + \sum_{v \in J}\ell_v'}\cdot G(\chi^2\omega_\itPi\omega_{\itPi'}) \cdot \Omega^I(\itPi)\cdot \Omega^J(\itPi')} \right)\\
& = \frac{L^{(\infty)}(m,{}^\sigma\!\itPi \times {}^\sigma\!\itPi' \times {}^\sigma\!\chi)}{(2\pi\sqrt{-1})^{[\F:\Q](2m+2r_0+r+r')}(\sqrt{-1})^{\sum_{v \in I}\ell_v + \sum_{v \in J}\ell_v'}\cdot G({}^\sigma\!\chi^2{}^\sigma\!\omega_\itPi{}^\sigma\!\omega_{\itPi'}) \cdot \Omega^{{}^\sigma\!I}({}^\sigma\!\itPi)\cdot \Omega^{{}^\sigma\!J}({}^\sigma\!\itPi')}
\end{align*}
for all $\sigma \in {\rm Aut}(\C)$.

\end{thm}

\begin{proof}
Consider the \'etale cubic algebra 
\[
\E = \F \times \F \times \F
\]
over $\F$.
We have $\Sigma_\E = \Sigma_\F^{(1)} \sqcup \Sigma_\F^{(2)}\sqcup \Sigma_\F^{(3)}$, where $\Sigma_\F^{(i)}$ is the set of algebra homomorphisms from $\E$ into $\R$ which are non-zero on its $i$-th component. 
For each $v \in \Sigma_\F$, there are three extensions $v^{(1)}, v^{(2)}, v^{(3)} \in \Sigma_\E$ of $v$. We arrange the index so that $v^{(i)} \in \Sigma_\F^{(i)}$.
Fix $m \in \Z$ such that
\[
\frac{1}{2} \leq m+r_0+\frac{r+r'}{2} \leq \frac{|\ell_v-\ell_v'|}{2}
\]
for all $v \in \Sigma_\F$ and put
\[
m' = 2m+2r_0+r+r'.
\]
Let $(\underline{\kappa},\underline{\bf r}) \in \Z_{\geq 1}[\Sigma_\E] \times \Z[\Sigma_\E]$ be motivic defined by
\[
\kappa_w = \begin{cases}
\ell_v & \mbox{ if $w=v^{(1)} \in \Sigma_\F^{(1)}$},\\
\ell'_v & \mbox{ if $w=v^{(2)} \in \Sigma_\F^{(2)}$},\\
m' & \mbox{ if $w=v^{(3)} \in \Sigma_\F^{(3)}$},
\end{cases} \quad {\bf r}_w = \begin{cases}
r & \mbox{ if $w=v^{(1)} \in \Sigma_\F^{(1)}$},\\
r' & \mbox{ if $w=v^{(2)} \in \Sigma_\F^{(2)}$},\\
-r-r' & \mbox{ if $w=v^{(3)} \in \Sigma_\F^{(3)}$}.
\end{cases}
\]
Note that $\underline{\kappa}$ satisfied the totally unbalanced condition (\ref{E:totally unbalanced}). Recall for each $v \in \Sigma_\F$, we define $\tilde{v}(\underline{\kappa}) \in \Sigma_\E$ be the homomorphism corresponding to $\max_{w \mid v}\{\kappa_w\}$, that is, $\kappa_{\tilde{v}(\underline{\kappa})} = \max_{w \mid v}\{\kappa_w\}$. Therefore, 
\[
\tilde{v}(\underline{\kappa}) = \begin{cases}
v^{(1)} & \mbox{ if $v \in I$},\\
v^{(2)} & \mbox{ if $v \in J$},
\end{cases}
\]
and 
\[
I_{\underline{\kappa}} = \{ \tilde{v}(\underline{\kappa}) \, \vert \, v \in \Sigma_\F\} = I \sqcup J \sqcup \emptyset
\]
with respect to the disjoint union $\Sigma_\E = \Sigma_\F^{(1)} \sqcup \Sigma_\F^{(2)}\sqcup \Sigma_\F^{(3)}$.
Consider the automorphic line bundle $\mathcal{L}_{(\underline{\kappa}(I_{\underline{\kappa}}),\underline{\bf r})}$ on the Shimura variety 
\[
{\rm Sh}({\rm R}_{\E/\Q}\GL_{2/\E},(\frak{H}^\pm)^{\Sigma_\E}) = {\rm Sh}({\rm R}_{\F/\Q}\GL_{2/\F},(\frak{H}^\pm)^{\Sigma_\F}) \times {\rm Sh}({\rm R}_{\F/\Q}\GL_{2/\F},(\frak{H}^\pm)^{\Sigma_\F}) \times {\rm Sh}({\rm R}_{\F/\Q}\GL_{2/\F},(\frak{H}^\pm)^{\Sigma_\F})
\]
and the autmorphic line bundles $\mathcal{L}_{(\underline{\ell}(I),\underline{r})}'$, $\mathcal{L}_{(\underline{\ell}'(J),\underline{r}')}'$, and $\mathcal{L}_{(\underline{m}',-\underline{r}-\underline{r}')}'$ on ${\rm Sh}({\rm R}_{\F/\Q}\GL_{2/\F},(\frak{H}^\pm)^{\Sigma_\F})$.
We have
\[
\mathcal{L}_{(\underline{\kappa}(I_{\underline{\kappa}}),\underline{\bf r})} = \mathcal{L}_{(\underline{\ell}(I),\underline{r})}' \times \mathcal{L}_{(\underline{\ell}'(J),\underline{r}')}' \times 
\mathcal{L}_{(\underline{m}',-\underline{r}-\underline{r}')}'.
\]
Let $[\delta(\underline{\kappa})] : \mathcal{L}_{(\underline{\kappa}(I_{\underline{\kappa}}),\underline{\bf r})} \rightarrow \mathcal{L}'_{(\underline{2},\underline{0})}$ be the trilinear differential operator constructed in Proposition \ref{P:differential operator}. In this case, it induces a homomorphism 
\[
[\delta(\underline{\kappa})] : H^{{}^\sharp\!I}((\mathcal{L}_{(\underline{\ell}(I),\underline{r})}')^{\rm sub})\otimes H^{{}^\sharp\!J}((\mathcal{L}_{(\underline{\ell}'(J),\underline{r}')}')^{\rm sub}) \otimes  H^{0}((\mathcal{L}'_{(\underline{m}',-\underline{r}-\underline{r}')})^{\rm can}) \longrightarrow H^{[\F:\Q]}((\mathcal{L}'_{(\underline{2},\underline{0})})^{\rm sub})
\]
which satisfies the Galois equivariant property:
\begin{align}\label{E:Galois equiv. 4.15}
T_\sigma([\delta(\underline{\kappa})](c_1 \otimes c_2 \otimes c_3)) = [\delta({}^\sigma\!\underline{\kappa})](T_\sigma c_1 \otimes T_\sigma c_2 \otimes T_\sigma c_3)
\end{align}
for all $\sigma \in {\rm Aut}(\C)$.
Moreover, if $c_1,c_2,c_3$ are represented by
\[
\varphi_1 \otimes \bigwedge_{v \in I}X_{+,v} \otimes {\bf v}_{(\underline{\ell}(I),\underline{r})},\quad \varphi_2 \otimes \bigwedge_{v \in J}X_{+,v} \otimes {\bf v}_{(\underline{\ell}'(J),\underline{r}')},\quad \varphi_3 \otimes {\bf v}_{(\underline{m}',-\underline{r}-\underline{r}')},
\]
respectively, then $[\delta(\underline{\kappa})](c_1 \otimes c_2 \otimes c_3)$ is represented by
\[
({\bf X}(\underline{\kappa})\cdot(\varphi_1,\varphi_2,\varphi_3) )\vert_{\GL_2(\A_\F)} \otimes \bigwedge_{v \in \Sigma_\F}X_{+,v} \otimes {\bf v}_{(\underline{2},\underline{0})}.
\]
Here $(\varphi_1,\varphi_2,\varphi_3)$ is the automorphic form on $\GL_2(\A_\E) = \GL_2(\A_\F) \times \GL_2(\A_\F) \times \GL_2(\A_\F)$ defined by 
\[
(\varphi_1,\varphi_2,\varphi_3)(g_1,g_2,g_3) = \varphi_1(g_1)\cdot \varphi_2(g_2) \cdot \varphi_1(g_2),
\]
and ${\bf X}(\underline{\kappa}) \in U(\frak{gl}_{2,\C}^{\Sigma_\E})$ is the differential operator defined in (\ref{E:differential operator}).
For $\varphi \in \itPi_{\rm hol}$, $\varphi' \in \itPi'_{\rm hol}$, and $\Phi \in \mathcal{S}(\A_{\F,f}^2)$, let $c({\varphi,\varphi',\chi,\Phi}) \in H^{{}^\sharp\!I}((\mathcal{L}_{(\underline{\ell}(I),\underline{r})}')^{\rm sub})\otimes H^{{}^\sharp\!J}((\mathcal{L}_{(\underline{\ell}'(J),\underline{r}')}')^{\rm sub}) \otimes  H^{0}((\mathcal{L}'_{(\underline{m}',-\underline{r}-\underline{r}')})^{\rm can})$ be the class defined by
\[
c({\varphi,\varphi',\chi,\Phi}) = \xi_I(\varphi) \otimes \xi_J(\varphi') \otimes \left(E^{[m']}(\chi,\chi^{-1}\omega_\itPi^{-1}\omega_{\itPi'}^{-1},\Phi) \otimes {\bf v}_{(\underline{m}',-\underline{r}-\underline{r}')}\right).
\]
Here $\xi_I$ and $\xi_J$ are defined in (\ref{E:xi_I}).
By Proposition \ref{P:Harris' period} and Lemma \ref{L:Eisenstein series}, this cohomology class satisfies the following Galois equivariant property:
\begin{align*}
&T_\sigma\left(\frac{c({\varphi,\varphi',\chi,\Phi})}{D_\F^{1/2}(2\pi\sqrt{-1})^{-[\F:\Q](m+r_0)}\cdot G(\chi\omega_\itPi\omega_{\itPi'})\cdot\Omega^I(\itPi)\cdot\Omega^J(\itPi')}\right)\\
& = \frac{c({{}^\sigma\!\varphi,{}^\sigma\!\varphi',{}^\sigma\!\chi,{}^\sigma\!\Phi})}{D_\F^{1/2}(2\pi\sqrt{-1})^{-[\F:\Q](m+r_0)}\cdot G({}^\sigma\!\chi{}^\sigma\!\omega_\itPi{}^\sigma\!\omega_{\itPi'})\cdot\Omega^{{}^\sigma\!I}({}^\sigma\!\itPi)\cdot\Omega^{{}^\sigma\!J}({}^\sigma\!\itPi')}
\end{align*}
for all $\sigma \in {\rm Aut}(\C)$.
It then follows from Lemma \ref{L:trace} and (\ref{E:Galois equiv. 4.15}) that 
\begin{align*}
&\sigma\left( \frac{\int_{{\rm Sh}({\rm R}_{\F/\Q}\GL_{2/\F},(\frak{H}^\pm)^{\Sigma_\F})}[\delta(\underline{\kappa})]c(\varphi,\varphi',\chi,\Phi)}{{D_\F^{1/2}(2\pi\sqrt{-1})^{-[\F:\Q](m+r_0)}\cdot G(\chi\omega_\itPi\omega_{\itPi'})\cdot\Omega^I(\itPi)\cdot\Omega^J(\itPi')}}\right) \\
&= \frac{\int_{{\rm Sh}({\rm R}_{\F/\Q}\GL_{2/\F},(\frak{H}^\pm)^{\Sigma_\F})}[\delta({}^\sigma\!\underline{\kappa})]c({}^\sigma\!\varphi,{}^\sigma\!\varphi',{}^\sigma\!\chi,{}^\sigma\!\Phi)}{D_\F^{1/2}(2\pi\sqrt{-1})^{-[\F:\Q](m+r_0)}\cdot G({}^\sigma\!\chi{}^\sigma\!\omega_\itPi{}^\sigma\!\omega_{\itPi'})\cdot\Omega^{{}^\sigma\!I}({}^\sigma\!\itPi)\cdot\Omega^{{}^\sigma\!J}({}^\sigma\!\itPi')}.
\end{align*}
By the explicit realization of $[\delta(\underline{\kappa})]$ described above and Lemma \ref{L:trace} again, we conclude that
\begin{align}\label{E:cohomological interpretation}
\begin{split}
&\sigma\left( \frac{\int_{\A_\F^\times \GL_2(\F) \backslash \GL_2(\A_\F)}{\bf X}(\underline{\kappa})\cdot(\varphi^I,(\varphi')^J, E^{[m']}(\chi,\chi^{-1}\omega_\itPi^{-1}\omega_{\itPi'}^{-1},\Phi))(g,g,g) \,dg^{\rm Tam}}{{D_\F^{1/2}(2\pi\sqrt{-1})^{-[\F:\Q](m+r_0)}\cdot G(\chi\omega_\itPi\omega_{\itPi'})\cdot\Omega^I(\itPi)\cdot\Omega^J(\itPi')}}\right) \\
&= \frac{\int_{\A_\F^\times \GL_2(\F) \backslash \GL_2(\A_\F)}{\bf X}({}^\sigma\!\underline{\kappa})\cdot({}^\sigma\!\varphi^{{}^\sigma\!I},({}^\sigma\!\varphi')^{{}^\sigma\!J}, E^{[m']}({}^\sigma\!\chi,{}^\sigma\!\chi^{-1}{}^\sigma\!\omega_\itPi^{-1}{}^\sigma\!\omega_{\itPi'}^{-1},{}^\sigma\!\Phi))(g,g,g) \,dg^{\rm Tam}}{D_\F^{1/2}(2\pi\sqrt{-1})^{-[\F:\Q](m+r_0)}\cdot G({}^\sigma\!\chi{}^\sigma\!\omega_\itPi{}^\sigma\!\omega_{\itPi'})\cdot\Omega^{{}^\sigma\!I}({}^\sigma\!\itPi)\cdot\Omega^{{}^\sigma\!J}({}^\sigma\!\itPi')}
\end{split}
\end{align}
for all $\sigma \in {\rm Aut}(\C)$.
On the other hand, assume 
\[
W_{\varphi,\psi_\F} = \prod_{v \in \Sigma_\F}W_{(\ell_v,r),\psi_{\F_v}}^+\prod_{v\nmid \infty} W_v, \quad W_{\varphi',\psi_\F} = \prod_{v \in \Sigma_\F}W_{(\ell_v',r'),\psi_{\F_v}}^+\prod_{v\nmid \infty} W_v',
\] and $\Phi = \bigotimes_{v\nmid \infty}\Phi_v$, we have the integral representation of $L^{(\infty)}(m,\itPi \times \itPi' \times \chi)$ recalled in (\ref{E:integral rep.}):
\begin{align*}
&\int_{\A_\F^\times \GL_2(\F) \backslash \GL_2(\A_\F)}{\bf X}(\underline{\kappa})\cdot(\varphi^I,(\varphi')^J, E^{[m']}(\chi,\chi^{-1}\omega_\itPi^{-1}\omega_{\itPi'}^{-1},\Phi))(g,g,g) \,dg^{\rm Tam}\\
&= D_\F^{-1}\zeta_\F(2)^{-1}\cdot L^{(\infty)}(m,\itPi \times \itPi' \times \chi)\cdot\prod_{v \nmid \infty}Z^*(W_v,W_v',f_{\chi_v,\chi_v^{-1}\omega_{\itPi,v}^{-1}\omega_{\itPi',v}^{-1},\Phi_v,s})\vert_{s=m}\\
&\times \prod_{v \in I}\sum_{m'+2m_1+2m_2=\ell_v-\ell_v'}c_{m_1,m_2}(\ell_v,\ell_v',m')Z(W_{(\ell_v,r),\psi_{\F_v}}^-,X_+^{m_1}\cdot W_{(\ell_v',r'),\psi_{\F_v}}^+,X_+^{m_2}\cdot f_{\chi_v,\chi_v^{-1}\omega_{\itPi,v}^{-1}\omega_{\itPi',v}^{-1},\Phi^{[m']},s})\vert_{s = m}\\
&\times \prod_{v \in J}\sum_{m'+2m_1+2m_2=\ell_v'-\ell_v}c_{m_1,m_2}(\ell_v,\ell_v',m')Z(X_+^{m_1}\cdot W_{(\ell_v,r),\psi_{\F_v}}^+,W_{(\ell_v',r'),\psi_{\F_v}}^-,X_+^{m_2}\cdot f_{\chi_v,\chi_v^{-1}\omega_{\itPi,v}^{-1}\omega_{\itPi',v}^{-1},\Phi^{[m']},s})\vert_{s = m}.
\end{align*}
Note that $\zeta_\F(2) \in D_\F^{1/2}\cdot \pi^{[\F:\Q]}\cdot \Q^\times$.
By (\ref{E:Galois Gauss sum}) and Lemmas \ref{L:Galois equiv. cusp form} and \ref{L:Galois equiv. local zeta}, we have
\begin{align}\label{E:Galois equiv. 4.8}
\begin{split}
W_{{}^\sigma\!\varphi,\psi_\F} &= \frac{\sigma(2\pi\sqrt{-1})^{-\Sigma_{v \in \Sigma_\F}(\ell_v+r)/2}}{(2\pi\sqrt{-1})^{-\Sigma_{v \in \Sigma_\F}(\ell_v+r)/2}}\prod_{v \in \Sigma_\F}W_{({}^\sigma\!\ell_v,r),\psi_{\F_v}}^+\prod_{v\nmid \infty} t_\sigma W_v,\\
W_{{}^\sigma\!\varphi',\psi_\F} &= \frac{\sigma(2\pi\sqrt{-1})^{-\Sigma_{v \in \Sigma_\F}(\ell_v'+r')/2}}{(2\pi\sqrt{-1})^{-\Sigma_{v \in \Sigma_\F}(\ell_v'+r')/2}}\prod_{v \in \Sigma_\F}W_{({}^\sigma\!\ell_v',r'),\psi_{\F_v}}^+\prod_{v\nmid \infty} t_\sigma W_v',
\end{split}
\end{align}
and 
\begin{align}\label{E:Galois equiv. 4.9}
\begin{split}
&\sigma \left( G(\chi)\prod_{v \nmid \infty}Z^*(W_v,W_v',f_{\chi_v,\chi_v^{-1}\omega_{\itPi,v}^{-1}\omega_{\itPi',v}^{-1},\Phi_v,s})\vert_{s=m}\right)\\
& = G({}^\sigma\!\chi)\prod_{v \nmid \infty}Z^*(t_\sigma W_v,t_\sigma W_v',f_{{}^\sigma\!\chi_v,{}^\sigma\!\chi_v^{-1}{}^\sigma\!\omega_{\itPi,v}^{-1}{}^\sigma\!\omega_{\itPi',v}^{-1},{}^\sigma\!\Phi_v,s})\vert_{s=m}.
\end{split}
\end{align}
By Lemma \ref{L:archimedean local zeta}, we have
\begin{align}\label{E:archimedean 4.10}
\begin{split}
&\prod_{v \in I}\sum_{m'+2m_1+2m_2=\ell_v-\ell_v'}c_{m_1,m_2}(\ell_v,\ell_v',m')Z(W_{(\ell_v,r),\psi_{\F_v}}^-,X_+^{m_1}\cdot W_{(\ell_v',r'),\psi_{\F_v}}^+,X_+^{m_2}\cdot f_{\chi_v,\chi_v^{-1}\omega_{\itPi,v}^{-1}\omega_{\itPi',v}^{-1},\Phi^{[m']},s})\vert_{s = m}\\
&\times \prod_{v \in J}\sum_{m'+2m_1+2m_2=\ell_v'-\ell_v}c_{m_1,m_2}(\ell_v,\ell_v',m')Z(X_+^{m_1}\cdot W_{(\ell_v,r),\psi_{\F_v}}^+,W_{(\ell_v',r'),\psi_{\F_v}}^-,X_+^{m_2}\cdot f_{\chi_v,\chi_v^{-1}\omega_{\itPi,v}^{-1}\omega_{\itPi',v}^{-1},\Phi^{[m']},s})\vert_{s = m} \\
& = C(\underline{\kappa},m')\cdot (\sqrt{-1})^{\sum_{v \in I}\ell_v + \sum_{v \in J}\ell_v'}\cdot \pi^{[\F:\Q]}\cdot (2\pi\sqrt{-1})^{-\sum_{v \in \Sigma_\F}(\ell_v+\ell_v'+3m')/2}.
\end{split}
\end{align}
Here $C(\underline{\kappa},m') \in \Q$ is the rational number
\begin{align*}
C(\underline{\kappa},m') &= \prod_{v \in I}\chi_v(-1)(-1)^{[\F:\Q]m'}2^{2[\F:\Q]-[\F:\Q]m'-\sum_{v \in I}\ell_v-\sum_{v \in J}\ell_v'}\prod_{v \in \Sigma_\F}\Gamma\left(\tfrac{\ell_v+\ell_v'+m'}{2}-1\right)\Gamma\left(\tfrac{|\ell_v-\ell_v'|+m'}{2}\right)\\
&\times\prod_{v \in \Sigma_\F}\sum_{m'+2m_1+2m_2=|\ell_v-\ell_v'|}(-1)^{m_1}c_{m_1,m_2}(\ell_v,\ell_v',m')
\end{align*}
Note that $C(\underline{\kappa},m')$ is non-zero as we will show in the proof of Lemma \ref{L:archimedean local period} below.
Finally, for each $v \nmid \infty$, we let the triplet $(W_v,W_v',\Phi_v)$ be chosen so that (cf.\,\cite[\S\,6.2]{CH2020})
\[
Z^*(W_v,W_v',f_{\chi_v,\chi_v^{-1}\omega_{\itPi,v}^{-1}\omega_{\itPi',v}^{-1},\Phi_v,s})=1.
\]
The theorem then follows from (\ref{E:cohomological interpretation})-(\ref{E:archimedean 4.10}).
This completes the proof.
\end{proof}

\section{Proof of main results}\label{S:5}

In this section, we prove the main results Theorems \ref{T: main thm}, \ref{C:main}, and \ref{T: main thm 2} of this paper.
In \S\,\ref{SS:Ichino}-\S\,\ref{SS:5.3}, we keep the notation of \S\,\ref{S:3} and 
let $\itPi = \bigotimes_v\pi_v$ be an irreducible cuspidal automorphic representation of $\GL_2(\A_\E)$ with central character $\omega_\itPi$, where $v$ runs through the places of $\F$. We assume the following conditions hold:
\begin{itemize}
\item $\omega_\itPi \vert_{\A_\F^\times}$ is trivial;
\item $\itPi$ is motivic of weight $(\underline{\kappa},\underline{r}) \in \Z_{\geq 1}[\Sigma_\E]\times \Z[\Sigma_\E]$;
\item $\underline{\kappa}$ satisfies the totally unbalanced condition
\[
2\max_{w \mid v}\{\kappa_w\} - \sum_{w \mid v}\kappa_w \geq 0 \quad
\]
for all $v \in \Sigma_\F$.
\end{itemize}

\subsection{Ichino's formula}\label{SS:Ichino}

In this subsection, we recall Ichino's central value formula \cite{Ichino2008} for $L(\tfrac{1}{2},\itPi,{\rm As})$ in terms of global trilinear period integral, which is a special case of the refined Gan--Gross--Prasad conjecture proposed by Ichino--Ikeda \cite{IchinoIkeda2010}.
Let $D$ be a totally indefinite quaternion algebra over $\F$ such that there exists an irreducible cuspidal automorphic representation $\itPi^D = \bigotimes_v\pi_v^D$ of $D^\times(\A_\F)$ associated to $\itPi$ by the Jacquet--Langlands correspondence. 
Define the functional $I^D \in {\rm Hom}_{D^\times(\A_\F)\times D^\times(\A_\F)}(\itPi^D \otimes (\itPi^D)^\vee,\C)$ by the global trilinear period integral
\begin{align}\label{E:global period}
I^D(\varphi_1\otimes\varphi_2) = \int_{\A_\F^\times D^\times(\F)\backslash D^\times(\A_\F)}\int_{\A_\F^\times D^\times(\F)\backslash D^\times(\A_\F)} \varphi_1(g_1)\varphi_2(g_2)\,dg_1^{\rm Tam}dg_2^{\rm Tam}.
\end{align}
Here $dg_1^{\rm Tam}$ and $dg_2^{\rm Tam}$ are the Tamagawa measures on $\A_\F^\times \backslash D^\times(\A_\F)$.
For each place $v$ of $\F$, we fix a non-zero $D^\times(\E_v)$-equivariant bilinear pairing 
\[
\<\,,\,\>_v:\pi_v^D \times (\pi_v^D)^\vee \longrightarrow \C.
\]
Let $dg_v$ be the Haar measure on $\F_v \backslash D^\times(\F_v)$ defined as follows:
\begin{itemize}
\item If $v$ is a finite place, let $dg_v$ be the Haar measure normalized so that ${\rm vol}(\o_{\F_v}^\times\backslash \o_{D_v}^\times,dg_v)=1$. Here $\o_{D_v}$ is a maximal order of $D(\F_v)$.
\item If $v$ is a real place, let 
\[
dg_v = \frac{dx_vdy_v}{|y_v|^2}dk_v
\]
for $g_v = {\bf n}(x_v){\bf a}(y_v)k_v$ with $x_v \in \R$, $y_v \in \R^\times$, and $k_v \in {\rm SO}(2)$.
Here $dx_v$ and $dy_v$ are the Lebesgue measures and $dk_v$ is the Haar measure on ${\rm SO}(2)$ such that ${\rm vol}({\rm SO}(2),dk_v)=2$.
\end{itemize}
Let $I_v^D \in {\rm Hom}_{D^\times(\F_v)\times D^\times(\F_v)}(\pi_v^D \otimes (\pi_v^D)^\vee,\C)$ be the functional define by the local trilinear period integrals
\begin{align}\label{E:local period integral}
I_v^D(\varphi_{1,v}\otimes\varphi_{2,v}) = \frac{\zeta_{\F_v}(2)}{\zeta_{\E_v}(2)}\cdot\frac{L(1,\pi_v,{\rm Ad})}{L(\tfrac{1}{2},\pi_v,{\rm As})}\cdot\int_{\F_v^\times\backslash D^\times(\F_v)}\<\pi_v^D(g_v)\varphi_{1,v},\varphi_{2,v}\>_v\,dg_v.
\end{align}
Here $\zeta_{\E_v}(s)$ and $\zeta_{\F_v}(s)$ are the local zeta functions of $\E_v$ and $\F_v$, respectively.
Note that $L(\tfrac{1}{2},\pi_v,{\rm As}) \neq 0$ (cf.\,\cite[Lemma 3.1]{Chen2020}) and the integral is absolutely convergent by \cite[Lemma 2.1]{Ichino2008}.
When $D$ is unramified at $v$, we also write $I_v=I_v^D$.
We normalize the pairings $\<\,,\,\>_v$ 
so that if $\varphi_1 = \bigotimes_v \varphi_{1,v} \in \itPi^D$ and $\varphi_2 = \bigotimes_v \varphi_{2,v} \in (\itPi^D)^\vee$, then $\<\varphi_{1,v},\varphi_{2,v}\>_v=1$ for almost all $v$ and 
\[
\int_{\A_\E^\times D^\times(\E)\backslash D^\times(\A_\E)}\varphi_1(g)\varphi_2(g)\,dg^{\rm Tam} = \prod_v \<\varphi_{1,v},\varphi_{2,v}\>_v.
\]
Here $dg^{\rm Tam}$ is the Tamagawa measure on $\A_\E^\times\backslash D^\times(\A_\E)$. 
Let $C_D$ be the constant such that
\[
C_D\cdot \prod_v dg_v
\]
is the Tamagawa measure on $\A_\F^\times\backslash D^\times(\A_\F)$. Then we have (cf.\,\cite[Lemma 6.1]{IP2018})
\begin{align}\label{E:ratio between measures}
C_D = \prod_{v \in \Sigma_D}(q_v-1)^{-1}\cdot D_\F^{-3/2}\cdot\zeta_\F(2)^{-1},
\end{align}
where $\Sigma_D$ is the set of places of $\F$ at which $D$ is ramified, $q_v$ is the cardinality of the residue field of $\F_v$, and $\zeta_\F$ is the completed Dedekind zeta function of $\F$.
We have the following central value formula of Ichino \cite{Ichino2008}.

\begin{thm}[Ichino]\label{T:Ichino formula}
As functionals in ${\rm Hom}_{D^\times(\A_\F)\times D^\times(\A_\F)}(\itPi^D \otimes (\itPi^D)^\vee,\C)$, we have
\[
I^D = \frac{C_D}{2^c}\cdot \frac{\zeta_\E(2)}{\zeta_\F(2)}\cdot\frac{L(\tfrac{1}{2},\itPi,{\rm As})}{L(1,\itPi,{\rm Ad})}\cdot\prod_v I_v^D.
\]
Here $\zeta_{\E}(s)$ and $\zeta_{\F}(s)$ are the completed Dedekind zeta functions of $\E$ and $\F$, respectively, and
\[
c = \begin{cases}
3 & \mbox{ if $\E = \F \times \F \times \F$},\\
2 & \mbox{ if $\E = \K \times \F$ for some totally real quadratic extension $\K$ of $\F$},\\
1 & \mbox{ if $\E$ is a field}.
\end{cases}
\]
\end{thm}




\subsection{Local trilinear period integrals}\label{SS:local period}

For each finite place $v$ of $\F$ and $\sigma \in {\rm Aut}(\C)$, we fix $\sigma$-linear isomorphisms
\[
t_{\sigma,v} : \pi_v^D \longrightarrow {}^\sigma\!\pi_v^D,\quad t_{\sigma,v}^\vee : (\pi_v^D)^\vee \longrightarrow {}^\sigma\!(\pi_v^D)^\vee.
\]
By abuse of notation, we denote by the same notation $I_v^D \in {\rm Hom}_{D^\times(\F_v)\times D^\times(\F_v)}({}^\sigma\!\pi_v^D \otimes {}^\sigma\!(\pi_v^D)^\vee,\C)$ the functional defined as in (\ref{E:local period integral}) with respect to the $D^\times(\F_v)$-equivariant bilinear pairing
\[
\<\,,\,\>_v : {}^\sigma\!\pi_v^D \times {}^\sigma\!(\pi_v^D)^\vee \longrightarrow \C
\]
defined by
\begin{align}\label{E:sigma pairing}
\<t_{\sigma,v}\varphi_{1,v},t_{\sigma,v}^\vee\varphi_{2,v}\>_v = \sigma \<\varphi_{1,v},\varphi_{2,v}\>_v
\end{align}
for $\varphi_{1,v} \in \pi_v^D$ and $\varphi_{2,v} \in (\pi_v^D)^\vee$.
In the following lemmas, we show that the local $L$-factors and the local trilinear period integrals satisfy the Galois equivariant property.

\begin{lemma}\label{L:Galois equiv. 1}
Let $v$ be a finite place of $\F$. For $\sigma \in {\rm Aut}(\C)$, we have
\[
\sigma L(1,\pi_v,{\rm Ad}) = L(1,{}^\sigma\!\pi_v,{\rm Ad}),\quad \sigma L(\tfrac{1}{2},\pi_v,{\rm As}) = L(\tfrac{1}{2},{}^\sigma\!\pi_v,{\rm As}).
\]
\end{lemma}

\begin{proof}
We prove the assertion for $L(\tfrac{1}{2},\pi_v,{\rm As})$ in the case when $\E_v$ is a field. The assertion for $L(1,\pi_v,{\rm Ad})$ or arbitrary $\E_v$ can be proved in a similar way and we omit it. 
Let $W_{\E_v}'$ and $W_{\F_v}'$ be the Weil--Deligne groups of $\E_v$ and $\F_v$, respectively. 
Fix $\sigma \in {\rm Aut}(\C)$.
Let $\chi_{\E_v,\sigma}$ and $\chi_{\F_v,\sigma}$ be the quadratic characters of $\E_v^\times$ and $\F_v^\times$, respectively, defined by
\[
\chi_{\E_v,\sigma} = \sigma(|\mbox{ }|_{\E_v}^{1/2})\cdot|\mbox{ }|_{\E_v}^{-1/2}, \quad \chi_{\F_v,\sigma} = \sigma(|\mbox{ }|_{\F_v}^{1/2})\cdot|\mbox{ }|_{\F_v}^{-1/2}.
\]
For $n \geq 1$, we identify the Langlands dual group ${}^L({\rm R}_{\E_v/\F_v}\GL_n)$ of ${\rm R}_{\E_v/\F_v}\GL_n$ with $\GL_n(\C)^3 \rtimes \Gal(\overline{\F}_v/\F_v)$ (cf.\,\cite[\S\,5]{Borel1979}), where the action of $\Gal(\overline{\F}_v/\F_v)$ on $\GL_n(\C)^3$ is the permutation of components induced by the natural homomorphism $\Gal(\overline{\F}_v/\F_v) \rightarrow \Gal(\E_v'/\F_v)$ with $\E_v'$ equals to the Galois closure of $\E_v /\F_v$.
We have a natural one to one correspondence described in \cite[Lemma 4.5]{Borel1979} between the set of $L$-parameters $W_{\F_v}' \rightarrow {}^L({\rm R}_{\E_v/\F_v}\GL_n)$ and the set of $n$-dimensional admissible representations $W_{\E_v}' \rightarrow \GL_n(\C)$. We will identify the two sets via this correspondence and note that the correspondence is compatible with $\sigma$-conjugation. We also identify characters of $\E_v^\times$ with $1$-dimensional admissible representations of $W_{\E_v}'$ via the local class field theory.
Let ${\rm As}$ be the Asai cube representation of ${}^L({\rm R}_{\E_v/\F_v}\GL_2)$ on $\C^2 \otimes \C^2 \otimes \C^2$ so that the restriction of ${\rm As}$ to $\GL_2(\C)^3$ is defined by
\[
{\rm As}(g_1,g_2,g_3) \cdot (v_1 \otimes v_2\otimes v_3) = (g_1\cdot v_1,g_v\cdot v_2,g_3 \cdot v_3)
\]
and the action of $\Gal(\overline{\F}_v/\F_v)$ on $\C^2 \otimes \C^2 \otimes \C^2$ is the permutation of components induced by the natural  homomorphism $\Gal(\overline{\F}_v/\F_v) \rightarrow \Gal(\E'_v/\F_v)$.
It is easy to verify that
\[
{}^\sigma\!({\rm As}\circ \phi) = {\rm As}\circ {}^\sigma\!\phi,\quad {\rm As}\circ (\phi\otimes\chi) = ({\rm As}\circ \phi)\otimes \chi\vert_{\F_v^\times}
\]
for $2$-dimensional admissible representation $\phi$ of $W_{\E_v}'$ and character $\chi$ of $\E_v^\times$. On the other hand, for any admissible representation $\itPhi : W_{\F_v}' \rightarrow \GL_8(\C)$, by \cite[Lemme 4.6]{Clozel1990} and \cite[Propri\'et\'e 3, \S\,7]{Henniart2001} we have
\[
{}^\sigma\! L(s+\tfrac{1}{2}, \itPhi)  = L(s+\tfrac{1}{2}, {}^\sigma\!\itPhi\otimes \chi_{\F_v,\sigma})
\]
as rational functions in $q_v^{-s}$.
Let $\phi_{\pi_v} : W_{\E_v}' \rightarrow \GL_2(\C)$ be the admissible representation associated to $\pi_v$ by the local Langlands correspondence established in \cite{Henniart1999} and \cite{HT2001}. Note that we have (cf.\,\cite[Propri\'et\'e 3, \S\,7]{Henniart2001})
\[
{}^\sigma\!\phi_{\pi_v} = \phi_{{}^\sigma\!\pi_v}\otimes\chi_{\E_v,\sigma}.
\]
We conclude that
\begin{align*}
{}^\sigma\! L(s+\tfrac{1}{2},\pi_v,{\rm As}) 
& = {}^\sigma\! L(s+\tfrac{1}{2}, {\rm As}\circ \phi_{\pi_v})\\
& = L(s+\tfrac{1}{2},{}^\sigma\!({\rm As}\circ \phi_{\pi_v})\otimes \chi_{\F_v,\sigma})\\
& = L(s+\tfrac{1}{2},({\rm As}\circ {}^\sigma\!\phi_{\pi_v})\otimes \chi_{\F_v,\sigma})\\
& = L(s+\tfrac{1}{2}, ({\rm As}\circ \phi_{{}^\sigma\!\pi_v})\otimes \chi_{\E_v,\sigma}\vert_{\F_v^\times}\cdot\chi_{\F_v,\sigma})\\
& = L(s+\tfrac{1}{2}, {\rm As}\circ \phi_{{}^\sigma\!\pi_v})\\
& = L(s+\tfrac{1}{2}, {}^\sigma\!\pi_v, {\rm As}).
\end{align*}
We obtain the assertion by evaluating at $s=0$. 
This completes the proof.
\end{proof}

\begin{lemma}\label{L:Galois equiv. 2}

Let $v$ be a finite place of $\F$.
For $\sigma \in {\rm Aut}(\C)$, we have
\[
\sigma I_v^D(\varphi_{1,v}\otimes \varphi_{2,v}) =  I_v^D(t_{\sigma,v}\varphi_{1,v}\otimes t_{\sigma,v}^\vee\varphi_{2,v})
\]
for $\varphi_{1,v} \in \pi_v^D$ and $\varphi_{2,v} \in (\pi_v^D)^\vee$.
\end{lemma}

\begin{proof}
Let $\varphi_{1,v} \in \pi_v^D$ and $\varphi_{2,v} \in (\pi_v^D)^\vee$. Note that by definition (\ref{E:sigma pairing}) we have
\[
\sigma \<\pi_v^D(g_v)\varphi_{1,v},\varphi_{2,v}\>_v = \<{}^\sigma\!\pi_v^D(g_v)t_{\sigma,v}\varphi_{1,v},t_{\sigma,v}^\vee\varphi_{2,v}\>_v
\]
for $g_v \in D^\times(\E_v)$ and $\sigma \in {\rm Aut}(\C)$.
Together with Lemma \ref{L:Galois equiv. 1}, it suffices to show that
\begin{align}\label{E:4.3}
\sigma\left(\int_{\F_v^\times\backslash D^\times(\F_v)}\<\pi_v^D(g_v)\varphi_{1,v},\varphi_{2,v}\>_v\,dg_v \right) = \int_{\F_v^\times\backslash D^\times(\F_v)}\sigma\<\pi_v^D(g_v)\varphi_{1,v},\varphi_{2,v}\>_v\,dg_v
\end{align}
for all $\sigma \in {\rm Aut}(\C)$.
If $D$ is ramified at $v$, then $\F_v\backslash D^\times(\F_v)$ is compact. Therefore the local period integral is a finite sum and equality (\ref{E:4.3}) holds trivially. Suppose $D$ is unramified at $v$ and identify $D^\times$ with $\GL_2$.
Let $K_v = \o_{\F_v}^\times \backslash \GL_2(\o_{\F_v})$. By the Cartan decomposition, we have
\begin{align*}
&\int_{\F_v^\times\backslash \GL_2(\F_v)}\<\pi_v(g_v)\varphi_{1,v},\varphi_{2,v}\>_v\,dg_v \\
&= \int_{K_v}\int_{K_v}\int_{\F_v^\times} \<\pi_v(k_{1,v}{\bf a}(t_v)k_{2,v})\varphi_{1,v},\varphi_{2,v}\>_v\mathbb{I}_{\o_{\F_v}}(t_v){\rm vol}(K_v {\bf a}(t_v) K_v,dg_v)\,d^\times t_vdk_{1,v}dk_{2,v},
\end{align*}
where $dk_{1,v}$, $dk_{2,v}$, and $d^\times t_v$ are Haar measures normalized so that
\[
{\rm vol}(K_v,dk_{1,v}) = {\rm vol}(K_v,dk_{2,v}) = {\rm vol}(\o_{\F_v}^\times,d^\times t_v) =1.
\]
Note that
\[
{\rm vol}(K_v {\bf a}(t_v) K_v,dg_v) = \begin{cases}
1 & \mbox{ if $|t_v|_{\F_v}=1$},\\
|t_v|_{\F_v}^{-1}(1+q_v^{-1}) & \mbox{ if $|t_v|_{\F_v}<1$}
\end{cases}
\]
for $t_v \in \o_{\F_v} \smallsetminus \{0\}$.
It is well-known that there exist characters $\chi_{1,v}$ and $\chi_{2,v}$ of $\E_v^\times$ depending only on $\pi_v$ and locally constant functions $f_{1,v}$ and $f_{2,v}$ on $\E_v \times \GL_2(\o_{\E_v}) \times \GL_2(\o_{\E_v})$ such that
\[
\<\pi_v(k_{1,v}{\bf a}(t_v)k_{2,v})\varphi_{1,v},\varphi_{2,v}\>_v = \chi_{1,v}(t_v)f_{1,v}(t_v,k_{1,v},k_{2,v}) + \chi_{2,v}(t_v)f_{2,v}(t_v,k_{1,v},k_{2,v})
\]
for $(t_v,k_{1,v},k_{2,v}) \in (\o_{\E_v} \smallsetminus \{0\}) \times \GL_2(\o_{\E_v}) \times \GL_2(\o_{\E_v})$.
For a character $\chi$ of $\F_v^\times$ and a locally constant function $f : \F_v \rightarrow \C$ with compact support, let $Z(\chi,f)$ be the Tate integral defined by
\[
Z(\chi,f) = \int_{\F_v^\times}\chi(t_v)f(t_v)\,d^\times t_v.
\]
It is easy to show that the integral converges absolutely  when $|\chi| = |\mbox{ }|_{\F_v}^\lambda$ for some $\lambda>0$ and we have
\begin{align}\label{E:4.4}
\sigma Z(\chi,f) = Z( {}^\sigma\!\chi,{}^\sigma\!f)
\end{align}
for all $\sigma \in {\rm Aut}(\C)$ when both sides are absolutely convergnet.
We have
\begin{align}\label{E:4.5}
\begin{split}
&\int_{\F_v^\times\backslash \GL_2(\F_v)}\<\pi_v(g_v)\varphi_{1,v},\varphi_{2,v}\>_v\,dg_v\\
& = [K_v:U_v]^{-2}\sum_{k_{1,v} \in K_v / U_v}\sum_{k_{2,v} \in K_v / U_v} \left[Z(\chi_{1,v}\vert_{\F_v^\times}\cdot|\mbox{ }|_{\F_v}^{-1},f^{(1)}_{k_{1,v},k_{2,v}}) + Z(\chi_{2,v}\vert_{\F_v^\times}\cdot|\mbox{ }|_{\F_v}^{-1},f^{(2)}_{k_{1,v},k_{2,v}})\right].
\end{split}
\end{align}
Here $U_v$ is an open compact normal subgroup of $K_v$ such that both $\varphi_{1,v}$ and $\varphi_{2,v}$ are $U_v$-invariant and 
\[
f^{(i)}_{k_{1,v},k_{2,v}}(t_v) = \begin{cases}
0 & \mbox{ if $t_v \notin \o_{\F_v}$},\\
f_{i,v}(t_v,k_{1,v},k_{2,v}) & \mbox{ if $|t_v|_{\F_v}=1$},\\
(1+q_v^{-1})f_{i,v}(t_v,k_{1,v},k_{2,v}) & \mbox{ if $|t_v|_{\F_v}<1$}
\end{cases}
\]
for $i=1,2$ and $t_v \in \F_v$. We remark that $|\chi_{i,v}\vert_{\F_v^\times}| = |\mbox{ }|_{\F_v}^{\lambda_i}$ for some $\lambda_i>1$ by the result of Kim--Shahidi \cite{KS2002}. Hence the above Tate integrals are absolutely convergent. Equality (\ref{E:4.3}) then follows immediately from (\ref{E:4.4}) and (\ref{E:4.5}). 
This completes the proof.
\end{proof}

Now we consider the archimedean local trilinear period integrals. The totally unbalanced condition implies that $D$ is unramified at $v$ and $\pi_v^D=\pi_v$ for each $v \in \Sigma_\F$.
Let $\pi_\infty = \bigotimes_{v \in \Sigma_\F}\pi_v$ be a representation of $D^\times(\E_\infty) = \GL_2(\R)^{\Sigma_\E}$. Define 
$I_\infty \in {\rm Hom}_{\GL_2(\F_\infty)\times \GL_2(\F_\infty)}(\pi_\infty \otimes \pi_\infty^\vee,\C)$ and $\<\,,\,\>_\infty : \pi_\infty \times \pi_\infty^\vee \rightarrow \C$ by
\[
I_\infty = \bigotimes_{v \in \Sigma_\F}I_v,\quad \<\,,\,\>_\infty = \prod_{v \in \Sigma_\F}\<\,,\,\>_v.
\]
We recall in the following lemma our previous calculation of 
local trilinear period integral.

\begin{lemma}\label{L:archimedean local period}
Let $\varphi_{(\underline{\kappa},\underline{r})} \in \pi_\infty,\varphi_{(\underline{\kappa},-\underline{r})} \in \pi_\infty^\vee$ be non-zero vectors of weight $\underline{\kappa}$ and ${\bf X}(\underline{\kappa}) \in U(\frak{gl}_{2,\C}^{\Sigma_\E})$ be the differential operator defined in (\ref{E:differential operator}).
We have
\begin{align*}
&\frac{I_\infty({\bf X}(\underline{\kappa})\cdot\pi_\infty(\tau_{I_{\underline{\kappa}}})\varphi_{(\underline{\kappa},\underline{r})}\otimes{\bf X}(\underline{\kappa})\cdot\pi_\infty^\vee(\tau_{I_{\underline{\kappa}}})\varphi_{(\underline{\kappa},-\underline{r})})}{\<\varphi_{(\underline{\kappa},\underline{r})},\pi_\infty^\vee(\tau_{\Sigma_\E})\varphi_{(\underline{\kappa},-\underline{r})}\>_\infty} \in (2\pi \sqrt{-1})^{\sum_{v \in \Sigma_\F}(2\max_{w \mid v}\{\kappa_w\} - \sum_{w \mid v}\kappa_w) }\cdot \Q^\times.
\end{align*}
Here $\tau_{I_{\underline{\kappa}}}, \tau_{\Sigma_\E} \in \GL_2(\R)^{\Sigma_\E}$ are defined in (\ref{E:tau_I}).
\end{lemma}

\begin{proof}
Let $v \in \Sigma_\F$ and $v^{(1)},v^{(2)},v^{(3)} \in \Sigma_\E$ be the extensions of $v$. 
Note that 
\[
\pi_v = \boxtimes_{i=1}^3 (D(\kappa_{v^{(i)}})\otimes |\mbox{ }|^{r_{v^{(i)}/2}}),\quad \pi_v^\vee = \boxtimes_{i=1}^3 (D(\kappa_{v^{(i)}})\otimes |\mbox{ }|^{-r_{v^{(i)}/2}}).
\]
Let ${\bf v}_{v^{(i)}} \in (D(\kappa_{v^{(i)}})\otimes |\mbox{ }|^{r_{v^{(i)}/2}})$ and ${\bf v}_{v^{(i)}}^\vee \in (D(\kappa_{v^{(i)}})\otimes |\mbox{ }|^{-r_{v^{(i)}/2}})$ be non-zero vectors of weight $\kappa_{v^{(i)}}$ for $i=1,2,3$. 
Put ${\bf v}_v = {\bf v}_{v^{(1)}}\otimes{\bf v}_{v^{(2)}} \otimes{\bf v}_{v^{(3)}}$ and ${\bf v}_v^\vee = {\bf v}_{v^{(1)}}^\vee\otimes{\bf v}_{v^{(2)}}^\vee \otimes{\bf v}_{v^{(3)}}^\vee$.
We may assume $v^{(1)} = \tilde{v}(\underline{\kappa})$.
Then
\begin{align*}
&\frac{I_v({\bf X}(\kappa_{v^{(1)}},\kappa_{v^{(2)}},\kappa_{v^{(3)}})\cdot \pi_v((\tau_{I_{\underline{\kappa}}})_v){\bf v}_v \otimes {\bf X}(\kappa_{v^{(1)}},\kappa_{v^{(2)}},\kappa_{v^{(3)}})\cdot \pi_v^\vee((\tau_{I_{\underline{\kappa}}})_v){\bf v}_v^\vee)}{\<{\bf v}_v,\pi_v({\bf a}(-1)){\bf v}_v^\vee\>_v}\\
&= \zeta_{\R}(2)^{-2}\cdot \frac{L(1,\pi_v,{\rm Ad})}{L(\tfrac{1}{2},\pi_v,{\rm As})}\cdot\sum_{m_1,m_2,m_1',m_2'} c_{m_1,m_2}(\kappa_{v^{(1)}},\kappa_{v^{(2)}},\kappa_{v^{(3)}})c_{m_1',m_2'}(\kappa_{v^{(1)}},\kappa_{v^{(2)}},\kappa_{v^{(3)}})\\
&\times\int_{\R^\times \backslash \GL_2(\R)} \frac{\<\pi_v(g(\tau_{I_{\underline{\kappa}}})_v)({\bf v}_{v^{(1)}} \otimes X_+^{m_1}{\bf v}_{v^{(2)}}\otimes X_+^{m_2}{\bf v}_{v^{(3)}}),\pi_v^\vee((\tau_{I_{\underline{\kappa}}})_v)({\bf v}_{v^{(1)}}^\vee \otimes X_+^{m_1}{\bf v}_{v^{(2)}}^\vee\otimes X_+^{m_2}{\bf v}_{v^{(3)}}^\vee)\>_v}{\<{\bf v}_v,\pi_v({\bf a}(-1)){\bf v}_v^\vee\>_v}.
\end{align*}
Here $m_1,m_2,m_1',m_2'$ runs through non-negative integers such that 
\[
2m_1+2m_2 = 2m_1'+2m_2' = \kappa_{v^{(1)}}- \kappa_{v^{(2)}}- \kappa_{v^{(3)}}.
\]
By Lemma \ref{L:differential operator}, the vector
\[
\sum_{2m_1+2m_2 = \kappa_{v^{(1)}}-\kappa_{v^{(2)}}-\kappa_{v^{(3)}}} c_{m_1,m_2}(\kappa_{v^{(1)}},\kappa_{v^{(2)}},\kappa_{v^{(3)}})(X_+^{m_1}{\bf v}_{v^{(2)}} \otimes X_+^{m_2}{\bf v}_{v^{(3)}})
\]
has weight $\kappa_{v^{(1)}}$ and the $(\frak{gl}_2,{\rm O}(2))$-module generated by it under the diagonal action is isomorphic to $D(\kappa_{v^{(1)}}) \otimes |\mbox{ }|^{r_{v^{(1)}}/2}$.
It then follows from the Schur orthogonality relations that the above integral is non-zero.
On the other hand, by \cite[Proposition 4.1 and Corollary 4.4]{CC2017}, the above integral is equal to 
\[
(2\pi\sqrt{-1})^{\kappa_{v^{(1)}}- \kappa_{v^{(2)}}- \kappa_{v^{(3)}}}\cdot 2^{\kappa_{v^{(1)}}- \kappa_{v^{(2)}}- \kappa_{v^{(3)}}+1}\cdot \left(\sum_{2m_1+2m_2 = \kappa_{v^{(1)}}- \kappa_{v^{(2)}}- \kappa_{v^{(3)}}}(-1)^{m_1}c_{m_1,m_2}(\kappa_{v^{(1)}},\kappa_{v^{(2)}},\kappa_{v^{(3)}})\right)^2.
\]
We conclude that
\begin{align*}
&\frac{I_\infty({\bf X}(\underline{\kappa})\cdot\pi_\infty(\tau_{I_{\underline{\kappa}}})\varphi_{(\underline{\kappa},\underline{r})}\otimes{\bf X}(\underline{\kappa})\cdot\pi_\infty^\vee(\tau_{I_{\underline{\kappa}}})\varphi_{(\underline{\kappa},-\underline{r})})}{\<\varphi_{(\underline{\kappa},\underline{r})},\pi_\infty^\vee(\tau_{\Sigma_\E})\varphi_{(\underline{\kappa},-\underline{r})}\>_\infty}\\
& = \prod_{v \in \Sigma_\F}\frac{I_v({\bf X}(\kappa_{v^{(1)}},\kappa_{v^{(2)}},\kappa_{v^{(3)}})\cdot \pi_v((\tau_{I_{\underline{\kappa}}})_v){\bf v}_v \otimes {\bf X}(\kappa_{v^{(1)}},\kappa_{v^{(2)}},\kappa_{v^{(3)}})\cdot \pi_v^\vee((\tau_{I_{\underline{\kappa}}})_v){\bf v}_v^\vee)}{\<{\bf v}_v,\pi_v({\bf a}(-1)){\bf v}_v^\vee\>_v}\\
& \in (2\pi \sqrt{-1})^{\sum_{v \in \Sigma_\F}(2\max_{w \mid v}\{\kappa_w\} - \sum_{w \mid v}\kappa_w) }\cdot \Q^\times.
\end{align*}
This completes the proof.
\end{proof}

Let $v$ be a place of $\F$. By the results of Prasad \cite{Prasad1990}, \cite{Prasad1992} and Loke \cite{Loke2001}. We have
\[
{\rm dim}_\C\,{\rm Hom}_{D^\times(\F_v)}(\pi_v^D,\C) \leq 1.
\]
\begin{lemma}\label{L:trilinear form}
Let $v$ be a place of $\F$. Then
${\rm Hom}_{D^\times(\F_v)}(\pi_v^D,\C) \neq 0$
if and only if $I_v^D \neq 0$.
\end{lemma}

\begin{proof}
When $v \in \Sigma_\F$, the assertion follows from Lemma \ref{L:archimedean local period}. Assume $v$ is a finite place. The assertion can be proved word by word following the proof of \cite[Corollary 9.6]{Chen2017} except we replace $\R$ and \cite[Theorem 4.12]{LZ1997} therein by $\F_v$ and \cite[Corollary 3.7]{KR1992}, respectively.
\end{proof}

\subsection{Proof of Theorem \ref{T: main thm}}\label{SS:5.3}

Let 
\[
\varepsilon(s,\itPi,{\rm As}) = \prod_v \varepsilon(s,\pi_v,{\rm As},\psi_v)
\]
be the global Asai $\varepsilon$-factor of $\itPi$, where $\bigotimes_v\psi_v$ is a non-trivial additive character of $\F \backslash \A_\F$ and $\varepsilon(s,\pi_v,{\rm As},\psi_v)$ is the local Asai $\varepsilon$-factor of $\pi_v$ with respect to $\psi_v$ defined via the Weil--Deligne representation. Note that $\varepsilon(s,\pi_v,{\rm As},\psi_v)=1$ for all finite places $v$ of $\F$ such that $\E_v$ is unramified over $\F_v$, $\pi_v$ is unramified, and $\psi_v$ is of conductor $\o_{\F_v}$.
By our assumption that $\omega_{\itPi}\vert_{\A_\F^\times}$ is trivial, we have $\varepsilon(\tfrac{1}{2},\pi_v,{\rm As},\psi_v) \in \{\pm1\}$ and is independent of the choice of $\psi_v$. We write 
\[
\varepsilon(\tfrac{1}{2},\pi_v,{\rm As})=\varepsilon(\tfrac{1}{2},\pi_v,{\rm As},\psi_v).
\]
Recall the global root number $\varepsilon(\itPi,{\rm As})$ is defined by
\[
\varepsilon(\itPi,{\rm As}) = \varepsilon(\tfrac{1}{2},\itPi,{\rm As}) \in \{\pm 1\}.
\]
On the other hand, analogous to the proof of Lemma \ref{L:Galois equiv. 1}, we can show that
\[
\sigma \varepsilon(\tfrac{1}{2},\pi_v,{\rm As}) = \varepsilon(\tfrac{1}{2},{}^\sigma\!\pi_v,{\rm As})
\]
for all finite places $v$ of $\F$ and $\sigma \in {\rm Aut}(\C)$.
In particular, we have
\[
\varepsilon(\itPi,{\rm As}) = \varepsilon({}^\sigma\!\itPi,{\rm As})
\]
for all $\sigma \in {\rm Aut}(\C)$.
We see that assertion (1) of Theorem \ref{T: main thm} follows immediately from the expected functional equation
\begin{align}\label{E:global f.e.}
L(s,\itPi,{\rm As}) = \varepsilon(s,\itPi,{\rm As}) L(1-s,\itPi,{\rm As}).
\end{align}
Indeed, let 
\[
L_{\rm PSR}(s,\itPi,{\rm As}) = \prod_v L_{\rm PSR}(s,\pi_v,{\rm As}), \quad \varepsilon_{\rm PSR}(s,\itPi,{\rm As}) = \prod_v \varepsilon_{\rm PSR}(s,\pi_v,{\rm As},\psi_v)
\]
be the global Asai $L$-function and $\varepsilon$-factor of $\itPi$ defined by the Rankin--Selberg method as well as the local zeta integrals developed by Piatetski-Shapiro--Rallis \cite{PSR1987} and Ikeda \cite{Ikeda1989}.
In the Rankin--Selberg context, the functional equation
\[
L_{\rm PSR}(s,\itPi,{\rm As}) = \varepsilon_{\rm PSR}(s,\itPi,{\rm As}) L_{\rm PSR}(1-s,\itPi,{\rm As})
\]
holds and is a direct consequence of the functional equation of Siegel Eisenstein seires.
On the other hand, by \cite{KS2002} and \cite[Corollary 1.4]{Chen2020}, we have
\[
L(s,\pi_v,{\rm As}) = L_{\rm PSR}(s,\pi_v, {\rm As}), \quad \varepsilon(s,\pi_v,{\rm As},\psi_v) = \varepsilon_{\rm PSR}(s,\pi_v,{\rm As},\psi_v)
\]
for all places $v$ of $\F$.
Hence the functional equation (\ref{E:global f.e.}) holds.

Now we assume $\varepsilon(\itPi,{\rm As})=1$. Let $\K$ be the quadratic discriminant algebra of $\E / \F$ and $\omega_{\mathbb{K}/\F} = \prod_v \omega_{\K_v/ \F_v}$ the quadratic character of $\F^\times\backslash \A_\F^\times$ associated to $\mathbb{K}/\F$ by class field theory. By the assumption $\varepsilon(\itPi,{\rm As})=1$, there exists a unique quaternion algebra $D$ over $\F$ such that $D$ is ramified at $v$ if and only if
\[
\varepsilon(\tfrac{1}{2},\pi_v,{\rm As})\cdot\omega_{\K_v/\F_v}(-1)=-1.
\]
In particular, it follows from the totally unbalanced condition that $D$ is totally indefinite. 
By the results of Prasad \cite{Prasad1990}, \cite{Prasad1992} and Loke \cite{Loke2001}, the above sign condition implies that there exists an irreducible cuspidal automorphic representation $\itPi^D = \bigotimes_v \pi_v^D$ of $D^\times(\A_\E)$ associated to $\itPi$ by the Jacquet--Langlands correspondence such that
\[
{\rm Hom}_{D^\times(\F_v)}(\pi_v^D,\C) \neq 0
\]
for all places $v$ of $\F$. 
Fix non-zero vectors $\varphi_{(\underline{\kappa},\underline{r})} \in \pi_\infty = \bigotimes_{v \in \Sigma_\F}\pi_v$ and $\varphi_{(\underline{\kappa},-\underline{r})} \in \pi_\infty^\vee = \bigotimes_{v \in \Sigma_\F}\pi_v^\vee$ of weight $\underline{\kappa}$, that is, 
\begin{align*}
\pi_\infty(\underline{k}_\theta)\varphi_{(\underline{\kappa},\underline{r})} &= \prod_{w \in \Sigma_\E}e^{\sqrt{-1}\,\kappa_w \theta_w}\cdot \varphi_{(\underline{\kappa},\underline{r})},\\
\pi_\infty^\vee(\underline{k}_\theta) \varphi_{(\underline{\kappa},-\underline{r})} &= \prod_{w \in \Sigma_\E} e^{\sqrt{-1}\,\kappa_w \theta_w} \cdot \varphi_{(\underline{\kappa},-\underline{r})}
\end{align*}
for $\underline{k}_\theta = (k_{\theta_w})_{w \in \Sigma_\E}\in {\rm SO}(2)^{\Sigma_\E}$.
Fix $\sigma \in {\rm Aut}(\C)$. Let $\varphi_{({}^\sigma\!\underline{\kappa},\underline{r})} \in {}^\sigma\!\pi_\infty$ and $\varphi_{({}^\sigma\!\underline{\kappa},-\underline{r})} \in {}^\sigma\!\pi_\infty^\vee$ be the vectors of weight ${}^\sigma\!\underline{\kappa}$ such that if $\varphi_1 = \varphi_{(\underline{\kappa},\underline{r})} \otimes\left( \bigotimes_{v\nmid \infty}\varphi_{1,v}\right) \in \itPi_{\rm hol}^D$ and $\varphi_2 = \varphi_{(\underline{\kappa},-\underline{r})} \otimes \left(\bigotimes_{v\nmid \infty}\varphi_{2,v}\right) \in (\itPi_{\rm hol}^D)^\vee$, then 
\[
{}^\sigma\!\varphi_1 = \varphi_{({}^\sigma\!\underline{\kappa},\underline{r})}\otimes\left( \bigotimes_{v\nmid \infty}t_{\sigma,v}\varphi_{1,v}\right) \in {}^\sigma\!\itPi_{\rm hol}^D,
\quad {}^\sigma\!\varphi_2 = \varphi_{({}^\sigma\!\underline{\kappa},\underline{-r})}\otimes\left( \bigotimes_{v\nmid \infty}t_{\sigma,v}^\vee\varphi_{2,v}\right) \in {}^\sigma\!(\itPi_{\rm hol}^D)^\vee
\]
Here the $\sigma$-linear isomorphisms $t_{\sigma,v}$, $t_{\sigma,v}^\vee$ for finite places $v$ of $\F$ are fixed in \S\,\ref{SS:local period} and
\begin{align*}
&\itPi_{\rm hol}^D \longrightarrow {}^\sigma\!\itPi_{\rm hol}^D,\quad \varphi \longmapsto {}^\sigma\!\varphi,\\
&(\itPi_{\rm hol}^D)^\vee \longrightarrow {}^\sigma\!(\itPi_{\rm hol}^D)^\vee,\quad \varphi^\vee \longmapsto {}^\sigma\!\varphi^\vee
\end{align*}
are defined in (\ref{E:sigma iso.}). 
By Lemma \ref{L:trilinear form}, for each finite place $v$ of $\F$, there exist $\varphi_{1,v} \in \pi_v^D$ and $\varphi_{2,v} \in (\pi_v^D)^\vee$ such that 
\[
I_v^D(\varphi_{1,v}\otimes \varphi_{2,v}) \neq 0.
\]
Moreover, when $v$ is a finite place such that $\E_v$ is unramified over $\F_v$ and $\pi_v$ is unramified, then $D$ is unramified at $v$ and we choose $\varphi_{1,v}$ and $\varphi_{2,v}$ be non-zero $\GL_2(\o_{\E_v})$-invariant vectors. In this case, by \cite[Lemma 2.2]{Ichino2008}, we have
\[
I_v^D(\varphi_{1,v}\otimes \varphi_{2,v}) = \<\varphi_{1,v},\varphi_{2,v}\>_v.
\]
With this local choice, we put $\varphi_1 = \varphi_{(\underline{\kappa},\underline{r})} \otimes\left( \bigotimes_{v\nmid \infty}\varphi_{1,v}\right) \in \itPi_{\rm hol}^D$ and $\varphi_2 = \varphi_{(\underline{\kappa},-\underline{r})} \otimes \left(\bigotimes_{v\nmid \infty}\varphi_{2,v}\right) \in (\itPi_{\rm hol}^D)^\vee$.
We also fix $\varphi_3 = \varphi_{(\underline{\kappa},\underline{r})} \otimes\left( \bigotimes_{v\nmid \infty}\varphi_{3,v}\right) \in \itPi_{\rm hol}^D$ and $\varphi_4 = \varphi_{(\underline{\kappa},-\underline{r})} \otimes \left(\bigotimes_{v\nmid \infty}\varphi_{4,v}\right) \in (\itPi_{\rm hol}^D)^\vee$ such that
\[
\<\varphi_3,\varphi_4\> = \int_{Z_G(\A)G(\Q)\backslash G(\A)}\varphi_3(g)\varphi_4(g\cdot t_{\Sigma_\E})\,dg^{\rm Tam} \neq 0.
\]
Let ${\bf X}(\underline{\kappa}) \in U(\frak{gl}_{2,\C}^{\Sigma_\E})$ be the differential operator defined in (\ref{E:differential operator}). By Ichino's formula Theorem \ref{T:Ichino formula}, we have
\begin{align*}
&\frac{I^D({\bf X}(\underline{\kappa})\cdot \varphi_1^{I_{\underline{\kappa}}} \otimes {\bf X}(\underline{\kappa})\cdot\varphi_2^{I_{\underline{\kappa}}})}{\<\varphi_3,\varphi_4\>} \\
&= \frac{C_D}{2^c}\cdot \frac{\zeta_\E(2)}{\zeta_\F(2)}\cdot\frac{L(\tfrac{1}{2},\itPi,{\rm As})}{L(1,\itPi,{\rm Ad})}\cdot\prod_{v\nmid \infty} \frac{I_v^D(\varphi_{1,v}\otimes\varphi_{2,v})}{\<\varphi_{3,v},\varphi_{4,v}\>_v}
\cdot\frac{I_\infty({\bf X}(\underline{\kappa})\cdot\pi_\infty(\tau_{I_{\underline{\kappa}}})\varphi_{(\underline{\kappa},\underline{r})}\otimes{\bf X}(\underline{\kappa})\cdot\pi_\infty^\vee(\tau_{I_{\underline{\kappa}}})\varphi_{(\underline{\kappa},-\underline{r})})}{\<\varphi_{(\underline{\kappa},\underline{r})},\pi_\infty^\vee(\tau_{\Sigma_\E})\varphi_{(\underline{\kappa},-\underline{r})}\>_\infty},\\
&\frac{I^D({\bf X}({}^\sigma\!\underline{\kappa})\cdot({}^\sigma\!\varphi_1)^{I_{{}^\sigma\!\underline{\kappa}}} \otimes {\bf X}({}^\sigma\!\underline{\kappa})\cdot({}^\sigma\!\varphi_2)^{I_{{}^\sigma\!\underline{\kappa}}})}{\<{}^\sigma\!\varphi_3,{}^\sigma\!\varphi_4\>} \\
&= \frac{C_D}{2^c}\cdot \frac{\zeta_\E(2)}{\zeta_\F(2)}\cdot\frac{L(\tfrac{1}{2},{}^\sigma\!\itPi,{\rm As})}{L(1,{}^\sigma\!\itPi,{\rm Ad})}\cdot\prod_{v\nmid \infty} \frac{I_v^D(t_{\sigma,v}\varphi_{1,v}\otimes t_{\sigma,v}^\vee\varphi_{2,v})}{\<t_{\sigma,v}\varphi_{3,v},t_{\sigma,v}^\vee\varphi_{4,v}\>_v}\\
&\times\frac{I_\infty({\bf X}({}^\sigma\!\underline{\kappa})\cdot{}^\sigma\!\pi_\infty(\tau_{I_{{}^\sigma\!\underline{\kappa}}})\varphi_{({}^\sigma\!\underline{\kappa},\underline{r})}\otimes{\bf X}({}^\sigma\!\underline{\kappa})\cdot{}^\sigma\!\pi_\infty^\vee(\tau_{I_{{}^\sigma\!\underline{\kappa}}})\varphi_{({}^\sigma\!\underline{\kappa},-\underline{r})})}{\<\varphi_{({}^\sigma\!\underline{\kappa},\underline{r})},{}^\sigma\!\pi_\infty^\vee(\tau_{\Sigma_\E})\varphi_{({}^\sigma\!\underline{\kappa},-\underline{r})}\>_\infty}.
\end{align*}
Here $\tau_{I_{\underline{\kappa}}}, \tau_{\Sigma_\E}$ and $\varphi_1^{I_{\underline{\kappa}}}, \varphi_2^{I_{\underline{\kappa}}}$ are defined in (\ref{E:tau_I}) and (\ref{E:phi^I}), respectively.
By Lemma \ref{L:trace} and Proposition \ref{P:differential operator}, we have the Galois equivariant property of the global trilinear period integral that
\begin{align}\label{E:5.3.1}
\sigma\left(\frac{I^D({\bf X}(\underline{\kappa})\cdot\varphi_1^{I_{\underline{\kappa}}} \otimes {\bf X}(\underline{\kappa})\cdot\varphi_2^{I_{\underline{\kappa}}})}{\Omega^{I_{\underline{\kappa}}}(\itPi^D)\cdot\Omega^{I_{\underline{\kappa}}}((\itPi^D)^\vee)} \right) = \frac{I^D({\bf X}({}^\sigma\!\underline{\kappa})\cdot({}^\sigma\!\varphi_1)^{I_{{}^\sigma\!\underline{\kappa}}} \otimes {\bf X}({}^\sigma\!\underline{\kappa})\cdot({}^\sigma\!\varphi_2)^{I_{{}^\sigma\!\underline{\kappa}}})}{\Omega^{I_{{}^\sigma\!\underline{\kappa}}}({}^\sigma\!\itPi^D)\cdot\Omega^{I_{{}^\sigma\!\underline{\kappa}}}({}^\sigma\!(\itPi^D)^\vee)}.
\end{align}
On the other hand, by Lemma \ref{L:Petersson norm} and Corollary \ref{C:Petersson norm}, we have
\begin{align}\label{E:5.3.2}
\sigma \left( \frac{L(1,\itPi,{\rm Ad})}{(2\pi\sqrt{-1})^{-\sum_{w \in \Sigma_\E} \kappa_w}\cdot\pi^{3[\F:\Q]}\cdot \<\varphi_3,\varphi_4\>}\right) &= \frac{L(1,{}^\sigma\!\itPi,{\rm Ad})}{(2\pi\sqrt{-1})^{-\sum_{w \in \Sigma_\E} \kappa_w}\cdot\pi^{3[\F:\Q]}\cdot\<{}^\sigma\!\varphi_3,{}^\sigma\!\varphi_4\>}.
\end{align}
By Lemmas \ref{L:Galois equiv. 2} and \ref{L:archimedean local period}, we have
\begin{align}\label{E:5.3.3}
\sigma\left(\prod_{v\nmid \infty} \frac{I_v^D(\varphi_{1,v}\otimes\varphi_{2,v})}{\<\varphi_{3,v},\varphi_{4,v}\>_v}\right) = \prod_{v\nmid \infty} \frac{I_v^D(t_{\sigma,v}\varphi_{1,v}\otimes t_{\sigma,v}^\vee\varphi_{2,v})}{\<t_{\sigma,v}\varphi_{3,v},t_{\sigma,v}^\vee\varphi_{4,v}\>_v},
\end{align}
and
\begin{align}\label{E:5.3.4}
\begin{split}
&\frac{I_\infty({\bf X}(\underline{\kappa})\cdot\pi_\infty(\tau_{I_{\underline{\kappa}}})\varphi_{(\underline{\kappa},\underline{r})}\otimes{\bf X}(\underline{\kappa})\cdot\pi_\infty^\vee(\tau_{I_{\underline{\kappa}}})\varphi_{(\underline{\kappa},-\underline{r})})}{\<\varphi_{(\underline{\kappa},\underline{r})},\pi_\infty^\vee(\tau_{\Sigma_\E})\varphi_{(\underline{\kappa},-\underline{r})}\>_\infty} \\
&= \frac{I_\infty({\bf X}({}^\sigma\!\underline{\kappa})\cdot{}^\sigma\!\pi_\infty(\tau_{I_{{}^\sigma\!\underline{\kappa}}})\varphi_{({}^\sigma\!\underline{\kappa},\underline{r})}\otimes{\bf X}({}^\sigma\!\underline{\kappa})\cdot{}^\sigma\!\pi_\infty^\vee(\tau_{I_{{}^\sigma\!\underline{\kappa}}})\varphi_{({}^\sigma\!\underline{\kappa},-\underline{r})})}{\<\varphi_{({}^\sigma\!\underline{\kappa},\underline{r})},{}^\sigma\!\pi_\infty^\vee(\tau_{\Sigma_\E})\varphi_{({}^\sigma\!\underline{\kappa},-\underline{r})}\>_\infty}\\
&\in (2\pi \sqrt{-1})^{\sum_{v \in \Sigma_\F}(2\max_{w \mid v}\{\kappa_w\} - \sum_{w \mid v}\kappa_w) }\cdot \Q^\times.
\end{split}
\end{align}
By (\ref{E:ratio between measures}) and the result of Siegel \cite{Siegel1969}, we have
\[
C_D \in D_\F^{1/2}\cdot \zeta_\F(2)^{-1}\cdot \Q^\times,\quad \zeta_\F(2) \in D_\F^{1/2}\cdot \pi^{[\F:\Q]} \cdot \Q^\times,\quad \zeta_\E(2) \in D_\E^{1/2}\cdot \pi^{3[\F:\Q]} \cdot \Q^\times.
\]
Also note that
\[
\prod_{v \in \Sigma_\F}L(\tfrac{1}{2},\pi_v,{\rm As}) \in \pi^{-2\sum_{v \in \Sigma_\F}\max_{w \mid v}\{\kappa_w\}}\cdot\Q^\times.
\]
The algebraicity for $L(\tfrac{1}{2},\itPi,{\rm As})$ then follows from (\ref{E:5.3.1})-(\ref{E:5.3.4}).
Finally, assume $D$ is the matrix algebra, we show that $\Omega^{I_{\underline{\kappa}}}((\itPi^D)^\vee)=\Omega^{I_{\underline{\kappa}}}(\itPi^\vee)$ can be replaced by $\Omega^{I_{\underline{\kappa}}}(\itPi)$.
When $L(\tfrac{1}{2},\itPi,{\rm As})=0$, the assertion holds by assertion (1). Thus we may assume $L(\tfrac{1}{2},\itPi,{\rm As}) \neq 0$. In this case, by Ichino's formula Theorem \ref{T:Ichino formula}, there exists $\varphi \in \itPi_{\rm hol}$ such that
\[
\int_{\A_\F^\times \GL_2(\F) \backslash \GL_2(\A_\F)} {\bf X}(\underline{\kappa})\cdot \varphi^{I_{\underline{\kappa}}}(g)\,dg^{\rm Tam} \neq 0.
\]
By Lemma \ref{L:trace} and Proposition \ref{P:differential operator}, we have
\[
\sigma\left(\int_{\A_\F^\times \GL_2(\F) \backslash \GL_2(\A_\F)} \frac{{\bf X}(\underline{\kappa})\cdot \varphi^{I_{\underline{\kappa}}}(g)}{\Omega^{I_{\underline{\kappa}}}(\itPi)}\,dg^{\rm Tam}\right) = \int_{\A_\F^\times \GL_2(\F) \backslash \GL_2(\A_\F)} \frac{{\bf X}({}^\sigma\!\underline{\kappa})\cdot {}^\sigma\!\varphi^{I_{{}^\sigma\!\underline{\kappa}}}(g)}{\Omega^{I_{{}^\sigma\!\underline{\kappa}}}({}^\sigma\!\itPi)}\,dg^{\rm Tam}.
\]
On the other hand, since $\omega_{\itPi}\vert_{\A_\F^\times}$ is trivial, we have $\varphi \vert_{\GL_2(\A_\F)} = (\varphi \otimes \omega_\itPi^{-1}) \vert_{\GL_2(\A_\F)}$.
By Lemma \ref{L:trace} and Proposition \ref{P:differential operator} again, we have
\begin{align*}
&\sigma\left(\int_{\A_\F^\times \GL_2(\F) \backslash \GL_2(\A_\F)} \frac{{\bf X}(\underline{\kappa})\cdot \varphi^{I_{\underline{\kappa}}}(g)}{\Omega^{I_{\underline{\kappa}}}(\itPi)}\,dg^{\rm Tam}\right)\\
& = \omega_\itPi(\det(\tau_{I_{\underline{\kappa}}}))\cdot\sigma\left( \frac{\Omega^{I_{\underline{\kappa}}}(\itPi^\vee)}{\Omega^{I_{\underline{\kappa}}}(\itPi)}\right)\cdot \sigma\left(\int_{\A_\F^\times \GL_2(\F) \backslash \GL_2(\A_\F)} \frac{{\bf X}(\underline{\kappa})\cdot (\varphi\otimes \omega_\itPi^{-1})^{I_{\underline{\kappa}}}(g)}{\Omega^{I_{\underline{\kappa}}}(\itPi^\vee)}\,dg^{\rm Tam}\right)\\
& = \omega_\itPi(\det(\tau_{I_{\underline{\kappa}}}))\cdot\sigma\left( \frac{\Omega^{I_{\underline{\kappa}}}(\itPi^\vee)}{\Omega^{I_{\underline{\kappa}}}(\itPi)}\right)\cdot \int_{\A_\F^\times \GL_2(\F) \backslash \GL_2(\A_\F)} \frac{{\bf X}({}^\sigma\!\underline{\kappa})\cdot {}^\sigma\!(\varphi\otimes \omega_\itPi^{-1})^{I_{{}^\sigma\!\underline{\kappa}}}(g)}{\Omega^{I_{{}^\sigma\!\underline{\kappa}}}({}^\sigma\!\itPi^\vee)}\,dg^{\rm Tam}.
\end{align*}
Here ${}^\sigma\!(\varphi\otimes \omega_\itPi^{-1})$ is defined in (\ref{E:sigma iso.}) with $\itPi$ replaced by $\itPi^\vee$. By the $q$-expansion principle, we have
\[
{}^\sigma\!(\varphi\otimes \omega_\itPi^{-1}) = {}^\sigma\!\varphi \otimes {}^\sigma\!\omega_\itPi^{-1}.
\]
Indeed, by Lemma \ref{L:Galois equiv. cusp form}, we have
\begin{align*}
W_{{}^\sigma\!(\varphi\otimes \omega_\itPi^{-1}),\psi_\E}^{(\infty)}(g_f) &= \frac{\sigma(2\pi\sqrt{-1})^{-\sum_{w \in \Sigma_\E}(\kappa_w-r_w)/2}}{(2\pi\sqrt{-1})^{-\sum_{w \in \Sigma_\E}(\kappa_w-r_w)/2}}\cdot \sigma\left( W^{(\infty)}_{\varphi\otimes \omega_\itPi^{-1},\psi_\F}({\bf a}(u)^{-1}g_f)\right)\\
&= {}^\sigma\!\omega_\itPi(u)\cdot\frac{\sigma(2\pi\sqrt{-1})^{-\sum_{w \in \Sigma_\E}(\kappa_w-r_w)/2}}{(2\pi\sqrt{-1})^{-\sum_{w \in \Sigma_\E}(\kappa_w-r_w)/2}}\cdot\sigma\left( W^{(\infty)}_{\varphi,\psi_\F}({\bf a}(u)^{-1}g_f)\right)\cdot{}^\sigma\!\omega_\itPi^{-1}(\det(g_f))\\
&= {}^\sigma\!\omega_\itPi(u)\cdot\frac{\sigma(2\pi\sqrt{-1})^{\sum_{w \in \Sigma_\E}r_w}}{(2\pi\sqrt{-1})^{\sum_{w \in \Sigma_\E}r_w}}\cdot W^{(\infty)}_{{}^\sigma\!\varphi,\psi_\F}(g_f)\cdot{}^\sigma\!\omega_\itPi^{-1}(\det(g_f))\\
& = W_{{}^\sigma\!\varphi\otimes {}^\sigma\!\omega_\itPi^{-1},\psi_\E}^{(\infty)}(g_f)
\end{align*}
for all $g_f \in \GL_2(\A_{\E,f})$. Here $u \in \widehat{\Z}^\times$ is the unique element such that $\sigma(\psi_{\E}(x)) = \psi_{\E}(ux)$ for $x \in \A_{\E,f}$. Note that the last equality follows from the conditions that $\omega_{\itPi}\vert_{\A_\F^\times}$ is trivial and $\sum_{w \in \Sigma_\E}r_w=0$.
Therefore, we have
\begin{align*}
&\omega_\itPi(\det(\tau_{I_{\underline{\kappa}}}))\cdot\int_{\A_\F^\times \GL_2(\F) \backslash \GL_2(\A_\F)} {\bf X}({}^\sigma\!\underline{\kappa})\cdot {}^\sigma\!(\varphi\otimes \omega_\itPi^{-1})^{I_{{}^\sigma\!\underline{\kappa}}}(g)\,dg^{\rm Tam}\\
&=\int_{\A_\F^\times \GL_2(\F) \backslash \GL_2(\A_\F)} {\bf X}({}^\sigma\!\underline{\kappa})\cdot {}^\sigma\!\varphi^{I_{{}^\sigma\!\underline{\kappa}}}(g)\cdot \omega_\itPi^{-1}(\det(g))\,dg^{\rm Tam}\\
&=\int_{\A_\F^\times \GL_2(\F) \backslash \GL_2(\A_\F)} {\bf X}({}^\sigma\!\underline{\kappa})\cdot {}^\sigma\!\varphi^{I_{{}^\sigma\!\underline{\kappa}}}(g)\,dg^{\rm Tam}.
\end{align*}
We conclude that
\[
\sigma\left( \frac{\Omega^{I_{\underline{\kappa}}}(\itPi^\vee)}{\Omega^{I_{\underline{\kappa}}}(\itPi)}\right) = \frac{\Omega^{I_{{}^\sigma\!\underline{\kappa}}}({}^\sigma\!\itPi^\vee)}{\Omega^{I_{{}^\sigma\!\underline{\kappa}}}({}^\sigma\!\itPi)}.
\]
This completes the proof.

\subsection{Proof of Theorem \ref{C:main}}

First we recall the result of Shimura \cite{Shimura1978} on the algebraicity of special values of the twisted standard $L$-functions for motivic irreducible cuspidal automorphic representations of $\GL_2(\A_\F)$. We also refer to \cite{RT2011} for a different proof.
\begin{thm}[Shimura]\label{T:Shimura}
Let $\itPi$ be a motivic irreducible cuspidal automorphic representation of $\GL_2(\A_\F)$ of motivic weight $(\underline{\ell},r) \in \Z_{\geq 2}[\Sigma_\F] \times \Z$ and central character $\omega_\itPi$.
There exist complex numbers $p(\varepsilon,\itPi) \in \C^\times$ defined for $\varepsilon \in \{\pm 1\}^{\Sigma_\F}$ satisfying the following assertions:
\begin{itemize}
\item[(1)] We have 
\begin{align*}
\sigma\left(\frac{L^{(\infty)}(m+\tfrac{1}{2},\itPi\otimes\chi)}{(2\pi\sqrt{-1})^{[\F:\Q]m}\cdot G(\chi)\cdot p((-1)^m{\rm sgn}(\chi),\itPi)}\right) = \frac{L^{(\infty)}(m+\tfrac{1}{2},{}^\sigma\!\itPi\otimes{}^\sigma\!\chi)}{(2\pi\sqrt{-1})^{[\F:\Q]m}\cdot G({}^\sigma\!\chi)\cdot p((-1)^m{\rm sgn}({}^\sigma\!\chi),{}^\sigma\!\itPi)}
\end{align*} 
for any finite order Hecke character $\chi$ of $\A_\F^\times$, $\sigma \in {\rm Aut}(\C)$, and $m \in \Z$ such that
\[
-\frac{\min_{v \in \Sigma_\F}\{\ell_v\}}{2}-\frac{r}{2} < m < \frac{\min_{v \in \Sigma_\F}\{\ell_v\}}{2}-\frac{r}{2}.
\]
\item[(2)] We have
\begin{align*}
&\sigma\left(\frac{p(\varepsilon,\itPi)\cdot p(-\varepsilon,\itPi)}{(2\pi\sqrt{-1})^{[\F:\Q](1+r)} (\sqrt{-1})^{\sum_{v \in \Sigma_\F}\ell_v}\cdot G(\omega_\itPi)\cdot \Omega^{\Sigma_\F}(\itPi)}\right) \\
&= \frac{p({}^\sigma\!\varepsilon,{}^\sigma\!\itPi)\cdot p(-{}^\sigma\!\varepsilon,{}^\sigma\!\itPi)}{(2\pi\sqrt{-1})^{[\F:\Q](1+r)} (\sqrt{-1})^{\sum_{v \in \Sigma_\F}\ell_v}\cdot G({}^\sigma\!\omega_\itPi)\cdot \Omega^{\Sigma_\F}({}^\sigma\!\itPi)}
\end{align*}
for any $\varepsilon \in \{\pm 1\}^{\Sigma_\F}$ and $\sigma \in {\rm Aut}(\C)$.
\end{itemize}
\end{thm}

\begin{rmk}
In \cite[Theorem 4.3]{Shimura1978}, the theorem was stated for motivic $(\underline{\ell},r) \in \Z_{\geq 3}[\Sigma_\F] \times \Z$. It is straightforward to extend the results to $(\underline{\ell},r) \in \Z_{\geq 2}[\Sigma_\F] \times \Z$ by \cite[(4.16)]{Shimura1978} and the non-vanishing theorem of Friedberg--Hoffstein \cite{FriedbergHoffstein1995}. We also refer to Lemma \ref{L:Petersson norm} for the period relation between Petersson norm and $\Omega^{\Sigma_\F}(\itPi)$.
\end{rmk}

Now we begin the proof of Theorem \ref{C:main}.
Let $\itPi = \bigotimes_v\pi_v$, $\itPi' = \bigotimes_v\pi'_v$ be motivic irreducible cuspidal automorphic representations of $\GL_2(\A_\F)$ with central characters $\omega_\itPi$, $\omega_{\itPi'}$ and of weights $\underline{\ell}, \,\underline{\ell}' \in \Z_{\geq 1}[\Sigma_\F]$, respectively. By assumption (i), the automorphic representation ${\rm Sym}^2(\itPi)\times \itPi'$ of $\GL_3(\A_\F) \times \GL_2(\A_\F)$ is self-dual. Since ${\rm Sym}^2(\itPi)\times \itPi'$ is isobaric, we have the global functional equation
\[
L(s,{\rm Sym}^2(\itPi)\times \itPi') = \varepsilon(s,{\rm Sym}^2(\itPi)\times \itPi') L(1-s,{\rm Sym}^2(\itPi)\times \itPi').
\]
Recall the global root number $\varepsilon({\rm Sym}^2(\itPi)\times \itPi')$ is defined by
\[
\varepsilon({\rm Sym}^2(\itPi)\times \itPi') = \varepsilon(\tfrac{1}{2},{\rm Sym}^2(\itPi)\times \itPi') \in \{\pm 1\}.
\]
Analogous to the proof of Lemma \ref{L:Galois equiv. 1}, we can show that
\[
\varepsilon({\rm Sym}^2(\itPi)\times \itPi') = \varepsilon({\rm Sym}^2({}^\sigma\!\itPi)\times {}^\sigma\!\itPi')
\]
for all $\sigma \in {\rm Aut}(\C)$. In particular, if $\varepsilon({\rm Sym}^2(\itPi)\times \itPi')=-1$, then it follows from the global functional equation that
\[
L(\tfrac{1}{2}, {\rm Sym}^2({}^\sigma\!\itPi)\times {}^\sigma\!\itPi')=0
\]
for all $\sigma \in {\rm Aut}(\C)$. 
Therefore, we may assume that $\varepsilon({\rm Sym}^2(\itPi)\times \itPi')=1$. By assumption (iii) and the non-vanishing theorem of Friedberg--Hoffstein \cite{FriedbergHoffstein1995}, there exists a totally real quadratic extension $\K$ over $\F$ such that the base change lift $\itPi_\K = \bigotimes_v \pi_{\K,v}$ of $\itPi$ to $\GL_2(\A_\K)$ is cuspidal and 
\[
L(\tfrac{1}{2}, \itPi' \otimes \omega_\itPi\omega_{\K/\F}) \neq 0.
\]
Consider the totally real \'etale cubic algebra $\E=\K \times \F$ over $\F$ and the motivic irreducible cuspidal automorphic representation $\itPi_\K \times \itPi'$ of $\GL_2(\A_\E)$. We identify $\Sigma_\E$ with $\Sigma_\K \sqcup \Sigma_\F$ in a natural way. Note that the weight $\underline{\kappa} \in \Z_{\geq 1}[\Sigma_\E]$ of $\itPi_\K \times \itPi'$ is given as follows: for $w \in \Sigma_\E$ lying over $v \in \Sigma_\F$, we have
\[
\kappa_w = \begin{cases}
\ell_v & \mbox{ if $w \in \Sigma_\K$},\\
\ell_v' & \mbox{ if $w \in \Sigma_\F$}.
\end{cases}
\]
In particular, $\itPi_\K \times \itPi'$ is totally unbalanced and $I_{\underline{\kappa}} = \Sigma_\F$.
Let $D$ be the unique totally indefinite quaternion algebra over $\F$ so that there exists an irreducible cuspidal automorphic representation 
\[
\itPi_\K^D \times (\itPi')^D = \bigotimes_v (\pi_{\K,v}^D \times  (\pi_v')^D)
\]
of $D^\times(\A_\E)$ associated to $\itPi_\K \times \itPi'$ by the Jacquet--Langlands correspondence such that
\[
{\rm Hom}_{D^\times(\F_v)}(\pi_{\K,v}^D \times  (\pi_v')^D,\C) \neq 0
\]
for all places $v$ of $\F$. By Lemma \ref{L:decomposition}, we have the period relation
\begin{align}\label{E:4.8}
\sigma \left( \frac{\Omega^{\Sigma_\F}(\itPi_\K^D \times (\itPi')^D)}{\Omega^{\emptyset}(\itPi_\K^D)\cdot\Omega^{\Sigma_\F}((\itPi')^D)}\right) = \frac{\Omega^{\Sigma_\F}({}^\sigma\!\itPi_\K^D \times {}^\sigma\!(\itPi')^D)}{\Omega^{\emptyset}({}^\sigma\!\itPi_\K^D)\cdot\Omega^{\Sigma_\F}({}^\sigma\!(\itPi')^D)}
\end{align}
for all $\sigma \in {\rm Aut}(\C)$.
Note that by definition we may take $\Omega^{\emptyset}(\itPi_\K^D)=1$ and by Corollary \ref{C:Petersson norm} we have
\begin{align}\label{E:4.9}
\sigma\left(  \frac{\Omega^{\Sigma_\F}(\itPi')^2}{\Omega^{\Sigma_\F}(\itPi'^D)\cdot\Omega^{\Sigma_\F}(\itPi'^\vee)^D)}\right) = \frac{\Omega^{\Sigma_\F}({}^\sigma\!\itPi')^2}{\Omega^{\Sigma_\F}({}^\sigma\!\itPi'^D)\cdot\Omega^{\Sigma_\F}({}^\sigma\!(\itPi'^\vee)^D)}
\end{align}
for all $\sigma \in {\rm Aut}(\C)$.
Also we have the well-known fact for the rationality of the quadratic Gauss sum that
\begin{align}\label{E:4.10}
D_\K^{1/2}G(\omega_{\K/\F}) \in \Q^\times.
\end{align}
Therefore, by (\ref{E:4.8})-(\ref{E:4.10}) and Theorem \ref{T: main thm}, we have
\begin{align*}
\sigma\left(\frac{L^{(\infty)}(\tfrac{1}{2},\itPi_\K \times \itPi', {\rm As}) }{D_\F^{1/2}(2\pi\sqrt{-1})^{2[\F:\Q]}\cdot G(\omega_{\K/\F})\cdot\Omega^{\Sigma_\F}(\itPi')^2}\right)= \frac{L^{(\infty)}(\tfrac{1}{2},{}^\sigma\!\itPi_\K \times {}^\sigma\!\itPi', {\rm As}) }{D_\F^{1/2}(2\pi\sqrt{-1})^{2[\F:\Q]}\cdot G(\omega_{\K/\F})\cdot\Omega^{\Sigma_\F}({}^\sigma\!\itPi')^2}
\end{align*}
for all $\sigma \in {\rm Aut}(\C)$.
Finally, we have the factorization of $L$-functions:
\[
L(s,\itPi_\K \times \itPi', {\rm As}) = L(s,{\rm Sym}^2(\itPi)\times \itPi')L(s,\itPi' \otimes \omega_\itPi\omega_{\K/\F}).
\]
Theorem \ref{C:main} then follows immediately from Theorem \ref{T:Shimura} of Shimura for the central critical value
\[
L^{(\infty)}(\tfrac{1}{2},\itPi' \otimes \omega_\itPi\omega_{\K/\F}) = L^{(\infty)}(r+\tfrac{1}{2},\itPi' \otimes |\mbox{ }|_{\A_\F}^{-r}\omega_{\itPi}\omega_{\K/\F})
\]
and the condition that $L(\tfrac{1}{2}, \itPi' \otimes \omega_\itPi\omega_{\K/\F}) \neq 0$.
This completes the proof.

\subsection{Proof of Theorem \ref{T: main thm 2}}\label{SS:5.5}

Let $\itPi = \bigotimes_v \pi_v$ be a motivic irreducible cuspidal automorphic representation of $\GL_2(\A_\F)$ and $\chi$ an algebraic Hecke character of $\A_\F^\times$. Assume that $\itPi$ has motivic weight $(\underline{\ell},r) \in \Z_{\geq 4}[\Sigma_\F] \times \Z$ and $|\chi| = |\mbox{ }|_{\A_\F}^{r_0}$ for some $r_0 \in \Z$. Let $I$ be a subset of $\Sigma_\F$. We fix a motivic irreducible cuspidal automorphic representation $\itPi' = \bigotimes_v \pi_v'$ of $\GL_2(\A_\F)$ with motivic weight $(\underline{\ell}',r') \in \Z_{\geq 2}[\Sigma_\F] \times \Z$ satisfying the following conditions:
\begin{itemize}
\item $\min_{v \in \Sigma_\F}\{|\ell_v-\ell_v'|\} \geq 2$; 
\item $I=\{ v \in \Sigma_\F\,\vert\,\ell_v > \ell_v'\}$.
\end{itemize}
The existence of $\itPi'$ is guaranteed by the assumption that $\underline{\ell} \in \Z_{\geq 4}[\Sigma_\F]$ and the result of Weinstein \cite{Weinstein2009}.
Consider the Rankin--Selberg $L$-function 
\[
L(s,\itPi \times \itPi' \times \chi ) = L(s,(\itPi\otimes\chi) \times \itPi' \times {\bf 1})
\]
for the triplets $(\itPi,\itPi',\chi)$ and $(\itPi\otimes\chi,\itPi',{\bf 1})$.
Note that at the rightmost critical point 
\[
m = -r_0-\frac{r+r'}{2} + \frac{\min_{v \in \Sigma_\F} \{|\ell_v-\ell_v'|\}}{2},
\]
we have $L(m,\itPi \times \itPi' \times \chi ) \neq 0$ by the condition that $\min_{v \in \Sigma_\F}\{|\ell_v-\ell_v'|\} \geq 2$. Applying Theorem \ref{T:RS} to the two triplets, we deduce that
\[
\sigma\left(\frac{\Omega^I(\itPi)}{\Omega^I(\itPi\otimes\chi)}\right) = \frac{\Omega^{{}^\sigma\!I}({}^\sigma\!\itPi)}{\Omega^{{}^\sigma\!I}({}^\sigma\!\itPi \otimes {}^\sigma\!\chi)}
\]
for all $\sigma \in {\rm Aut}(\C)$.
This completes the proof.

\end{document}